\documentclass[11pt]{article}

\usepackage{amsmath,amsthm,amsfonts,amssymb,latexsym,amscd}
\usepackage{color}
\usepackage{hyperref, graphicx}
\input{prepictex.tex}
\input{pictex.tex}
\input{postpictex.tex}

\begin{document}
\newtheorem*{theo}{Theorem} 
\newtheorem*{pro} {Proposition}
\newtheorem*{cor} {Corollary}
\newtheorem*{lem} {Lemma}
\newtheorem{theorem}{Theorem}[section]
\newtheorem{corollary}[theorem]{Corollary}
\newtheorem{lemma}[theorem]{Lemma}
\newtheorem{proposition}[theorem]{Proposition}
\newtheorem{conjecture}[theorem]{Conjecture}
\newtheorem{problem}[theorem]{Problem}

\theoremstyle{definition}
 \newtheorem{definition}[theorem]{Definition}
  \newtheorem{example}[theorem]{Example}
   \newtheorem{remark}[theorem]{Remark}
   
\newcommand{\NN}{{\mathbb{N}}}
\newcommand{\RR}{{\mathbb{R}}}
\newcommand{\CC}{{\mathbb{C}}}
\newcommand{\PP}{{\mathbb{P}}}
\newcommand{\ZZ}{{\mathbb{Z}}}
\newcommand{\bZ}{{\mathbb{Z}}}
\newcommand{\HH}{\mathfrak H}
\newcommand{\KK}{\mathfrak K}
\newcommand{\LL}{\mathfrak L}
\newcommand{\as}{\ast_{\sigma}}
\newcommand{\supp}{\mbox{supp}}
\newcommand{\tn}{\vert\hspace{-.3mm}\vert\hspace{-.3mm}\vert}
\def\A{{\cal A}}
\def\B{{\cal B}}
\def\D{{\cal D}}
\def\E{{\cal E}}
\def\F{{\cal F}}
\def\H{{\cal H}}
\def\J{{\cal J}}
\def\K{{\cal K}}
\def\L{{\cal L}}
\def\R{{\cal R}}
\def\N{{\cal N}}
\def\M{{\cal M}}
\def\MI{{\cal MI}}
\def\gM{{\frak M}}
\def\O{{\cal O}}
\def\P{{\cal P}}
\def\S{{\cal S}}
\def\T{{\cal T}}
\def\U{{\cal U}}
\def\V{{\mathcal V}}
\def\qed{\hfill$\square$}

\title{
Hypercube subgroups of (outer) reduced Weyl groups of the Cuntz algebras}

\author{Francesco Brenti\footnote{Partially supported by MIUR Excellence Department Projects 
awarded to the Department of Mathematics of the University of Rome ``Tor Vergata'', CUP E83C18000100006 and E83C23000330006.}
, Roberto Conti\footnote{Partially supported by Sapienza University of Rome (Progetti di Ateneo 2023) and INDAM-GNAMPA.}
, Gleb Nenashev
\\}
\date{\today}
\maketitle
\markboth{F. Brenti, R. Conti, G. Nenashev}{}
\renewcommand{\sectionmark}[1]{}

\begin{abstract} 
We develop some tools, of an algebraic and combinatorial nature, which enable us to obtain a detailed
description of certain quadratic subgroups of the (outer) reduced Weyl group of the Cuntz algebra $\O_n$. 
In particular, for $n=4$ our findings give a self-contained theoretical interpretation of the groups tabulated 
in \cite{AJS18}, which were obtained with the help of a computer. For each of these groups we provide a set of generators. 
A prominent role in our analysis is played by a certain family of subgroups of the symmetric group of a discrete square
which we call bicompatible.
\end{abstract}



\newpage

\tableofcontents

\section{Introduction}

This work is a contribution to the study of Cuntz algebras automorphisms. The Cuntz algebras $\O_n$ are a well-known family of $C^*$-algebras, that have been central to many investigations in operator algebras (see \cite{ACR21} for a broad overview) since their first appearence in the celebrated work \cite{Cu77}. Although Cuntz himself suggested it as an interesting line of investigation, 
inspired by a comparison with the theory of semisimple Lie groups, the study of the automorphisms of the Cuntz algebras did not immediately fluorish as expected (but see e.g. \cite{Arc,Evans,MaTo}), but started receiving a great deal of attention only during the last twenty years, beginning with the work of Conti and Szymanski \cite{CoSz11}, see also \cite{CKS10,CHS12b,CHS12c}. 
In general, this is a rich topic, quite far from being trivial and
with many different facets. In this paper, we focus mostly on the combinatorial side of this wide area,
following our previous works on the subject \cite{BC,BCN,BCN2}.
The reason why combinatorics enters the game, as foreseen and suggested already in \cite{Cu80}, is that the natural definition of a reduced Weyl group for ${\rm Aut}(\O_n)$ leads to the consideration of a special family of automorphisms that are described, through the Cuntz-Takesaki correspondence, by suitable permutations of the hypercubes $\{1,\ldots,n\}^t$, $t \geq 1$.

In order to describe the problem we are interested in, we need to introduce some notation (see Sect. \ref{prelim} for a more systematic account).

Given an integer $n \geq 2$, the Cuntz algebra $\O_n$ is generated, as a $C^*$-algebra, by a family of $n$ isometries $S_i, \ i = 1,\ldots,n$ such that $\sum_{i=1}^n S_i S_i^* = 1$.
Moreover, any (unital, $*$-preserving) endomorphism of $\O_n$ is of the form $\lambda_u$, for some unitary $u$,
where 
$\lambda_u(S_i) = uS_i, \ i=1,\ldots,n \ . $

For any set of unitaries $E \subset \U(\O_n)$, we define
$$\lambda(E)^{-1}:=\{\lambda_u \in {\rm Aut}(\O_n) \ | \ u \in E\} \ , $$
i.e. the set of automorphisms of $\O_n$ induced by unitaries in $E$.

Defining $\P_n := \bigcup_{t \in {\mathbb N}} \P_n^t$, with 
$$\P_n^t = \Big\{u \in \U(\O_n) \ | \ u = 
\sum\limits_{\substack{(\alpha,\beta) \\ \alpha=(\alpha_1,\ldots,\alpha_t) \\ \beta=(\beta_1,\ldots,\beta_t)}}
S_{\alpha_1} \cdots S_{\alpha_t}  S_{\beta_t}^* \cdots S_{\beta_1}^* \Big\}, $$ 
being naturally identified with the set of permutation matrices in $M_{n^t}({\mathbb C})$ (a set of cardinality $n^t !$), it turns out that the reduced Weyl group $W_n$ of the Cuntz algebra $\O_n$ can be described as
$\lambda(\P_n)^{-1} $.

Now, fix the ``level'' $t \in {\mathbb N}$.  
As $\lambda(\P_n^1)^{-1} := \{\lambda_u \ | \ u \in \P_n^1 \} \simeq \P_n^1 \simeq S_n$ 
(these automorphisms of $\O_n$ are usually referred to as permutative Bogolubov, or quasi-free automorphisms),
the first nontrivial case of course occurs for $t=2$.

For $u \in \P_n^t$ it is well understood at least theoretically
when $\lambda_u \in {\rm Aut}(\O_n)$ (see \cite[Theorem 3.2]{CoSz11} and also Theorem 2.1 therein). However determining $\lambda(\P_n^t)^{-1} $ becomes seriously demanding from the computational point of view, unless $t$ and $n$ are sufficiently small (say $n+t \leq 7$, but not all such cases are fully computed). 
See e.g. \cite{AJS18}, Table 2 for a recent overview of the known cardinalities of such sets. These figures were obtained through massive computer calculations that probably 
even now can not be extended.

Notice that any $\lambda(\P_n^t)^{-1}$ contains automorphisms with both finite and infinite order. In this paper we will only consider those of finite order,
as we will be looking for groups sitting in $\lambda(\P_n^t)^{-1}$
and, more precisely, for the maximal groups in the image $\pi(\lambda(\P_n^t)^{-1})$ of $\lambda(\P_n^t)^{-1}$ in the outer automorphism group ${\rm Out}(\O_n)$, 
the quotient of ${\rm Aut}(\O_n)$ modulo inner automorphisms.
We can now formulate the problem that motivated this work.

\begin{problem}
Given integers $n \geq 2$ and $t \geq 2$, what are the maximal groups contained in  $\{\pi(\lambda_u) \in {\rm Out}(\O_n) \ | \ u \in \P_n^t\}?$
\end{problem}

\medskip
A discussion of such groups for a few values of $n$ and $t$
appears in \cite[Section 8.2]{AJS18}, see the tables 3-5 therein obtained with the help of a computer. 
In particular, for $n=4$ and $t=2$ one finds out that there are 46 such distinct maximal groups in 
$\{\pi(\lambda_u) \in {\rm Out}(\O_4) \ | \ u \in \P_4^2\}$
(see Table 1,
which reproduces Table 5 in \cite{AJS18})
which is a set of cardinality 240480 (cf. \cite[p. 5866]{CoSz11}).

In this work we provide a detailed explicit description of all these 46 maximal groups.
To the best of our knowledge, this is the first time that finite groups of automorphisms of the Cuntz algebras that are not of quasi-free type are examined in such detail.
In order to do this we generalize
and refine some of the results in \cite{BC,BCN,BCN2}. In
particular, as a consequence of these generalizations, we
obtain a proof of Conjecture 12.2 of \cite{BC} which is
quite different from the one recently given in \cite{Pan}.
Our investigation shows that some of these groups are special cases
of groups that exist for any $n \geq 4$, while others do not
seem to be easily generalizable to $n>4$.
Many of the arguments that we present are more general than those strictly needed to get the result, and might be useful elsewhere. Indeed, we believe that we have laid down the basis for handling 
several other cases as well.

The paper is organized as follows. In the next section we recall some notation and results that we use in the rest of the work. Among various things, we introduce the notion of {\it bicompatible subgroup} which plays a crucial role in the sequel. In Section \ref{horiz_vert} we present a number of results pertaining
to the compatibility of elements of $S([n]^2)$  that permute some rows (resp. columns) all in the same way and leave everything else fixed, with other permutations. In particular, we obtain in Theorem \ref{solveconj} a new proof of the conjecture mentioned above, and some related enumerative results. In Section \ref{Subgroups} we construct some notable bicompatible subgroups
of $S([n]^2)$, 
see Theorems \ref{Bicomp1} (that works for all $n$) and \ref{NBCprec} (only for even $n$). The $n=4$ case of these subgroups will be useful later. 
In Section \ref{Quadraticsub}, which is the original motivation of this work, all the previous results are put together to describe in detail all the 46 maximal groups of outer automorphisms of $\O_4$ arising from $S([4]^2)$. 
In Section \ref{bicomp_subgs} we derive a series of results 
about bicompatible subgrups of $S([n]^2)$ and then we apply 
them to the classification of all the bicompatible subgroups 
of $S([4]^2)$ that are isomorphic to $S_4$. These results 
are not needed anywhere else in the paper but we present them because 
we feel that they are of independent interest, and because they show
how some of the subgroups presented and studied in Section 
\ref{Quadraticsub} were found. 
The last section contains some directions for further research.
In particular, we include a couple of figures obtained with the help of a computer describing the stable permutations of rank one in $S([4]^2)$ of cycle type $(2,2)$ and $(2,2,2)$, for which there is no complete theory yet.

\section{Preliminaries} \label{prelim}

Let $n$ be a positive integer. We denote by $[n]$ the set $\{1,\ldots,n\}$ and, more generally, $[n]^k$ for the cartesian power $\{1,\ldots,n\}^{\times k}$. Given a set $A$ we let
$S(A):= \{ \sigma: A \rightarrow A : \sigma \mbox{ bijection }\}$ 
be the symmetric group on $A$ and $S_n:=S([n])$. 
We denote by $1$ the identity of $S_n$. We write permutations $\sigma \in S(A)$ in {\em disjoint cycle form},
omitting the $1$-cycles (i.e., the fixed points). 
So, for example, $\sigma = ( (2,2),(2,4),(4,2)) $ denotes
the permutation $\sigma \in S([4]^2)$ such that $\sigma(2,2)=(2,4)$,
$\sigma(2,4)=(4,2)$, $\sigma(4,2)=(2,2)$, and $\sigma(i,j)=(i,j)$
for every other element $(i,j)$ of $[4]^2$.
For $u \in S([n]^t)$ and $v \in S([n]^r)$ we let the {\em tensor product} of $u$ and $v$
be the permutation $u \otimes v \in S([n]^{t+r})$ defined by
\[
(u \otimes v)(\alpha,\beta) := (u(\alpha),v(\beta))
\]
for all $\alpha \in [n]^t$ and $\beta \in [n]^s$.

\medskip
For $n \geq 2$, the Cuntz algebra $\O_n$ is the universal $C^*$-algebra generated by $n$ isometries $S_1, \ldots, S_n$ whose range projections sum up to 1, i.e.
$\sum_{i=1}^n S_i S_i^* = 1$ ($1$ denotes the unit of the algebra). It readily follows that $S^*_i S_j = \delta_{ij} 1$ for all $i,j \in [n]$.
The Cuntz-Takesaki correspondence is a well-known bijection between the semigroup ${\rm End}(\O_n)$ of (unital, $*$-)endomorphisms of $\O_n$ and the unitary group $\U(\O_n) = \{u \in \O_n \ | \ u^* u = 1 = u u^*\}$ that associates to the unitary $u$ the endomorphism $\lambda_u$ uniquely determined
$\lambda_u(S_i) = u S_i$ for all $i \in [n]$. 
A simple computation then shows that 
$$\lambda_u(S_{i_1} \cdots S_{i_h} S^*_{j_k} \cdots S^*_{j_1}) 
= u_h S_{i_1} \cdots S_{i_h} S^*_{j_k} \cdots S^*_{j_1} (u_k)^*
\ , $$
where, for every $m \geq 1$, $u_m := u \varphi(u) \cdots \varphi^{m-1}(u)$ and $\varphi(x)=\sum_{i=1}^n S_i x S_i^*$ is the so-called canonical endomorphism of $\O_n$, which is associated to the unitary $F := \sum_{i,j=1}^n S_i S_j S_i^* S_j^*$, 
and satisfies $S_i x = \varphi(x)S_i$ for all $x \in \O_n$ and $i \in [n]$.
For convenience, we will sometimes write $S_\alpha$ for $S_{i_1} \cdots S_{i_h} $ where $\alpha$ is the multi-index $(i_1,\ldots,i_h)$ and say that $\ell(\alpha) = h$ is the length of $\alpha$. So, for example, $S_{311}=S_3 S_1 S_1$ and hence
$S_{311}^*:=(S_{311})^*=S_1^* S_1^* S_3^*$.
It is not hard to see that 
for every $u,v \in \U(\O_n)$ one has the identities
\begin{equation}
\label{product}
\lambda_u \circ \lambda_v = \lambda_{\lambda_u(v)u}
\end{equation}
and
\[ 
 {\rm ad}(u) \circ \lambda_v = \lambda_{u v \varphi(u^*)} \, 
\]
that we will repeatedly use throughout the paper without further mention. Here, ${\rm ad}(u)$ denotes the inner automorphism defined by ${\rm ad}(u)(x) = u x u^*$, $x \in \O_n$.

We denote by ${\rm Aut}(\O_n)$ the group of ($*$-)automorphisms of $\O_n$, by ${\rm Out}(\O_n)$ its quotient by the normal subgroup of inner automorphisms and by $\pi: {\rm Aut}(\O_n)$ $\to {\rm Out}(\O_n)$ the canonical projection.

For any integer $t \geq 1$, we set 
$$
\P_n^t = \Big\{u \in \U(\O_n) \ | \ u = \sum_{(\alpha,\beta)} S_\alpha S_\beta^*, \ \ell(\alpha) = t = \ell(\beta)\Big\}
$$ 
which can be identified with $S([n]^t)$, the symmetric group of the hypercube $[n]^t$, and thus also with the set of permutation matrices in $M_{n^t}({\mathbb C}) \simeq M_n \otimes \cdots \otimes M_n$ ($t$ times).
So, for example, if $\sigma = ( (2,2),(2,4),(4,2)) \in S([4]^2)$
then the corresponding unitary of $\P_4^2$ is
\[
S_{24} \, S_{22}^*+S_{42} \, S_{24}^*+S_{22} \, S_{42}^*+
\sum_{(i,j) \in [4]^2 \setminus \{ (2,2),(2,4),(4,2) \}}
S_{ij} \, S_{ij}^*\, ,
\]
(which can also be written as $S_{24} \, S_{22}^*+S_{42} \, S_{24}^*+S_{22} \, S_{42}^*+S_1 S_1^* + S_3 S_3^* +S_{21} \, S_{21}^*+$ 
$S_{23} \, S_{23}^*+S_{41} \, S_{41}^*+S_{43} \, S_{43}^*+S_{44} \, S_{44}^*$).
In the rest of this work we will switch freely between the
notations $\P_n^t$ and $S([n]^t)$.
In particular, if $u \in S([n]^t)$ then we will write $\lambda_u$
to mean the endomorphism associated to the unitary corresponding 
to $u$.
Note that $\P_n^t \subset \P_n^{t+1}$ and
$\varphi(\P_n^t) \subset \P_n^{t+1}$, for all $t$.
Furthermore, it is not hard to check that if $u \in \P^t$ corresponds to a permutation $\sigma \in S([n]^t)$ then $u$, seen 
as an element of $\P^{t+1}$, corresponds to $\sigma \otimes 1 \in S([n]^{t+1})$, and $\varphi(u) \in \P^{t+1}$ corresponds to
$1 \otimes \sigma \in S([n]^{t+1})$. 
For example, $(1,2) \in S([2]^1)$ corresponds
to the unitary $u=S_1 S_2^*+S_2 S_1^*=S_1 1 S_2^*+S_2 1 S_1^*=
S_1 (S_1 S_1^*+S_2 S_2^*) S_2^*+S_2 (S_1 S_1^*+S_2 S_2^*) S_1^*$
$=S_{11} \, S_{21}^*+S_{12} \, S_{22}^*+S_{21} \, S_{11}^*+S_{22} \, S_{12}^*$, which in turn maps to the permutation
$(1,2) \otimes 1 \in S([2]^2)$. Similarly, 
$\varphi(u)=S_1 u S_1^* + S_2 u S_2^*=
S_1 (S_1 S_2^*+S_2 S_1^*) S_1^*+S_2 (S_1 S_2^*+S_2 S_1^*) S_2^*=
S_{11} \, S_{12}^*$ 
$+S_{12} \, S_{11}^*+S_{21} \, S_{22}^*+S_{22} \, S_{21}^*$
maps to $1 \otimes (1,2) \in S([2]^2)$.
In particular, $\lambda_u = \lambda_{u \otimes 1}$ for all
$u \in S([n]^t)$, while $\lambda_{1 \otimes u}= \lambda_{\varphi(u)}= {\rm ad}(u^*) \circ \lambda_u$.
We also set $\P_n = \bigcup_{t \in {\mathbb N}} \P_n^t$.
Note that if $v \in \P_n^2$ then 
$$\lambda_u \circ \lambda_v = \lambda_{u \varphi(u) v \varphi(u^*)}$$ 
for any $u \in \U(\O_n)$.

\medskip
The reduced Weyl group $W_n$ (\cite{CoSz11}) of the Cuntz algebra $\O_n$ is isomorphic to
\[
\lambda^{-1}(\P_n) := \{\lambda_u \in {\rm Aut}(\O_n) \ | \ u \in \P_n\} . 
\]
Our main object of study in this work is the image of $W_n$ in
${\rm Out}(\O_n)$, $\pi(\lambda^{-1}(\P_n))$.
\smallskip

Recall for convenience the following fact.
\begin{proposition}
Let $u \in \P_n^t$ be such that $\lambda_u \in {\rm Aut}(\O_n)$, and let $v \in \P_n^t$ and $w \in \U(\O_n)$ be such that ${\rm Ad}(w) \lambda_u = \lambda_v$ (i.e., 
$\pi(\lambda_u) = \pi(\lambda_v)$). Then $w \in \P_n^{t-1}$.
\end{proposition}
\begin{proof}
Deduce that ${\rm Ad}(w) = \lambda_v (\lambda_u)^{-1}$, where the automorphism in the r.h.s. is associated to an element in $\P_n$. 
Thus , by \cite[Lemma 2.3]{CoSz11}, $w$ belongs to $\P_n$. Moreover 
$w u \varphi(w^*) = v$ and hence the conclusion.
\end{proof}
\noindent 
In particular, for 
$u,v \in \P_n^2$, $\pi(\lambda_u) = \pi(\lambda_v)$ if and only if $v = zu \varphi(z)^*$, for some $z \in \P_n^1$.

\bigskip 
Given a permutation $u \in S([n]^t)$, 
define a sequence of permutations $\psi_k(u) \in S([n]^{t+ k})$, $k \geq 0$ by setting $\psi_0(u) := u^{-1}$ and, for $k \in {\mathbb N}$,
\begin{equation}
	\label{def_psi_k}
	\psi_k(u) = \prod_{i=0}^k 
	({\underbrace{1 \otimes \ldots \otimes 1}_{k-i}} \otimes u^{-1} \otimes {\underbrace{1 \otimes \ldots \otimes 1}_i})  
	\prod_{i=1}^k 
	({\underbrace{1 \otimes 1 \ldots \otimes 1}_i} \otimes u \otimes {\underbrace{1 \otimes \ldots \otimes 1}_{k-i}}) 
	\ . 
\end{equation}
Then $u$ as above is said to be {\it stable} if there exists some integer $k \geq 1$ such that
\begin{equation}\label{stable}
	\psi_{k+h}(u) = \psi_{k-1}(u) \otimes {\underbrace{1 \otimes \cdots \otimes 1}_{h+1}}, \quad h \geq 0 \ , 
\end{equation}
and then ${\rm rk}(u)$, the {\it rank} of $u$, is the least such value of $k$.
So, for example, if $t=1$ then all permutations $u \in S([n])$ are 
stable of rank $1$. As is costumary, we call the corresponding
automorphisms, {\em permutative Bogolubov automorphisms}.
If $n=4$ and $t=2$ there are exactly 36 transpositions that are 
stable of rank one \cite[Corollary 8.3]{BC}, of which 12 are horizontal and 12 are vertical 
(see Figure 2, and also 
Corollary \ref{horizontal_in_same_columns}). In the figure,
differently from the rest of this work, where
we represent elements of
$S([n]^2)$ graphically by drawing the cycles of the permutation 
as directed (except for the cycles of length two) 
cycles of the square grid $ [n]^2$,
we draw families of 12 transpositions in a single grid for brevity. 

The following result is proved in \cite[Theorem 3.2]{CoSz11} (with slightly different notation).
\begin{theorem}
Let $n,t \in \NN$, $n \geq 2$, and $u \in S([n]^t)$. Then
$u$ is stable if and only if $\lambda_u$ is an automorphism
of $\O_n$.
\end{theorem}
\noindent 
Hence, the reduced Weyl group of $\O_n$ can be equivalently
described as $\{ \lambda_u : u \in S([n]^t), u \mbox{ stable },
t \geq 1 \}$.

\medskip 
For $u \in S([n]^{t})$ we write 
$$
(u_1(x_1, \ldots , x_t), \ldots , u_t(x_1, \ldots , x_t)) := u(x_1, \ldots , x_t)
$$
for all $(x_1, \ldots , x_t) \in [n]^t$.
If $u \in S([n]^2)$ 
we let ${}^{a}u \in S([n]^2)$ be the  {\em antitransposed} of $u$ defined by
\begin{equation}
	\label{defatr}
	{}^{a}u(x,y) := (n+1-u_2(n+1-y,n+1-x), n+1-u_1(n+1-y,n+1-x)). 
\end{equation}
Given $u,v \in S([n]^2)$, following \cite{BC}, we say that $u$ is {\em compatible} with $v$ (or that 
$u$ is compatible with $v$ in this order, for emphasis) if
\begin{equation}
	(v \otimes 1) (1 \otimes u)=(1 \otimes u)(v \otimes 1)
\end{equation}
in $S([n]^3)$.
We say that $u$ and $v$ are {\em bicompatible} if $u$ is compatible with $v$ and $v$ is compatible with $u$. 
The following result is a special case of \cite[Prop. 5.13]{BC}.
\begin{proposition}
	\label{at-bicompatible}
Let $u,v \in S([n]^2)$. Then $u$ and $v$ are bicompatible if and only if ${}^{a}u$ and ${}^{a}v$ are bicompatible.
\end{proposition}

Given a subgroup $S \leq S([n]^2)$ we say that $S$ is
a {\em bicompatible subgroup} if $u$ and $v$ are
bicompatible for all $u,v \in S$. In particular, all the elements
of a bicompatible subgroup are stable of rank 1.

The next result is a slight refinement of \cite[Corollary 5.5]{BC}.
\begin{proposition}
	\label{isogroups}
	Let $u_1, u_2, \ldots, u_m \in S([n]^2)$ be stable permutations such that
	$u_i$ is bicompatible with $u_j$ for 
	all $1 \leq i,j \leq m$. Then the subgroup of ${\rm Aut}(\O_n)$ generated by $\{ \lambda_{u_1}, \ldots , \lambda_{u_m} \}$ is contained in 
	$\lambda^{-1}(P_n^2)$ and is 
	isomorphic to the subgroup of $S([n]^2)$ generated by $\{ u_1, \ldots , u_m \}$.
	Moreover, all the elements of this last subgroup are stable of rank 1.
\end{proposition}
\begin{proof}
	This follows immediately from \cite[Corollary 5.5]{BC} and the observation that
	if $i_j \in [m]$ for all $j=1,..,k$ then 
	$\lambda_{u_{i_1}}  \cdots \lambda_{u_{i_k}} = 
	\lambda_{u_{i_1} \cdots u_{i_k}}$ and
	$(u_{i_1} \cdots u_{i_k} \otimes 1) (1 \otimes u_{i_1} \cdots u_{i_k}) =
	(1 \otimes u_{i_1} \cdots u_{i_k}) (u_{i_1} \cdots u_{i_k} \otimes 1)$, both of
	which can be easily checked by induction.
\end{proof}

\section{Horizontal and vertical cycles}
\label{horiz_vert}

In this section we establish various results concerning the
compatibility of permutations of $S([n]^2)$ that permute 
certain rows (resp., columns) of $[n]^2$ all in the same way, and leave every other row (resp., column) fixed, with general 
permutations of $S([n]^2)$. These results will be used often in the rest of this work, and generalize various 
others that have already appeared in \cite{BC,BCN,BCN2}.
In particular, as a consequence of them, we obtain a different
proof of \cite[Conj. 12.2]{BC} from the one given in \cite{Pan}. 

Our first result introduces the permutations that we are 
interested in, and characterizes the permutations of $[n]^2$
with which these are compatible.
\begin{proposition}
	\label{horizontal_same_columns}
	Let $1 \leq a_1 < \cdots < a_k \leq n$,
	$v \in S([n]^2)$, $\sigma \in S_n$, and $u \in S([n]^2)$ be such that
	\[
	u(x,y):= 
	\begin{cases}
		(x,\sigma(y)) 
		\mbox{ if } x \in \{ a_1, \ldots , a_k\}, \\
		(x,y),
		\mbox{ otherwise},
	\end{cases}
	\]	
for all $(x,y) \in [n]^2$. 
	Then $u$ is compatible with $v$ if and only if
	$v$ leaves $[n] \times \{ a_1, \ldots , a_k \}$ invariant.
\end{proposition}
\begin{proof}
It is enough to show this in the case that $\sigma$ is a transposition. So assume that 
$u:= \prod_{r=1}^{k} \left( (a_r,i) , (a_r,j) \right)$
for some $1 \leq i < j \leq n$.
	We then have that
	\[
	1 \otimes u = \prod_{x=1}^{n} \prod_{r=1}^{k} 
	\left( (x,a_r,i) , (x,a_r,j) \right)
	= \prod_{s \in [n] \times \{ a_1, \ldots , a_k \}}
	\left( (s,i) , (s,j) \right)
	\]
	while
	\begin{align*}
		(v \otimes 1)(1 \otimes u)(v^{-1} \otimes 1) = \prod_{x=1}^{n} \prod_{r=1}^{k} 
		\left( (v(x,a_r),i) , (v(x,a_r),j) \right) \\
		= \prod_{s \in [n] \times \{ a_1, \ldots , a_k \}}
		\left( (v(s),i) , (v(s),j) \right).
	\end{align*}
and the result follows.
\end{proof}

Note that this last proposition generalizes the special case 
$a=i$ of Proposition 5.6 of \cite{BC}. While we believe that 
a result that characterizes the permutations that are compatible
with those studied in Proposition \ref{horizontal_same_columns}
is possible, the following one suffices for our purposes.

\begin{proposition}
	\label{horizontal_same_columns_2_cycles}
	Let $1 \leq a_1 < \cdots < a_k \leq n$, $\{ e_1, \ldots , e_s, f_1, \ldots , f_s \} \subseteq [n]$,
	$v \in S([n]^2)$, and 
	\[
	u:= \prod_{r=1}^{k} \left( (a_r,e_1) , \ldots , (a_r,e_s) \right) \left( (a_r,f_1) , \ldots , (a_r,f_s) \right).
	\]	
	Then $v$ is compatible with $u$ if and only if there exist $\sigma \in S(\{ e,f \} \times [n])$ and $t: \{ e,f \} \times [n] \to [s]$ such that
	$$
	v(\alpha_i,x) = 
	((\sigma_1(\alpha,x))_{i+t(\alpha, x)}, \sigma_2(\alpha,x))
	$$
	for all $i \in [s]$ and $(\alpha,x) \in \{ e,f \} \times [n]$ (where indices are modulo $s$, and $(\sigma_1(\alpha,x),
	\sigma_2(\alpha,x)):=\sigma(\alpha,x)$).
\end{proposition}
\begin{proof}
	We have that
	$$
	(u \otimes 1) = \prod_{r=1}^k \prod_{\alpha \in \{ e,f \}} \prod_{x=1}^n 
	((a_r,\alpha_1,x), \ldots , (a_r,\alpha_s,x))  $$
	and hence
	$$
	(1 \otimes v)(u \otimes 1)(1 \otimes v)^{-1} = 
	\prod_{r=1}^k \prod_{\alpha \in \{ e,f \}} \prod_{x=1}^n  
	((a_r,v(\alpha_1,x)), \ldots ,(a_r,v(\alpha_s,x))).
	$$
	So $v$ is compatible with $u$ if and only if
	$$ 
	\prod_{(\alpha,x) \in \{ e,f \} \times [n]} ((a_r,\alpha_1,x),\ldots , (a_r,\alpha_s,x)) = 
	\prod_{(\alpha,x) \in \{ e,f \} \times [n]} 
	((a_r,v(\alpha_1,x)), \ldots , (a_r,v(\alpha_s,x))) 
	$$
	for all $r=1,\ldots,k$. 
	We claim that this happens if and only if there exist $\sigma \in S(\{ e,f \} \times [n])$ and $t: \{ e,f \} \times [n] \to [s]$ such that
	$$
	v(\alpha_i,x) = 
	((\sigma_1(\alpha,x))_{i+t(\alpha, x)}, \sigma_2(\alpha,x))
	$$
	for all $i \in [s]$ and $(\alpha,x) \in \{ e,f \} \times [n]$.
	Indeed, if this last condition holds then
\begin{align*}
& \prod_{(\alpha,x) \in \{ e,f \} \times [n]} 
((a_r,v(\alpha_1,x)), \ldots , (a_r,v(\alpha_s,x))) \\
& = \prod_{(\alpha,x) \in \{ e,f \} \times [n]}   ((a_r,((\sigma_1(\alpha,x))_{1+t(\alpha, x)}, \sigma_2(\alpha,x)),\ldots , (a_r,((\sigma_1(\alpha,x))_{s+t(\alpha, x)}, \sigma_2(\alpha,x))  \\
& = \prod_{(\beta,y) \in \{ e,f \} \times [n]}  ((a_r,\beta_{1+t(\sigma^{-1}(\beta,y))},y), \ldots , (a_r,\beta_{s+t(\sigma^{-1}(\beta,y))},y)) \ 
\end{align*}
for all $r \in [k]$. 
Conversely, if
	$$ 
	\prod_{(\alpha,x) \in \{ e,f \} \times [n]} 
	((a_r,v(\alpha_1,x)), \ldots , (a_r,v(\alpha_s,x)))  =
	\prod_{(\alpha,x) \in \{ e,f \} \times [n]} ((a_r,\alpha_1,x),\ldots , (a_r,\alpha_s,x)) 
	$$
	for all $r \in [k]$ then for any $(\alpha,x) \in \{ e,f \} \times [n]$ there exists $\sigma(\alpha,x) \in \{ e,f \} \times [n]$ such that
\begin{align*}
	((a_r,v(\alpha_1,x)), \ldots , & (a_r,v(\alpha_s,x)))  = \\
 & ((a_r,(\sigma_1(\alpha,x))_1,\sigma_2(\alpha,x)),\ldots , (a_r,(\sigma_1(\alpha,x))_s,\sigma_2(\alpha,x)))
\end{align*}	
	for all $r \in [k]$, and this map $\sigma$ must be a bijection
of $\{ e,f \} \times [n]$. Therefore
	$$
	(v(\alpha_1,x),\ldots , v(\alpha_s,x)) = 
	(((\sigma_1(\alpha,x))_1,\sigma_2(\alpha,x)),\ldots , ((\sigma_1(\alpha,x))_s,\sigma_2(\alpha,x)))
	$$
(as cycles) for all $(\alpha,x) \in \{ e,f \} \times [n]$, so there exists $t(\alpha,x) \in [s]$ such that
	$$
	v(\alpha_i,x) = 
	((\sigma_1(\alpha,x))_{i+t(\alpha,x)},\sigma_2(\alpha,x))
	$$
	for all $i \in [s]$ and $(\alpha,x) \in \{ e,f \} \times [n]$.
\end{proof}

While we will sometimes need the full generality of Proposition \ref{horizontal_same_columns_2_cycles} we use the following
special case of it very often.
\begin{proposition}
	\label{horizontal_same_columns2_cycles}
	Let $1 \leq a_1 < \cdots < a_k \leq n$, $\{ b_1, \ldots , b_s\} \subseteq [n]$,
	$v \in S([n]^2)$, and 
	\[
	u:= \prod_{r=1}^{k} \left( (a_r,b_1) , \ldots , (a_r,b_s) \right).
	\]	
	Then $v$ is compatible with $u$ if and only if there exist $\sigma \in S_n$ and $t: [n] \to [s]$ such that
	$$
	v(b_i,x) = (b_{i+t(x)}, \sigma(x))
	$$
	for all $i \in [s]$ and $x \in [n]$ (where indices are modulo $s$).
\end{proposition}

As an immediate consequence we obtain the following.

\begin{corollary}
	\label{horizontal_in_same_columns}
	Let $1 \leq a_1 < \cdots < a_k \leq n$, $1 \leq i < j \leq n$,
	and 
	\[
	u:= \prod_{r=1}^{k} \left( (a_r,i) , (a_r,j) \right).
	\]	
	Then $u$ is stable of rank $1$ if and only if
	either
	\[
	\{ i,j \} \subseteq \{ a_1, \ldots , a_k \}
	\]
	or
	\[
	\{ i,j \} \cap \{ a_1, \ldots , a_k \} = \emptyset.
	\]
	In particular, there are 
	$
	\binom{n}{k} \Big( \binom{n-k}{2} + \binom{k}{2} \Big)
	$
	such stable permutations of rank 1.
\end{corollary}
\begin{proof}
	The first statement follows immediately from Proposition
	\ref{horizontal_same_columns}.
	For the second one note first that the two cases are disjoint,
	and that there are $\binom{n}{k}$ choices for 
	$\{ a_1, \ldots , a_k \}$.
	Furthermore, in the first case there are $\binom{k}{2}$ possibilities for $\{  i, j \}$, and $\binom{n-k}{2}$ in the second one. 
\end{proof}
\noindent 
So, for example, for $n=4$ and $k=3$, there are 12 such permutations (see also Figure 4).
Together with the previous result for $k=2$ the next one covers 
all the pairs of ``horizontal'' transpositions.
\begin{proposition}
	Let $\{ a_1, a_2 \}, \{ i_1, j_1 \}, \{ i_2, j_2 \} \in \binom{[n]}{2}$
	be such that $| \{ i_1, j_1 \} \cap \{ i_2, j_2 \} | \leq 1$,
	$v \in S([n]^2)$, and 
	\[
	u:= \left( (a_1,i_1) , (a_1,j_1) \right) \left( (a_2,i_2) , (a_2,j_2) \right).
	\]
	Then $u$ is compatible with $v$ if and only if $v$ leaves 
	$[n] \times \{ a_1 \}$ and $[n] \times \{ a_2 \}$ invariant.
\end{proposition}
\begin{proof}
	We have that
	\begin{equation}
		\label{1tensoru}
		1 \otimes u = \prod_{x=1}^{n} \prod_{r=1}^{2} 
		\left( (x,a_r,i_r) , (x,a_r,j_r) \right)
	\end{equation}
	while
	\begin{equation}
		\label{1tensoru_conj}
		(v \otimes 1)(1 \otimes u)(v^{-1} \otimes 1) = \prod_{x=1}^{n} \prod_{r=1}^{2} 
		\left( (v(x,a_r),i_r) , (v(x,a_r),j_r) \right).
	\end{equation}
	Assume first that $\{ i_1, j_1 \} \cap \{ i_2, j_2 \}  = \emptyset$. If $u$ is compatible with $v$ then $\left( v(x,a_2), j_2 \right)$ is 
	not a fixed point of $1 \otimes u$ for all $x \in [n]$.
	Hence, since $j_2 \notin \{ i_1, j_1, i_2 \}$, there is 
	$y \in [n]$ such that $v(x,a_2) = (y,a_2)$ so $v(x,a_2) \in 
	[n] \times \{ a_2 \}$. Hence $v$ leaves $[n] \times \{ a_2 \}$
	invariant. Similarly, $v$ leaves $[n] \times \{ a_1 \}$ 
	invariant. Conversely, if $v$ leaves $[n] \times \{ a_1 \}$
	and $[n] \times \{ a_2 \}$ invariant then there are $\sigma_1, \sigma_2 \in S_n$ such that $v(x,a_i) = (\sigma_i(x),a_i)$
	for all $x \in [n]$ and all $i=1,2$ so the result follows 
	immediately from (\ref{1tensoru}) and (\ref{1tensoru_conj}).
	
	Suppose now that 
	$| \{ i_1, j_1 \} \cap \{ i_2, j_2 \} | = 1$. We may clearly
	assume that $i_1 = i_2$ so $j_1 \neq j_2$. Then  
	$j_2 \notin \{ i_1, j_1, i_2 \}$ and $j_1 \notin \{ i_1, j_2, i_2 \}$ so we conclude as in the previous case.
\end{proof}

We again have the following special case.
\begin{corollary}
	\label{horizontal_not_same_columns}
	Let $\{ a_1, a_2 \}, \{ i_1, j_1 \}, \{ i_2, j_2 \} \in \binom{[n]}{2}$
	be such that $| \{ i_1, j_1 \} \cap \{ i_2, j_2 \} | \leq 1$, and 
	\[
	u:= \left( (a_1,i_1) , (a_1,j_1) \right) \left( (a_2,i_2) , (a_2,j_2) \right).
	\]
	Then $u$ is stable of rank $1$ if and only if 
	\[
	\{ a_1,a_2 \} \cap \{ i_1,j_1,i_2,j_2 \} = \emptyset.
	\]
\end{corollary}

As a consequence of the previous results we obtain a proof 
of Conjecture 12.2 in \cite{BC} 
(cf. \cite{Pan}).
\begin{theorem} \label{solveconj}
	Let $(a_1,i_1), (a_1,j_1),(a_2,i_2),(a_2,j_2) \in [n]^2$
	be distinct, $a_1 \neq a_2$, and
	\[
	u:= \left( (a_1,i_1) , (a_1,j_1) \right) \left( (a_2,i_2) , (a_2,j_2) \right).
	\]
	Then $u$ is stable of rank $1$ if and only if either
	\begin{equation}
		\label{empty_intersection}
		\{ a_1,a_2 \} \cap \{ i_1,j_1,i_2,j_2 \} = \emptyset,
	\end{equation}
	or
	\begin{equation}
		\label{equal}
		\{ a_1,a_2 \} = \{ i_1,j_1,i_2,j_2 \}.
	\end{equation}
	In particular, there are 
	$
	\binom{n}{2} \Big( \binom{n-2}{2}^2 +1 \Big)
	$
	such permutations.
\end{theorem}
\begin{proof}
	Suppose first that $u$ is stable of rank $1$. If
	$| \{ i_1, j_1 \} \cap \{ i_2, j_2 \} | = 2$ (i.e., 
	if $ \{ i_1, j_1 \} = \{ i_2, j_2 \}$) then from Corollary \ref{horizontal_in_same_columns} (with $k=2$) we conclude that
	either (\ref{empty_intersection}) holds or 
	$\{ i_1, j_1 \} \subseteq \{a_1, a_2 \}$ which implies 
	that (\ref{equal}) holds. If 
	$| \{ i_1, j_1 \} \cap \{ i_2, j_2 \} | \leq 1$ then from 
	Corollary \ref{horizontal_not_same_columns} we conclude that 
	(\ref{empty_intersection}) holds.
	
	Conversely, if (\ref{empty_intersection}) holds then from 
	Corollaries \ref{horizontal_in_same_columns}  and
	\ref{horizontal_not_same_columns} we conclude that $u$ is
	stable of rank $1$, while if (\ref{equal}) holds then 
	necessarily $\{ i_1,j_1 \} = \{ i_2,j_2 \}$ so the result
	follows from Corollary \ref{horizontal_in_same_columns}.
	
	For the last statement note first that the two cases given in
	(\ref{empty_intersection}) and (\ref{equal}) are disjoint,
	and that there are $\binom{n}{2}$ choices for 
	$\{ a_1, a_2 \}$.
	Furthermore, in the first case, since $i_1 \neq j_1$,
	and $i_1,j_1 \notin \{ a_1, a_2 \}$, there are $\binom{n-2}{2}$ possibilities for
	$\{  i_1, j_1 \}$, and similarly for $\{ i_2,j_2 \}$. If
	(\ref{equal}) holds then necessarily $\{i_1,j_1\}=\{a_1,a_2\}
	=\{i_2,j_2 \}$. The result follows. 
\end{proof}
\noindent 
So, for example, for $n=4$ there are 12 such permutations (see also Fig. 3).
The following two results follow easily from Propositions \ref{horizontal_same_columns} and \ref{horizontal_same_columns2_cycles}
using some transposition techniques (see \cite[Prop. 5.13]{BC}), and their verifications are 
therefore omitted.
\begin{proposition}
	\label{vertical_same_rows2}
Let $1 \leq a_1 < \cdots < a_k \leq n$, $\{ b_1, \ldots , b_s\} \subseteq [n]$,
	$v \in S([n]^2)$, and 
	\[
	u:= \prod_{r=1}^{k} \left( (b_1,a_r) , \ldots , (b_s,a_r,) \right).
	\]	
Then $u$ is compatible with $v$ if and only if there exist $\sigma \in S_n$ and $t: [n] \to [s]$ such that
$$
v(x,b_i) = (\sigma(x), b_{i+t(x)})
$$
for all $x \in [n]$ and $i \in [s]$ (where indices are modulo $s$).
\end{proposition}

\begin{proposition}
	\label{vertical_same_rows}
	Let $1 \leq a_1 < \cdots < a_k \leq n$,
	$v \in S([n]^2)$, $\sigma \in S_n$, and $u \in S([n]^2)$ be such that
	\[
	u(x,y):= 
	\begin{cases}
		(\sigma(x),y) 
		\mbox{ if } y \in \{ a_1, \ldots , a_k\}, \\
		(x,y),
		\mbox{ otherwise},
	\end{cases}
	\]	
	for all $(x,y) \in [n]^2$. 
	Then $v$ is compatible with $u$ if and only if
	$v$ leaves $\{ a_1, \ldots , a_k \} \times [n]$ invariant.
\end{proposition}

\section{Subgroups and compatibility}
\label{Subgroups}
In this section we define some notable subgroups of $S([n]^2)$
which will be useful in the rest of this work, and establish some
of their basic properties.

Let $P \subseteq [n]$. We define a subgroup $R_P
\leq S([n]^2)$ by letting $u \in R_P$ if and only if:
\begin{itemize}
	\item[1)] 
	$u(\{ k \} \times ([n] \setminus P)) = \{ k \} \times ([n] \setminus P)$ for each $k \in P$;
	\item[2)] 
	$u(k,j)=(k,j)$ if either $k \in [n] \setminus P$ or $j \in P$.
\end{itemize}
So, given $u \in S([n]^2)$, $u \in R_P$ if and only if $u$ 
leaves point-wise invariant all the rows whose index is not in
$P$ and all columns whose index is in $P$, and leaves invariant 
all the rows indexed by elements of $P$.
\medskip 

We also define a subgroup $\check{R}_P
\leq S([n]^2)$ by letting $u \in \check{R}_P$ if and only if:
\begin{itemize}
	\item[1)] there is  $\sigma \in S([n] \setminus P)$ such that
	$u(k,j)=(k,\sigma(j))$ 
	for all $k,j \in [n] \setminus P$;
	\item[2)] 
	$u(k,j)=(k,j)$ if either $k \in P$ or $j \in P$.
\end{itemize}  
So, given $u$ as above, $u \in \check{R}_P$, if and only if 
there is  $\sigma \in S([n] \setminus P)$ such that 
$u(k,j)= (1 \otimes \sigma)(k,j)$
for all $(k,j) \in ([n] \setminus P) \times ([n] \setminus P)$,
and $u$ leaves point-wise invariant all rows and columns whose index is in $P$. Note that $\check{R}_P$ can also be obtained by performing the $(i_1)$-immersion of the $(i_2)$-immersion ... of the $(i_p)$-immersion of all the elements of 
$\{ 1 \} \otimes S([n] \setminus P)$ where $\{  i_1, \ldots , i_p \} =P$ (we refer the reader to \cite[Sec. 7]{BC} for details
about immersion).  

Analogously, we define subgroups $C_P$ and $\check{C}_P$ using columns.
Note that all elements in $\check{R}_P$ (resp., $\check{C}_P$) commute with every element of $R_P$ (resp., $C_P$), that
$R_P$ and $C_P$ are each isomorphic to $(S_{n-|P|})^{|P|}$, 
and that 
$\check{R}_P$ and $\check{C}_P$ are each isomorphic
to $S_{n-|P|}$.
We then have the following result.
\begin{theorem} \label{Bicomp1}
	\label{R_P}
	For each $P \subseteq [n]$ the subgroup generated by 
	$\check{R}_P \cup R_P$ is a bicompatible subgroup of $S([n]^2)$
	isomorphic to $(S_{n-|P|})^{|P|+1}$. Similarly for  
	$\check{C}_P \cup C_P$.
\end{theorem}
\begin{proof}
	Note first that all the elements in the subgroup generated by $\check{R}_P \cup R_P$ leave each row invariant, and fix each
	column in $P$ pointwise. Let $u,v$ be two such permutations
	and $(a,b,c) \in [n]^3$. By what has just been observed the
	first two coordinates of
	\[
	(1 \otimes u) (v \otimes 1)(a,b,c)
	\]
	and 
	\[
	(v \otimes 1) (1 \otimes u)(a,b,c)
	\]
	are equal to $v(a,b)$. If $c \in P$ then $u(x,c)=(x,c)$
	for all $x \in [n]$ so $(1 \otimes u) (v \otimes 1)(a,b,c)=
	(v(a,b),c)=(v \otimes 1) (1 \otimes u)(a,b,c)$. So assume that
	$c \notin P$. If $b \in P$ then 
	$(1 \otimes u) (v \otimes 1)(a,b,c)=(a,u(b,c))=
	(v \otimes 1) (1 \otimes u)(a,b,c)$, since $u(b,c)$ is in row $b$
	and $v$ fixes the $b$-th column point-wise. Finally, if $b \notin P$ then there is $d \notin P$ such that $v(a,b)=(a,d)$. Hence,
	\[
	(1 \otimes u) (v \otimes 1)(a,b,c)=(1 \otimes u)(a,d,c)
	=(a,d,\sigma(c))
	\]
	while
	\[
	(v \otimes 1) (1 \otimes u)(a,b,c)=(v \otimes 1)(a,b,\sigma(c))=
	(a,d,\sigma(c)).
	\]
	This shows that any two elements of $\check{R}_P \cup R_P$ are
	bicompatible, so the result follows.
\end{proof}

For $i \in [n]$ we find it convenient to let
$R_i:=R_{\{ i \}}$ and define similarly $C_i$, $\check{R}_i$,
and $\check{C}_i$.  We also define elements $\check{r}_{i,k}, \check{c}_{i,k} \in S([n]^2)$, 
for $k \in [n-1] \setminus \{ i \}$, by letting
\[
\check{r}_{i,k} := 
\begin{cases}
	\prod_{j \in [n] \setminus \{ i \} }((j,k),(j,k+1)), 
	\mbox{ if } k \neq i-1, \\
	\prod_{j \in [n] \setminus \{ i \} }((j,i-1),(j,i+1)),
	\mbox{ if } k = i-1,
\end{cases}
\]
and 
\[
\check{c}_{i,k} := 
\begin{cases}
	\prod_{j \in [n] \setminus \{ i \} }((k,j),(k+1,j)), 
	\mbox{ if } k \neq i-1, \\
	\prod_{j \in [n] \setminus \{ i \} }((j,i-1),(i+1,j)),
	\mbox{ if } k = i-1.
\end{cases}
\]
so that 
$
\check{R}_i := \langle \{ \check{r}_{i,k} : k \in [n-1] \setminus \{ i \}   \}  \rangle
$
and similarly for $\check{C}_i$.

\bigskip

For 
$i,j \in [n]$, $i \neq j$ we let
\[
R_{i,j} :=\{u \in R_i \ | \ u(i,j) = (i,j)\} \ , \quad C_{i,j} =\{u \in C_i \ | \ u(j,i) = (j,i)\} \ . 
\]

The elements of these subgroups possess various compatibility 
properties.
\begin{lemma}\label{compdiffR}
	Let $i,j \in [n]$ and $u \in R_{i,j}$, $v \in R_{j,i}$. Then $u$ and $v$ are bicompatible.
	Similarly for $C_{i,j}$ and $C_{j,i}$.
\end{lemma}
\begin{proof}
	It is clearly enough to show that $u$ is compatible with $v$. Let $a,b,c \in [n]$.
	If $a \neq j$ and $b \neq i$ then we have that
	\begin{align*}
		(v \otimes 1) (1 \otimes u)(a,b,c)  & =  (v \otimes 1) (a,b,c) = (a,b,c) =
		(1 \otimes u)(a,b,c) \\ & = (1 \otimes u)(v \otimes 1)(a,b,c). 
	\end{align*}
	If $a=j$ and $b \neq i$ then $v(j,b) \neq (j,i)$ so $v(j,b)_2 \neq i$ and so we
	obtain that 
	\[
	(1 \otimes u) (v \otimes 1) (j,b,c) = (1 \otimes u) (v(j,b),c)=(v(j,b),c).
	\]
	while
	\[
	(v \otimes 1) (1 \otimes u)(j,b,c) = (v \otimes 1)(j,b,c)=(v(j,b),c).
	\]
	If $a \neq j$ and $b=i$ then we have 
	\begin{align*}
		(v \otimes 1) (1 \otimes u)(a,i,c)  & =  (v \otimes 1) (a,u(i,c)) =   
		(a,u(i,c))= (1 \otimes u)(a,i,c) \\ & = (1 \otimes u)(v \otimes 1)(a,i,c). 
	\end{align*}
	Finally, if $a=j$ and $b=i$ then, since $u(i,c)_1=i$ and $v(j,i)=(j,i)$,
	we conclude that
	\begin{align*}
		(v \otimes 1) (1 \otimes u)(j,i,c)  & =  (v \otimes 1) (j,u(i,c)) = (j,u(i,c)) =   
		(1 \otimes u)(j,i,c) \\ & = (1 \otimes u)(v(j,i),c) = (1 \otimes u)(v \otimes 1)(j,i,c). 
	\end{align*}
\end{proof}

For $i \in [n]$ we let 
\[
S_i([n]) := \{ u \in S_n  : u(i)=i \},
\]
namely the stabilizer subgroup ${\rm Stab}(i)$ of $S_n$,
and for $i,j \in [n]$, $i \neq j$ 
\[
S_{i,j}:= \{u \in S_n \ | \ u(i)=i, \ u(j) = j\} = S_i([n]) \cap S_j([n]) = S_{j,i}\ , 
\] 
namely the stabilizer of $\{i,j\}$ in $S_n$.
Of course, these groups are isomorphic to $S_{n-1}$ and $S_{n-2}$,
respectively.

For $\sigma \in S_i([n])$ we let 
\[
u_\sigma(a,b)= 
\begin{cases}
	(a,b), \mbox{ if } b \neq i, \\
	(\sigma(a),b), \mbox{ if } b=i,
\end{cases}
\]
for all $a,b \in [n]$ (so $u_\sigma \in C_i$), and define 
$v_\sigma \in R_i$ similarly. 

\begin{proposition}
	\label{comp_iff_comm}
	Let $i \in [n]$ and $\sigma, \tau \in S_i([n])$. Then $u_\sigma$ is compatible with 
	$v_\tau$ if and only if $\sigma \tau = \tau \sigma$ in $S_n$.
\end{proposition}
\begin{proof}
	Let $ x,y,z \in [n]$. If either $x \neq i$ or $z \neq i$ then it is easy
	to check that $(v_\tau \otimes 1) (1 \otimes u_\sigma)(x,y,z) = 
	(1 \otimes u_\sigma)(v_\tau \otimes 1)(x,y,z)$. On the other hand we have
	that $(v_\tau \otimes 1) (1 \otimes u_\sigma)(i,y,i) = (i, \tau(\sigma(y)),i)$
	and $(1 \otimes u_\sigma) (v_\tau \otimes 1)(i,y,i) = (i, \sigma(\tau(y)),i)$
	so the result follows.
\end{proof}

The following compatibility results only hold for $n=4$.
They partly explain why certain subgroups that we find later in this work are probably not part of more general families.

\begin{lemma}\label{colrow1}
	Let $i,j \in [4]$ $i \neq j$, and $u \in C_{i,j}$, $v \in R_{j,i}$. Then $u$ is compatible with $v$.
\end{lemma}
\begin{proof}
	Let $ x,y,z \in [4]$. If either $z \neq i$ or $x \neq j$ then it is easy
	to check that $(v \otimes 1) (1 \otimes u)(x,y,z) = 
	(1 \otimes u)(v \otimes 1)(x,y,z)$.
	Let $\{ a,b \} := [4] \setminus \{ i,j \}$.
	Since $u$ and $v$ are such that $u(i,i)=(i,i)$, $v(j,j)=(j,j)$, and 
	$u(j,i)=(j,i)=v(j,i)$ then either one of them is the identity or they both
	act as the transposition $(a,b)$ on the $i$-th column and $j$-th row, respectively,
	and as the identity everywhere else. In the first cases the result is trivial,
	in the last one it is easy to check that $(v \otimes 1) (1 \otimes u)(j,y,i) = 
	(1 \otimes u)(v \otimes 1)(j,y,i)$ for all $y \in [4]$.
\end{proof}

\begin{corollary} \label{colrow2}
	Let $i,j \in [4]$, $i \neq j$, and $u \in C_{i,j}$, $v \in R_{i,j}$. Then $u$ is compatible with $v$.
\end{corollary}
\begin{proof}
	Let $\{ a,b \} := [4] \setminus \{ i,j \}$.
	By our hypotheses either 
	$u$ or $v$ is the identity or $u=u_\sigma$ and $v=v_\sigma$ where 
	$\sigma:=(a,b)$. The result then follows from Corollary \ref{comp_iff_comm}.
\end{proof}

Note that the previous results do not hold, in general, if $n \geq 5$.
For example, if $n=8$, $i=5$, $j=2$, $u=((1,5),(4,5))((3,5),(6,5))$, and
$v=((2,3),(2,4))((2,6),(2,8))$ then $(v \otimes 1) (1 \otimes u)(2,3,5)=
(2,8,5) \neq (2,1,5)= (1 \otimes u) (v \otimes 1)(2,3,5)$. Similarly,
if $n=8$, $i=3$, $j=8$, $u=((2,3),(6,3))((4,3),(5,3))$, and
$v=((3,1),(3,6))((3,2),(3,5))$ then $(v \otimes 1) (1 \otimes u)(3,2,3)=
(3,1,3) \neq (3,4,3)= (1 \otimes u) (v \otimes 1)(3,2,3)$.

\bigskip 

For $m \in \NN$ let $C \in \binom{[2m]}{m}$,
and $\prec$ be a total order on $[2m]$.
Let $B := [2m] \setminus C$. 
We define two elements
$c_0,r_0 \in S([2m]^2)$ by
\[
c_0:=\prod_{y \in C} ((c_1,y), \ldots , (c_m,y)),
\]
and
\[
r_0:=\prod_{x \in B} ((x,b_1), \ldots , (x,b_m)),
\]
where $\{ c_1, \ldots , c_m \} := C$,
$\{ b_1, \ldots , b_m \} := B$, $c_1 \prec \cdots \prec c_m$, and
$b_1 \prec \cdots \prec b_m$,
We let 
$N(B,C)_\prec$ be the subgroup of $S([2m]^2)$ generated by
$c_0$, $r_0$, and $S(C \times B)$,
where this last group is embedded in $S([2m]^2)$ in the natural
way.
\begin{theorem} \label{NBCprec}
Let $m \in \NN$, $C \in \binom{[2m]}{m}$, 
$B := [2m] \setminus C$, and 
$\prec$ be a total order of $[2m]$. Then 
$N(B,C)_\prec$ is a bicompatible subgroup of $S([2m]^2)$
isomorphic to $(S_{m^2}) \times (\ZZ_{m})^{\times 2}$. 
\end{theorem}
\begin{proof}
We begin by showing that every element in $N(B,C)_\prec$ is compatible with $c_0$. By Proposition \ref{vertical_same_rows}
an element $v \in S([2m]^2)$ is compatible with $c_0$ if and 
only if $v$ leaves $C \times [2m]$ invariant, and this is easily checked for all the elements of $N(B,C)_\prec$. In particular,
$c_0$ is (bi)compatible with itself (i.e., it is stable of rank $1$).

We now show that $c_0$ is compatible with every other element in $N(B,C)_\prec$. By Proposition \ref{vertical_same_rows2}
$c_0$ is compatible with an element $v \in S([2m]^2)$ if and
only if there are $\sigma \in S_{2m}$ and $t: [2m] \rightarrow [m]$ such that 
\[
v(x,c_i)=(\sigma(x),c_{i+t(x)})
\] 
for all $x \in [2m]$ and all $i \in [m]$ (where indices are modulo
$m$). This condition is satisfied for all the elements in
$\langle S(B,C), r_0 \rangle$
by taking $t(x):=0$ for all $x \in [2m]$ and $\sigma:=1$, since all these
elements leave fixed each column in $C$. 
This is already enough to deduce that $c_0$ is compatible with every element in $N(B,C)_\prec$.

Similarly, we can check that $r_0$ is bicompatible with any element in $N(B,C)_\prec$.
Indeed, by Proposition \ref{horizontal_same_columns}, $r_0$ is compatible with some $v \in S([2m]^2)$ if and only if $v$ leaves $[2m] \times B$ invariant,
and this is easily verified for all $v \in N(B,C)_\prec$. 
Conversely, one can use Proposition \ref{horizontal_same_columns2_cycles} (with $\sigma=1$ and $t(x)=0$ for all $x \in [2m]$) to check that all the elements $S(C \times B)$ are compatible with $r_0$, and the claim readily follows.

Finally, one can easily see from \cite[Proposition 5.15]{BC} that any two elements in $S(C \times B)$ are bicompatible.
The conclusion is now clear.
\end{proof}

\section{Quadratic subgroups} \label{Quadraticsub}
In this section, using the results obtained in the previous 
ones, we give a detailed and explicit description of the
subgroups of ${\rm Out}(\O_4)$ which are maximal in $\pi(\lambda^{-1}(\P_4^2))$,
so the ones that are tabulated in \cite[Table 5]{AJS18} 
(reproduced here as Table 1 for the reader's convenience).
Our results show that some of them are special cases of
subgroups that exist in ${\rm Out}(\O_n)$ for any $n \geq 3$, while
others seem to be specific to the $n=4$ case.

\begin{table}[ht] \label{Table1}
	\caption{Maximal shift space automorphism groups at level 2 in  $\O_4$}
	\centering
	\begin{tabular}{c c c c}
		\hline\hline
		Case & Order & Group & Multiplicity  \\ [0.5ex] 
		\hline
		1& 128 & 
		$\big( {\mathbb Z}_2 \times {\mathbb Z}_2 \times {\mathbb Z}_2 \times {\mathbb Z}_2 \times {\mathbb Z}_2 \times {\mathbb Z}_2\big)\rtimes {\mathbb Z}_2$ & 3 \\
		2& 96 & ${\mathbb Z}_2 \times {\mathbb Z}_2 \times S_4 $ & 6 \\
		3& 64 & ${\mathbb Z}_2 \times {\mathbb Z}_2 \times {\mathbb Z}_2 \times D_8$ & 12 \\
		4 & 64 & ${\mathbb Z}_2 \times \big( ({\mathbb Z}_2 \times {\mathbb Z}_2 \times {\mathbb Z}_2 \times {\mathbb Z}_2) \rtimes {\mathbb Z}_2 \big)$ & 6 \\
		5 & 54 & $({\mathbb Z}_3 \times {\mathbb Z}_3 \times {\mathbb Z}_3) \rtimes {\mathbb Z}_2 $ & 4 \\ 
		6 & 36 & $S_3 \times S_3$ & 8 \\
		7 & 24 & $S_4$ & 7 \\  [1ex]
		\hline
	\end{tabular}
	\label{table:nonlin}
\end{table}

\subsection{The first line}

We keep the notation as in the previous section. So $i,j \in [4]$, $i \neq j$, and 
$\{h,k \} := [4] \setminus \{ i,j \}$.
Let $u_0,v_0,w_0,z_0,u'_0,v'_0,w'_0,z'_0,\sigma_0,\sigma'_0$ be the generators of $R_{i,j},R_{j,i},C_{i,j},C_{j,i},$ $R_{h,k},R_{k,h},C_{h,k},C_{k,h},S_{h,k},S_{i,j}$, respectively. 
So 
\[
u_0=((i,h),(i,k)), v_0=((j,h),(j,k)), w_0=((h,i),(k,i)), z_0=((h,j),(k,j))
\]
\[
u'_0=((h,i),(h,j)), v'_0=((k,i),(k,j)),
w'_0=((i,h),(j,h)), z'_0=((i,k),(j,k))
\]
\[
\sigma_0=(i,j), \sigma'_0=(h,k).
\]
\noindent 
Note that $u_0,v_0 \in R_{\{ i,j \}}$, $w_0,z_0 \in C_{\{ i,j \}}$,
$u'_0,v'_0 \in R_{\{ h,k \}}$, and $w'_0,z'_0 \in C_{\{ h,k \}}$.
Also define $\omega_0:=(i,h)(j,k) \in S_4$.
We let $H_{i,j}$ be the subgroup of ${\rm Aut}(\O_4)$ generated by 
$$
\{ \lambda_{u_0} \lambda_{v_0}, \lambda_{w_0} \lambda_{z_0}, \lambda_{u'_0} \lambda_{v'_0}, \lambda_{w'_0} \lambda_{z'_0}, \lambda_{\sigma_0 \otimes 1}, 
\lambda_{\sigma'_0 \otimes 1},\lambda_{\omega_0 \otimes 1} \}
$$
(so $H_{i,j}=H_{j,i}=H_{h,k}=H_{k,h}$). Observe that, by Theorem \ref{R_P} (or Lemma
\ref{compdiffR}), $\lambda_{u_0} \lambda_{v_0}=\lambda_{u_0 v_0}$,
$\lambda_{w_0} \lambda_{z_0}=\lambda_{w_0 z_0}$,
$\lambda_{u'_0} \lambda_{v'_0}=\lambda_{u'_0 v'_0}$, and
$\lambda_{w'_0} \lambda_{z'_0}=\lambda_{w'_0 z'_0}$.
For the reader's benefit we include a graphical representation of the seven permutations $u_0 v_0$, $w_0 z_0$, $u'_0 v'_0$, $w'_0 z'_0$, $\sigma_0 \otimes 1$, $\sigma'_0 \otimes 1$
and $\omega_0 \otimes 1$ in $S([4]^2)$ for $i=1$, $j=3$,
$h=2$, and $k=4$.
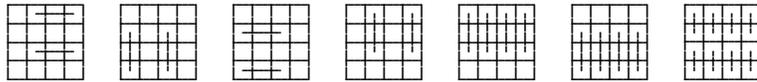
\[  \beginpicture
\setcoordinatesystem units <0.25cm,0.25cm>
\setplotarea x from -5 to 20, y from 2 to 5

\setlinear

\plot 0 0 0 4 /
\plot 1 0 1 4 /
\plot 2 0 2 4 /
\plot 3 0 3 4 /
\plot 4 0 4 4 /

\plot 0 0 4 0 /
\plot 0 1 4 1 /
\plot 0 2 4 2 /
\plot 0 3 4 3 /
\plot 0 4 4 4 /

\plot 1.5 1.5 3.5 1.5 /
\plot 1.5 3.5 3.5 3.5 /

\plot 6 0 6 4 /
\plot 7 0 7 4 /
\plot 8 0 8 4 /
\plot 9 0 9 4 /
\plot 10 0 10 4 /

\plot 6 0 10 0 /
\plot 6 1 10 1 /
\plot 6 2 10 2 /
\plot 6 3 10 3 /
\plot 6 4 10 4 /

\plot 6.5 0.5 6.5 2.5 /
\plot 8.5 0.5 8.5 2.5 /

\plot 12 0 12 4 /
\plot 13 0 13 4 /
\plot 14 0 14 4 /
\plot 15 0 15 4 /
\plot 16 0 16 4 /

\plot 12 0 16 0 /
\plot 12 1 16 1 /
\plot 12 2 16 2 /
\plot 12 3 16 3 /
\plot 12 4 16 4 /

\plot  12.5 0.5 14.5 0.5 /
\plot  12.5 2.5 14.5 2.5 /

\plot 18 0 18 4 /
\plot 19 0 19 4 /
\plot 20 0 20 4 /
\plot 21 0 21 4 /
\plot 22 0 22 4 /

\plot 18 0 22 0 /
\plot 18 1 22 1 /
\plot 18 2 22 2 /
\plot 18 3 22 3 /
\plot 18 4 22 4 /

\plot 19.5 1.5 19.5 3.5 /
\plot 21.5 1.5 21.5 3.5 /

\plot 24 0 24 4 /
\plot 25 0 25 4 /
\plot 26 0 26 4 /
\plot 27 0 27 4 /
\plot 28 0 28 4 /

\plot 24 0 28 0 /
\plot 24 1 28 1 /
\plot 24 2 28 2 /
\plot 24 3 28 3 /
\plot 24 4 28 4 /

\plot 24.5 1.5 24.5 3.5 /
\plot 25.5 1.5 25.5 3.5 /
\plot 26.5 1.5 26.5 3.5 /
\plot 27.5 1.5 27.5 3.5 /

\plot 30 0 30 4 /
\plot 31 0 31 4 /
\plot 32 0 32 4 /
\plot 33 0 33 4 /
\plot 34 0 34 4 /

\plot 30 0 34 0 /
\plot 30 1 34 1 /
\plot 30 2 34 2 /
\plot 30 3 34 3 /
\plot 30 4 34 4 /

\plot 30.5 0.5 30.5 2.5 /
\plot 31.5 0.5 31.5 2.5 /
\plot 32.5 0.5 32.5 2.5 /
\plot 33.5 0.5 33.5 2.5 /

\plot 36 0 36 4 /
\plot 37 0 37 4 /
\plot 38 0 38 4 /
\plot 39 0 39 4 /
\plot 40 0 40 4 /

\plot 36 0 40 0 /
\plot 36 1 40 1 /
\plot 36 2 40 2 /
\plot 36 3 40 3 /
\plot 36 4 40 4 /

\plot 36.5 0.5 36.5 1.5 /
\plot 36.5 2.5 36.5 3.5 /
\plot 37.5 0.5 37.5 1.5 /
\plot 37.5 2.5 37.5 3.5 /
\plot 38.5 0.5 38.5 1.5 /
\plot 38.5 2.5 38.5 3.5 /
\plot 39.5 0.5 39.5 1.5 /
\plot 39.5 2.5 39.5 3.5 /

\put {The seven permutations in $S([4]^2)$ indexing the generators of $H_{i,j}$} at 20 -2
\put {for $i=1, j=3, h=2$, and $k=4$.} at 20 -4

\endpicture \]

We leave to the reader to check that if we had replaced $\omega_0$ with the other possible choice $\omega'_0:=(i,k)(j,h) = \sigma_0 \omega_0 \sigma_0  = \sigma'_0 \omega_0 \sigma'_0 \in S_4$ 
(which corresponds to switching $h$ and $k$)
in the definition of $H_{ij}$ then this group would remain unchanged.
\begin{theorem}
	For $i,j \in [4]$, $i \neq j$, $H_{i,j}$
	is contained in $\lambda^{-1}(\P_4^2)$ and is isomorphic to $\ZZ_2^{\times 6} \rtimes \ZZ_2$.
	The kernel of the canonical projection $\pi : H_{i,j} \rightarrow {\rm Out}(\O_4)$
	is trivial and the 3 groups $\pi(H_{i,j})$ are mutually distinct.
\end{theorem}

\begin{proof}
	It is easy to check that the six elements $u_0 v_0$, $w_0 z_0$, $u'_0 v'_0$, $w'_0z'_0$, $\sigma'_0 \otimes 1$, and $\sigma_0 \otimes 1$ in $\P_4^2$
	commute with each other. 
	We now show that they are all bicompatible.
Indeed, by Propositions \ref{horizontal_same_columns} and
\ref{horizontal_same_columns2_cycles}, a permutation is bicompatible with 
$u_0 v_0$ (resp. $u'_0 v'_0$) if and only if it leaves $[n] \times \{  i,j \}$ (resp. $[n] \times \{  h,k \}$) 
invariant and permutes the elements of $\{ h,k \} \times [n]$
(resp.,  $\{ i,j \} \times [n]$) 
in the way prescribed by Proposition \ref{horizontal_same_columns2_cycles} and this is easily checked
for all these permutations. Similarly, a permutation is 
bicompatible with $w_0 z_0$ (resp., $w'_0 z'_0$) if and only if 
it leaves $\{ i,j \} \times [n]$ (resp., $\{ h,k \}
\times [n]$) invariant and permutes the elements of 
$[n] \times \{ h,k \}$ (resp., $[n] \times \{ i,j \}$) in the appropriate way, which is also easily checked. Finally, by 
the same results, a 
permutation is bicompatible with $\sigma'_0 \otimes 1$ if
and only if it leaves $[4] \times [4]$ invariant (which is a 
vacuous condition) and permutes the elements of $[n] \times \{
 h,k \}$ in the appropriate way which is the case for 
 $\sigma_0 \otimes 1$.

	This, by Proposition \ref{isogroups}, shows that the group generated by $\lambda_{u_0 v_0}$, $\lambda_{w_0 z_0}$, $\lambda_{u'_0 v'_0}$, $\lambda_{w'_0 z'_0}$, $\lambda_{\sigma'_0 \otimes 1}$, and $\lambda_{\sigma_0 \otimes 1}$ in ${\rm Aut}(\O_4)$ is contained in $\lambda^{-1}(\P_4^2)$ and is isomorphic to the group generated by
	$u_0 v_0$, $w_0 z_0$, $u'_0 v'_0$, $w'_0z'_0$, $\sigma'_0 \otimes 1$, and $\sigma_0 \otimes 1$ in $S([4]^2)$. This last group is easily seen to be isomorphic to
	$(\ZZ_2)^{\times 6}$. Indeed, let $(\alpha, \beta, \gamma, \delta, \mu, \nu) \in (\langle u_0 v_0 \rangle, \langle w_0 z_0 \rangle,\langle u'_0 v'_0 \rangle,\langle w'_0 z'_0 \rangle,\langle \sigma'_0 \otimes 1 \rangle,\langle \sigma_0 \otimes 1 \rangle)$ be such that $\alpha \beta \gamma \delta \mu \nu =1$ in $S([4]^2)$.
	Then $(\mu \nu) (a,a)= (\mu \nu \alpha \beta \gamma \delta) (a,a)=(a,a)$ for 
	all $a \in [4]$ which implies that $\mu = \nu =1$. Furthermore, $\alpha, \beta$ and $\gamma$ all leave the $i$-th row of $[4]^2$ invariant, so $\delta=1$. Similarly, 
	by considering the $i$-th column, we conclude that $\gamma=1$ , which in turn implies 
	that $\alpha=\beta=1$.
	\medskip
	
	Now, $\omega_0 \, \sigma_0 \, \omega_0 = \sigma'_0$ in $S_4$ so 
$$ \lambda_{\omega_0 \otimes 1} \lambda_{\sigma_0 \otimes 1} = \lambda_{\omega_0 \sigma_0} =
	\lambda_{\sigma'_0 \omega_0} = \lambda_{\sigma'_0 \otimes 1} \lambda_{\omega_0 \otimes 1}.
$$
Similarly 
	$(\omega_0 \otimes \omega_0) \; u_0 v_0 \; (\omega_0 \otimes \omega_0) =
	u'_0 v'_0$ 
	in $S([4]^2)$, so 
\begin{align*}
	\lambda_{\omega_0 \otimes 1} \lambda_{u_0 v_0} & = 
	\lambda_{\omega_0} \lambda_{u_0 v_0} = \lambda_{(\omega_0 \otimes \omega_0) u_0 v_0 (1 \otimes \omega_0)} =  \lambda_{u'_0 v'_0 ( \omega_0 \otimes 1)} \\
	& = \lambda_{u'_0 v'_0} \lambda_{\omega_0}
	=  \lambda_{u'_0 v'_0} \lambda_{\omega_0 \otimes 1}
\end{align*}
and $(\omega_0 \otimes \omega_0) \; w_0 z_0 \; ( \omega_0 \otimes \omega_0) =
	w'_0 z'_0$ 
	in $S([4]^2)$ so 
\begin{align*}
\lambda_{\omega_0 \otimes 1} \lambda_{w_0 z_0} & = 
\lambda_{\omega_0} \lambda_{w_0 z_0} =
\lambda_{( \omega_0 \otimes \omega_0 ) 
w_0 z_0 (1 \otimes \omega_0 )} 
=  \lambda_{w'_0 z'_0 ( \omega_0  \otimes 1)} \\
 & = \lambda_{w'_0 z'_0} \lambda_{\omega_0} 
= \lambda_{w'_0 z'_0} \lambda_{\omega_0 \otimes 1}.
\end{align*}
That $H_{i,j}$
is contained in $\lambda^{-1}(\P_4^2)$ and is isomorphic to $\ZZ_2^{\times 6} \rtimes \ZZ_2$ readily follows from these relations and the facts that $\lambda_{\omega_0}$ is an
involution and 
$\lambda_{\omega_0} \notin \langle \lambda_{u_0 v_0}, \lambda_{w_0 z_0}, \lambda_{u'_0 v'_0}, \lambda_{w'_0 z'_0}, \lambda_{\sigma'_0 \otimes 1},$ $\lambda_{\sigma_0 \otimes 1} \rangle $. Indeed, let $(\alpha, \beta, \gamma, \delta, \mu, \nu)$ as above 
	be such that $\lambda_{\alpha} \lambda_{\beta} \lambda_{\gamma} \lambda_{\delta} \lambda_{\mu} \lambda_{\nu} = \lambda_{\omega_0}$. Then 
	$\lambda_{\alpha \beta \gamma \delta \mu \nu} =\lambda_{\omega_0}$ and so
	$\alpha \beta \gamma \delta \mu \nu = \omega_0 \otimes 1$. But the permutation on the 
	LHS of this last equation leaves invariant the union of the $i$th and $j$-th rows of
	$[4]^2$, while $\omega_0 \otimes 1$ does not.
	
	\medskip
	We now show that the intersection of $H_{ij}$ with the kernel of the canonical projection onto ${\rm Out}(\O_4)$ is trivial. Let $(\alpha, \beta, \gamma, \delta, \mu, \nu)$ as above, $\eta \in \langle \omega_0 \otimes 1 \rangle$, and $p \in S_4$ be such 
	that $\lambda_{\alpha} \lambda_{\beta} \lambda_{\gamma} \lambda_{\delta} \lambda_{\mu} \lambda_{\nu} \lambda_{\eta} = \lambda_{p \otimes p^{-1}}$. Then 
	$\lambda_{\alpha \beta \gamma \delta \mu \nu \eta} = \lambda_{p \otimes p^{-1}}$
	which implies that $\alpha \beta \gamma \delta \mu \nu \eta =p \otimes p^{-1}$.
	Hence $\eta^{-1} \nu^{-1} \mu ^{-1} (\alpha \beta \gamma \delta)^{-1} = p ^{-1} \otimes p$ so $(\eta^{-1} \nu^{-1} \mu ^{-1})(a,a) = (p^{-1}(a),p(a))$ for all 
	$a \in [4]$ which shows that $p=1$, since $\eta, \nu$, and $\mu$ don't change the second coordinate.
	From the above discussion it follows at once that $\alpha$, $\beta$, $\gamma$, $\delta$, $\mu$, $\nu$ and $\eta$ are all trivial.

	\medskip
	It remains to be shown that the images of the three groups $H_{12}, H_{13}, H_{1,4}$ are distinct in ${\rm Out}(\O_4)$. Let $\ell, m \in [2,4]$, $\ell  < m$, 
	and $\{ h \} := [4] \setminus \{ 1,\ell ,m \}$.
	Then $[\lambda_{(\ell,h) \otimes 1}] \in \pi(H_{1,m})$. 
	We claim that $[\lambda_{(\ell,h) \otimes 1}] \notin 
	\pi(H_{1,\ell})$. Suppose not, then there are  
	$(\alpha, \beta, \gamma, \delta, \mu, \nu, \eta) \in (\langle u_0 v_0 \rangle, \langle w_0 z_0 \rangle,\langle u'_0 v'_0 \rangle,\langle w'_0 z'_0 \rangle,\langle \sigma'_0 \otimes 1 \rangle,\langle \sigma_0 \otimes 1 \rangle, \langle \omega_0 \otimes 1 \rangle)$ (where the elements on the RHS of this last equation are here taken in $H_{1,\ell}$ so
	$\alpha \in \langle  ((1,h),(1,m))((\ell,h),(\ell,m)) \rangle$,
	$\beta \in \langle  ((h,1),(m,1))((h,\ell),(m,\ell)) \rangle$,
	$\gamma \in \langle  ((h,1),(h,\ell))((m,1),(m,\ell)) \rangle$,
	$\delta \in \langle  ((1,h),(\ell,h))((1,m),(\ell,m)) \rangle$,
	$\mu \in \langle  (h,m) \otimes 1 \rangle$,
	$\nu \in \langle  (1,\ell) \otimes 1 \rangle$, and
	$\eta \in \langle  (1,h)(\ell,m) \otimes 1 \rangle$)  
	such that $[\lambda_{\alpha \beta \gamma \delta \mu \nu
		\eta}]=[\lambda_{(\ell,h) \otimes 1}]$. Hence there is $p \in S_4$ such 
	that $\alpha \beta \gamma \delta \mu \nu
	\eta= (p (\ell,h))\otimes p^{-1}$. Let $a \in [4]$ and $(b,a) := (\eta^{-1}
	\nu^{-1} \mu^{-1})(a,a)$. Then 
	$$
	(a,a)=\alpha \beta \gamma \delta \mu \nu
	\eta (b,a)= ((p (\ell,h))\otimes p^{-1})(b,a)=(p(\ell,h)(b),p^{-1}(a))
	$$
	so $p=1$. Hence $\alpha \beta \gamma \delta \mu \nu
	\eta= (\ell,h)\otimes 1$ and therefore 
	$\alpha \beta \gamma \delta = ((\ell,h) \otimes 1) \eta^{-1} \nu^{-1} \mu^{-1} $.
	In particular 
	$$
	(a,a)=(\alpha \beta \gamma \delta)(a,a)=((\ell,h) \otimes 1)\eta^{-1} \nu^{-1} \mu^{-1}(a,a)
	$$ 
	for all $a \in [4]$. 
	This implies that $(\ell,h) = \tilde{\mu} \tilde{\nu} \tilde{\eta}$ where 
	$\tilde{\mu} \in \langle (h,m) \rangle$, $\tilde{\nu} \in \langle (1,\ell) \rangle$,
	and $\tilde{\eta} \in \langle (1,h)(\ell,m) \rangle$. 
	But $(\ell,h)$ leaves $1$ fixed so $\tilde{\nu}=1$ and
 $\tilde{\eta} = 1$ so $(\ell,h) \in \langle (h,m) \rangle$ which is a contradiction.
	Hence $[\lambda_{(\ell,h) \otimes 1}] \notin 
	\pi(H_{1,\ell})$ and so $\pi(H_{1,\ell}) \neq \pi(H_{1,m})$.
\end{proof}

Note that the three groups $H_{1,2},H_{1,3},H_{1,4}$ are conjugate 
in $\rm Aut(\O_4)$ by Bogolubov automorphisms. Indeed, 
we have that 
$(H_{1,\ell})^{\lambda_{(\ell,m)}} = H_{1,m}$ for all $2 \leq \ell <m \leq 4$ since the indices of the corresponding generators in $S([4]^2)$ 
are conjugated by $(\ell,m) \otimes (\ell,m)$.
\medskip

\subsection{The second line}

We keep some notation as in the previous section. So $i,j \in [4]$, $i \neq j$, 
$\{h,k \} := [4] \setminus \{ i,j \}$, and
\[
u_0=((i,h),(i,k)), \, v_0=((j,h),(j,k)), \, w'_0=((i,h),(j,h)).
\]
We also let
\[
x_0=((i,i),(j,i)) ((i,j),(j,j)), \; y_0=((h,h),(h,k)) ((k,h),(k,k)).
\]
Let $K_{i,j}$ be the subgroup of ${\rm Aut}(\O_4)$ generated by 
$$
\{ \lambda_{u_0}, \lambda_{v_0}, \lambda_{w'_0}, \lambda_{x_0}, \lambda_{y_0} \}.
$$
Note that $K_{i,j}=K_{j,i}$. 
For the reader's convenience we include a graphical representation of the five permutations $u_0$, $w'_0$, $v_0$, $x_0$, $y_0$ in $S([4]^2)$ for $i=1$, $j=2$,
$h=3$, and $k=4$.
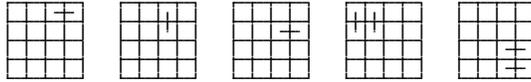
	\[  \beginpicture
		
	\setcoordinatesystem units <0.25cm,0.25cm>
	\setplotarea x from 10 to 20, y from 2 to 5
	
	\setlinear
	
	\plot 0 0 0 4 /
	\plot 1 0 1 4 /
	\plot 2 0 2 4 /
	\plot 3 0 3 4 /
	\plot 4 0 4 4 /
	
	\plot 0 0 4 0 /
	\plot 0 1 4 1 /
	\plot 0 2 4 2 /
	\plot 0 3 4 3 /
	\plot 0 4 4 4 /
	
	\plot 2.5 3.5 3.5 3.5 /
	
	\plot 6 0 6 4 /
	\plot 7 0 7 4 /
	\plot 8 0 8 4 /
	\plot 9 0 9 4 /
	\plot 10 0 10 4 /
	
	\plot 6 0 10 0 /
	\plot 6 1 10 1 /
	\plot 6 2 10 2 /
	\plot 6 3 10 3 /
	\plot 6 4 10 4 /
	
	\plot 8.5 2.5 8.5 3.5 /
	
	\plot 12 0 12 4 /
	\plot 13 0 13 4 /
	\plot 14 0 14 4 /
	\plot 15 0 15 4 /
	\plot 16 0 16 4 /
	
	\plot 12 0 16 0 /
	\plot 12 1 16 1 /
	\plot 12 2 16 2 /
	\plot 12 3 16 3 /
	\plot 12 4 16 4 /
		
	\plot  14.5 2.5 15.5 2.5 /
	
	\plot 18 0 18 4 /
	\plot 19 0 19 4 /
	\plot 20 0 20 4 /
	\plot 21 0 21 4 /
	\plot 22 0 22 4 /
	
	\plot 18 0 22 0 /
	\plot 18 1 22 1 /
	\plot 18 2 22 2 /
	\plot 18 3 22 3 /
	\plot 18 4 22 4 /
		
	\plot 18.5 2.5 18.5 3.5 /
	\plot 19.5 2.5 19.5 3.5 /
	
	\plot 24 0 24 4 /
	\plot 25 0 25 4 /
	\plot 26 0 26 4 /
	\plot 27 0 27 4 /
	\plot 28 0 28 4 /
	
	\plot 24 0 28 0 /
	\plot 24 1 28 1 /
	\plot 24 2 28 2 /
	\plot 24 3 28 3 /
	\plot 24 4 28 4 /
		
	\plot 26.5 0.5 27.5 0.5 /
	\plot 26.5 1.5 27.5 1.5 /
		
	\put {The five permutations in $S([4]^2)$ indexing the generators of $K_{i,j}$} at 15 -2
	\put{for $i=1, j=2, h=3, k=4$} at 15 -4
	
	\endpicture \]

\begin{theorem}
	\label{S4xZ2xZ2alt}
	For $i,j \in [4]$, $i \neq j$, $K_{i,j}$
	is contained in $\lambda^{-1}(\P_4^2)$ and is isomorphic to $S_4 \times \ZZ_2^{\times 2}$. The kernel of the canonical projection $\pi : K_{i,j} \rightarrow {\rm Out}(\O_4)$
	is trivial and the 6 groups $\pi(K_{i,j})$ are mutually distinct for 
	$1 \leq i < j \leq 4$.
\end{theorem}
\begin{proof}
	It is clear that the subgroup $N_{i,j}$ of $S([4]^2)$ generated by $u_0, v_0, w'_0$, $x_0$, and $y_0$ is isomorphic to $S_4 \times \ZZ_2 \times \ZZ_2$. Furthermore, 
	by Theorem \ref{NBCprec} (with $n=2m=4$), all the elements in $N_{i,j}$ are bicompatible with each other.
	Therefore, by Proposition \ref{isogroups}, $K_{i,j}$ is contained in $\lambda^{-1}(\P_4^2)$ and is isomorphic to $S_4\times \ZZ_2 \times \ZZ_2$. 
	
	Let now $\lambda_v \in K_{i,j} \cap ker(\pi)$. Then there are
	$u \in S_4$, $\alpha \in \langle \{ u_0,v_0,w'_0 \} \rangle$,
	$\beta \in \langle \{ x_0 \} \rangle$ and $\gamma \in \langle \{ y_0 \} \rangle$ such that $\lambda_{u \otimes u^{-1}} =
	\lambda_v = \lambda_{\alpha} \lambda_{\beta} \lambda_{\gamma} =
	\lambda_{\alpha \beta \gamma}$. 
	Hence $u \otimes u^{-1} = \alpha \beta \gamma$. 
	Let $a \in \{ h,k \}$ and $b \in \{ i,j \}$. Then, since
	$(a,b)$ is a fixed point of $\alpha$, $\beta$, and $\gamma$, we
	have that 
	\[
	(u(a), u^{-1}(b))= (\alpha \beta \gamma)(a,b)= (a,b)
	\]
	and hence $u(a)=a$ and $u^{-1}(b)=b$. So $u=1$, which in turn implies that $v=1$.
	Therefore $K_{i,j} \cap ker(\pi) = \{ id \}$.
	
We now show that the six subgroups $\pi(K_{i,j})$ are all 
distinct for $1 \leq i < j \leq 4$. Suppose first that $\pi(K_{i,j}) = \pi(K_{i,\ell})$, where $\ell \in \{ h,k \}$.
Let $\{ m,n \} := [4] \setminus \{ i, \ell \}$.
It is easy to see that $(i,j) \otimes 1 \in N_{i,j}$, so
$\lambda_{(i,j) \otimes 1} \in K_{i,j}$ and hence 
	$\pi(\lambda_{(i,j) \otimes 1}) \in \pi(K_{i,j}) = \pi(K_{i,\ell})$. Hence there is 
	$\tau \in N_{i,\ell}$ such that 
$
	 \pi(\lambda_{(i,j) \otimes 1}) = \pi(\lambda_{\tau})
$
	and so there is $z \in S_4$ such that 
	\[
	(i,j) \otimes 1 = (z \otimes 1) \tau (1 \otimes z^{-1})
	\]
	which implies that $z^{-1} (i,j) \otimes z = \tau$. Let
$b \in \{ i, \ell \}$, then $\tau (m,b) = (m,b)$. But the
second coordinate of $(z^{-1} (i,j) \otimes z) (m,b)$ is
$z(b)$ so we conclude that $z(b)=b$ if $b \in \{ i, \ell \}$.
Now, $j \in \{ m,n \}$, so we may assume that $m \neq j$.
Then $(z^{-1} (i,j) \otimes z) (m,m) = (z^{-1}(m), z(m))$. But
$\tau$ keeps elements of the diagonal in either the same row
or the same column so we conclude that $m$ is a fixed point of $z$ and so that $z=1$. Therefore $(i,j) \otimes 1 = \tau$. But 
$((i,j) \otimes 1)(i,i) = (j,i)$ while the first coordinate
of $\tau (i,i)$ is, by our definition of $K_{i, \ell}$, either
$i$ or $\ell$.
This shows that $\pi(K_{i,j})$ and $\pi(K_{a,b})$,
are distinct if $| \{i,j \} \cap \{ a,b \}|=1$.
	
	There remains to show that $\pi(K_{i,j}) \neq \pi(K_{h,k})$.
	Suppose not, then $\pi(\lambda_{u_0}) $ $\in \pi(K_{h,k})$. 
	Hence reasoning as above we see that there are 
	$\tau \in N_{h,k}$ and $z \in S_4$ such that $(z^{-1} \otimes 1) ((i,h),(i,k)) (1 \otimes z) = \tau$.
Let $a \in [4] \setminus \{ i \}$, then 
$((z^{-1} \otimes 1) ((i,h),(i,k)) (1 \otimes z))(a,a)=(z^{-1}(a),z(a))$ but, as observed in the previous 
paragraph, all the elements of $N_{h,k}$ map elements of the diagonal in either the same row or column, so $a$ is a fixed 
point of $z$ and hence $z=1$.
Therefore $((i,h),(i,k)) = \tau$, which is a 
contradiction since any element of $N_{h,k}$ leaves column $k$
invariant.
\end{proof}

Note that if one defines $K'_{i,j}$ to be the subgroup of ${\rm Aut}(\O_4)$ generated by 
$$
\{ \lambda_{u_0}, \lambda_{v_0}, \lambda_{w'_0}, \lambda_{(i,j) \otimes 1}, \lambda_{(h,k) \otimes 1} \},
$$
then one can show that $\pi(K_{i,j}) = \pi(K'_{i,j})$.
Furthermore, one can check that the last two automorphisms generate a copy of
$\ZZ_2 \times \ZZ_2$  which acts on the copy of $S_4$ generated by 
the first three in such a way so that $K'_{i,j} \simeq \pi(K'_{i,j}) \simeq S_4 \rtimes (\ZZ_2 \times \ZZ_2)$.

\subsection{The third line}

We keep notation as in the previous sections so $i,j \in [4]$,
 $i<j$, $\{ h,k \}_< := [4] \setminus 
\{ i,j \}$
(so $\{ h,k \} := [4] \setminus \{ i,j \}$ and $h<k$) and
$$ 
w'_0 = ((i,h),(j,h)), \ x_0 = ((i,i),(j,i))((i,j),(j,j)) , \ 
w_0 = ((h,i),(k,i)) $$
$$ 
y_0 = ((h,h),(h,k))((k,h),(k,k)), \ u'_0 = ((h,i),(h,j)), \ v'_0 = ((k,i),(k,j)),
$$
$$
z_0 = ((h,j),(k,j)), \ v_0 = ((j,h),(j,k)).
$$
We also define
\[
c_0 := ((i,h),(i,k),(j,h),(j,k)), \,
c'_0 := ((i,h),(j,k),(i,k),(j,h)),
\] 
$\widetilde{M}_{i,j}$ to be the subgroup of $S([4]^2)$ generated by 
\[
\{ c_0, \, w'_0, \, x_0, \, y_0, \, u'_0 v'_0 \}
\]
and $M_{i,j}$ to be the subgroup of ${\rm Aut}(\O_4)$ generated by $\lambda_g$, $g \in \tilde{M}_{i,j}$.
Similarly we define
$\widetilde{M}'_{i,j}$ to be the subgroup of $S([4]^2)$
generated by 
$$
\{  c'_0 , \, v_0, \, x_0, \, y_0 , \, w_0 z_0 \} , 
$$
and $M'_{i,j}$ to be the subgroup of ${\rm Aut}(\O_4)$ generated by $\lambda_g$, $g \in \widetilde{M}'_{i,j}$.

Note that we could have defined in the same way $\widetilde{M}_{i,j}$ and $\widetilde{M'}_{i,j}$ for all $i,j$ with $i \neq j$, and that then $\widetilde{M}_{i,j} = \widetilde{M}_{j,i}$ and $\widetilde{M'}_{i,j} = \widetilde{M'}_{j,i}$.
A graphical representation of the generators of the groups 
$\widetilde{M}_{1,2}$ and $\widetilde{M'}_{1,2}$ is 
depicted here below.

	\[  \beginpicture
	
	\setcoordinatesystem units <0.25cm,0.25cm>
	\setplotarea x from 10 to 20, y from 2 to 5
	
	\setlinear
	
	\plot 0 0 0 4 /
	\plot 1 0 1 4 /
	\plot 2 0 2 4 /
	\plot 3 0 3 4 /
	\plot 4 0 4 4 /
	
	\plot 0 0 4 0 /
	\plot 0 1 4 1 /
	\plot 0 2 4 2 /
	\plot 0 3 4 3 /
	\plot 0 4 4 4 /

	\arrow <0.15cm> [0.1,0.5] from 2.5 3.5 to 3.5 3.5 
	\plot 2.5 2.5 3.5 2.5 /
	\plot 2.5 2.5 3.5 3.5 /
	\plot 2.5 3.5 3.5 2.5 /
	
	\plot 6 0 6 4 /
	\plot 7 0 7 4 /
	\plot 8 0 8 4 /
	\plot 9 0 9 4 /
	\plot 10 0 10 4 /
	
	\plot 6 0 10 0 /
	\plot 6 1 10 1 /
	\plot 6 2 10 2 /
	\plot 6 3 10 3 /
	\plot 6 4 10 4 /
		
	\plot 8.5 2.5 8.5 3.5 /
	\plot 12 0 12 4 /
	\plot 13 0 13 4 /
	\plot 14 0 14 4 /
	\plot 15 0 15 4 /
	\plot 16 0 16 4 /
	
	\plot 12 0 16 0 /
	\plot 12 1 16 1 /
	\plot 12 2 16 2 /
	\plot 12 3 16 3 /
	\plot 12 4 16 4 /
		
	\plot  12.5 2.5 12.5 3.5 /
	\plot  13.5 2.5 13.5 3.5 /
	
	\plot 18 0 18 4 /
	\plot 19 0 19 4 /
	\plot 20 0 20 4 /
	\plot 21 0 21 4 /
	\plot 22 0 22 4 /
	
	\plot 18 0 22 0 /
	\plot 18 1 22 1 /
	\plot 18 2 22 2 /
	\plot 18 3 22 3 /
	\plot 18 4 22 4 /
	
	\plot 20.5 0.5 21.5 0.5 /
	\plot 20.5 1.5 21.5 1.5 /
	
	\plot 24 0 24 4 /
	\plot 25 0 25 4 /
	\plot 26 0 26 4 /
	\plot 27 0 27 4 /
	\plot 28 0 28 4 /
	
	\plot 24 0 28 0 /
	\plot 24 1 28 1 /
	\plot 24 2 28 2 /
	\plot 24 3 28 3 /
	\plot 24 4 28 4 /
		
	\plot 24.5 0.5 25.5 0.5 /
	\plot 24.5 1.5 25.5 1.5 /
	
	\put {The five permutations in $S([4]^2)$ indexing the generators of $M_{i,j}$} at 15 -2
	\put{for $i=1, j=2, h=3, k=4$} at 15 -4
	
	\endpicture \]

\[  \beginpicture

\setcoordinatesystem units <0.25cm,0.25cm>
\setplotarea x from 10 to 20, y from 2 to 5

\setlinear

\plot 0 0 0 4 /
\plot 1 0 1 4 /
\plot 2 0 2 4 /
\plot 3 0 3 4 /
\plot 4 0 4 4 /

\plot 0 0 4 0 /
\plot 0 1 4 1 /
\plot 0 2 4 2 /
\plot 0 3 4 3 /
\plot 0 4 4 4 /

\arrow <0.15cm> [0.1,0.5] from 3.5 2.5 to 3.5 3.5 
\plot 3.5 3.5 2.5 2.5 /
\plot 2.5 2.5 2.5 3.5 /
\plot 2.5 3.5 3.5 2.5 /

\plot 6 0 6 4 /
\plot 7 0 7 4 /
\plot 8 0 8 4 /
\plot 9 0 9 4 /
\plot 10 0 10 4 /

\plot 6 0 10 0 /
\plot 6 1 10 1 /
\plot 6 2 10 2 /
\plot 6 3 10 3 /
\plot 6 4 10 4 /

\plot 8.5 2.5 9.5 2.5 /

\plot 12 0 12 4 /
\plot 13 0 13 4 /
\plot 14 0 14 4 /
\plot 15 0 15 4 /
\plot 16 0 16 4 /

\plot 12 0 16 0 /
\plot 12 1 16 1 /
\plot 12 2 16 2 /
\plot 12 3 16 3 /
\plot 12 4 16 4 /

\plot  12.5 2.5 12.5 3.5 /
\plot  13.5 2.5 13.5 3.5 /

\plot 18 0 18 4 /
\plot 19 0 19 4 /
\plot 20 0 20 4 /
\plot 21 0 21 4 /
\plot 22 0 22 4 /

\plot 18 0 22 0 /
\plot 18 1 22 1 /
\plot 18 2 22 2 /
\plot 18 3 22 3 /
\plot 18 4 22 4 /

\plot 20.5 0.5 21.5 0.5 /
\plot 20.5 1.5 21.5 1.5 /

\plot 24 0 24 4 /
\plot 25 0 25 4 /
\plot 26 0 26 4 /
\plot 27 0 27 4 /
\plot 28 0 28 4 /

\plot 24 0 28 0 /
\plot 24 1 28 1 /
\plot 24 2 28 2 /
\plot 24 3 28 3 /
\plot 24 4 28 4 /

\plot 24.5 0.5 24.5 1.5 /
\plot 25.5 0.5 25.5 1.5 /

\put {The five permutations in $S([4]^2)$ indexing the generators of $M'_{i,j}$} at 15 -2
\put{for $i=1, j=2, h=3, k=4$} at 15 -4

\endpicture \]

\begin{remark}
	Note that $\widetilde{M}_{i,j}$ leaves the following subsets of $[4]^2$ invariant:
	\begin{itemize}
		\item $\{i,j\} \times \{i,j\}$, $\{i,j\} \times \{h,k\}$;
		\item $\{h\} \times \{i,j\}$, $\{h\} \times \{h,k\}$;
		\item$\{k\} \times \{i,j\}$, $\{k\} \times \{h,k\}$.
	\end{itemize}
	In particular, $\widetilde{M}_{i,j}$ leaves row $h$ and row $k$ invariant, as well as the union of columns $i$ and $j$.
	Similarly, $\widetilde{M}'_{i,j}$ leaves the following subsets of $[4]^2$ invariant:
	\begin{itemize}
		\item $\{i,j\} \times \{h,k\}$, $\{h,k\} \times \{h,k\}$;
		\item $\{h,k\} \times \{i\}$, $\{h,k\} \times \{j\}$;
		\item$\{i,j\} \times \{i\}$, $\{i,j\} \times \{j\}$,
	\end{itemize}
	so $\widetilde{M}'_{i,j}$ leaves column $i$ and column $j$ invariant, as well as the union of rows $h$ and $k$.
\end{remark}

We denote by $D_8$ the dihedral group of order $8$ (the group of symmetries of the square).

\begin{theorem}
	\label{D4xZ2xZ2xZ2}
	For $i,j\in [4]$, $i<j$, $M_{i,j}$ and $M'_{i,j}$ 
	are contained in $\lambda^{-1}(\P_4^2)$ and are isomorphic to $D_8 \times \ZZ_2^{3}$.
	The kernels of the canonical projections $\pi : M_{i,j} \rightarrow {\rm Out}(\O_4)$ and 
	$\pi : M'_{i,j} \rightarrow {\rm Out}(\O_4)$
	are trivial and the 12 groups $\pi(M_{i,j})$ and $\pi(M'_{i,j})$ are mutually distinct for $i,j \in [4]$,
	$i<j$.
\end{theorem}
\begin{proof}	
	It is easy to check that $\widetilde{M}_{i,j}$
	is isomorphic to 
	$D_8 \times \ZZ_2 \times \ZZ_2 \times \ZZ_2$. 
That each one of $w'_0$, $x_0$, $y_0$, and $u'_0 v'_0$ is
bicompatible with all the other generators follows easily from 
Propositions \ref{horizontal_same_columns},
\ref{horizontal_same_columns2_cycles},
\ref{vertical_same_rows2}, and \ref{vertical_same_rows}. 
That $c_0$ is (bi)compatible with itself (i.e., that it is
stable of rank 1) follows from \cite[Prop. 5.15]{BC}.
Therefore, by Proposition \ref{isogroups}, $M_{i,j}$ 
	is contained in $\lambda^{-1}(\P_4^2)$  and is isomorphic to
	$D_8 \times \ZZ_2 \times \ZZ_2 \times \ZZ_2$.
	
	Next we show that none of the elements of $M_{i,j}$ different from the identity is inner. Indeed, it is clear that every element of $\widetilde{M}_{i,j}$ can be written in the form $uvwzt$ where 
	$u \in \langle \{ u'_0 v'_0 \} \rangle$,
	$v \in \langle \{ x_0 \} \rangle$,
	$w \in \langle \{ y_0 \} \rangle$,
	$z \in \langle \{ c_0 \} \rangle$,
	and $t \in \langle \{ w'_0 \} \rangle$. 
	Suppose that $uvwzt = \sigma \otimes \sigma^{-1}$ for some $\sigma \in S_n$, with $u,v,w,z$, and $t$ as above. Then we have that
	\[
	(\sigma(i),\sigma^{-1}(i))=(\sigma \otimes \sigma^{-1})(i,i)=(uvwzt)(i,i)=\begin{cases}
		(i,i), \mbox{ if } v=1, \\
		(j,i), \mbox{ if } v \neq 1,
	\end{cases}
	\]
	so $v=1$ and $\sigma(i)=i$. Therefore,
	$(\sigma(j),\sigma^{-1}(j))=(\sigma \otimes \sigma^{-1})(j,j)=(uwzt)$ $(j,j)=(j,j)$ so $\sigma(j)=j$, and also
	\[
	(\sigma(h),\sigma^{-1}(h))=(\sigma \otimes \sigma^{-1})(h,h)=(uwzt)(h,h)=
	\begin{cases}
		(h,h), \mbox{ if } w=1, \\
		(h,k), \mbox{ if } w \neq 1,
	\end{cases}
	\]
	so $w=1$ and $\sigma(h)=h$. Hence $\sigma=1$ and $uzt=1$ which
	implies that $u=z=t=1$, as claimed.
	
Noting that the permutations which
	index the generators of $M'_{i,j}$ are the antitransposed of
	those which index the generators of $M_{i,j}$, using 
Proposition \ref{at-bicompatible}, 
and the fact that $^a (\sigma \otimes \sigma^{-1})=(w_0 \sigma^{-1} w_0 \otimes w_0 \sigma w_0)
	= (w_0 \sigma^{-1} w_0) \otimes (w_0 \sigma^{-1} w_0)^{-1}$
	for all $\sigma \in S_n$, where $w_0=n \ldots 3 2 1$ is the longest permutation is $S_n$ (see \ref{defatr}) shows that the corresponding statements 
	also hold for $M'_{i,j}$.
	
	\medskip
	Finally, we prove that the images in ${\rm Out}(\O_4)$  of the twelve groups $M_{i,j}$, $M'_{i,j}$, $i,j \in [4]$, $i<j$, are different from each other. 
	
By definition $\pi(\lambda_{w'_0}) \in \pi(M_{i,j})$. 
	Suppose first that 
	$\pi(\lambda_{w'_0}) \in \pi(M_{i,b})$ where $b \neq j$. Let $\{a,j\} :=
	[4] \setminus \{ i,b\}$. Then there are $\sigma \in S_4$ and $g \in \widetilde{M}_{i,b}$ such that ${w'_0}= (\sigma \otimes 1) g 
	(1 \otimes \sigma^{-1})$. Since every element of $\widetilde{M}_{i,b}$ 
	leaves row $j$ invariant and $(j,j)$ is a fixed point of ${w'_0}$ 
	we deduce that
	$(\sigma \otimes 1) g (1 \otimes \sigma^{-1})(j,j)$ is
	in row $\sigma(j)$ and therefore that $\sigma(j)=j$. 
	Similarly we conclude that $\sigma(a)=a$. Furthermore, if
	$\sigma=(i,b)$ then we have that 
	$(\sigma \otimes 1) g (1 \otimes \sigma^{-1})(i,i)
	=((\sigma \otimes 1) g )(i,b)$ which is in column $b$,
	contradicting the fact that $(i,i)$ is a fixed point of ${w'_0}$.
	Therefore $\sigma=1$ so ${w'_0}=g$ and hence ${w'_0} \in \widetilde{M}_{i,b}$ which is a contradiction since all the elements of $\widetilde{M}_{i,b}$ leave row $j$ invariant. This shows that
	$\pi(M_{i,j}) \neq \pi(M_{i,b})$ if $b \neq j$.
	Similarly $\pi(M_{i,j}) \neq \pi(M_{a,j})$ if $a \neq i$.
	
	Suppose now that $\pi(\lambda_{w'_0}) \in \pi(M_{h,k})$. Then, as above, there are $\sigma \in S_4$ and $g \in \widetilde{M}_{h,k}$ such that $w'_0= (\sigma \otimes 1) g (1 \otimes \sigma^{-1})$. Since $\widetilde{M}_{h,k}$
	leaves rows $i$ and $j$ invariant and $(i,i)$ and $(j,j)$ are
	fixed points of $w'_0$ we conclude as above that $\sigma(i)=i$ and
	$\sigma(j)=j$. Finally, if $\sigma=(h,k)$ then 
	$(\sigma \otimes 1) g (1 \otimes \sigma^{-1})(k,k)
	=((\sigma \otimes 1) g)(k,h)$ is in column $h$
	which contradicts the fact that $(k,k)$ is a fixed 
	point of $w'_0$. So $\sigma=1$ and hence $w'_0 \in \widetilde{M}_{h,k}$ which is
	impossible since $\widetilde{M}_{h,k}$ leaves rows $i$ and $j$ invariant.
	This shows that $\pi(M_{i,j}) \neq \pi(M_{h,k})$. Therefore the six groups $\pi(M_{i,j})$ are all distinct for $i,j \in [4]$, $i<j$. 
	
	Similarly, $\pi(\lambda_{v_0}) \in \pi(M'_{i,j})$. Suppose first that 
	$\pi(\lambda_{v_0}) \in \pi(M'_{i,b})$ where $b \neq j$. Let $\{a,j\} :=
	[4] \setminus \{ i,b\}$. Then there are $\sigma \in S_4$ and $g \in \widetilde{M}'_{i,b}$ such that $v_0= (\sigma \otimes 1) g 
	(1 \otimes \sigma^{-1})$. Since every element of $\widetilde{M}'_{i,b}$ 
	leaves columns $i$ and $b$ invariant and $(i,i)$ and $(b,b)$
	are fixed points of $v_0$ reasoning as above we deduce that
	if $\sigma(i) \neq i$ then $(\sigma \otimes 1) g (1 \otimes \sigma^{-1})(i,i)$ lies in a column $\neq i$, which is a
	contradiction, and similarly
	for $(b,b)$. So $\sigma(i)=i$ and $\sigma(b)=b$. If $\sigma=(a,j)$
	then $(\sigma \otimes 1) g 
	(1 \otimes \sigma^{-1})(a,a)=((\sigma \otimes 1) g)(a,j)$
	is in row $j$ (because every element of $\widetilde{M}'_{i,b}$
	maps $(a,j)$ in row $a$) and this is a contradiction since $(a,a)$ is a 
	fixed point of $v_0$. Hence $\sigma=1$ and so $v_0 \in \widetilde{M}'_{i,b}$
	but every element of $\widetilde{M}'_{i,b}$ leaves column $b$ invariant
	and $b \in \{ h,k \}$ so either column $h$ or column $k$ are
	left invariant by $v_0$ which is a contradiction. This shows that
	$\pi(M'_{i,j}) \neq \pi(M'_{i,b})$ if $b \neq j$. In the same way
	one shows that  $\pi(M'_{i,j}) \neq \pi(M'_{a,j})$ if $a \neq i$.
	Finally, suppose that $\pi(\lambda_{v_0}) \in \pi(M'_{h,k})$, so that there 
	are $\sigma \in S_4$ and $g \in \widetilde{M}'_{h,k}$ such that $(\sigma^{-1} \otimes 1) v_0 (1 \otimes \sigma)= g$. 
Since $v_0$ leaves every row invariant we conclude that 
$(\sigma^{-1} \otimes 1) v_0 (1 \otimes \sigma)$ maps row $x$ to
row $\sigma^{-1}(x)$ for all $x \in [4]$. But every element of 
$\widetilde{M}'_{h,k}$ leaves $\{ h,k \} \times [4]$ invariant
and hence also $\{ i,j \} \times [4]$. Therefore $\sigma(\{ h,k \})=\{ h,k \}$ and  $\sigma(\{ i,j \})=\{ i,j \}$. If $\sigma(h)=k$ then we would have that $\sigma(k)=h$ and therefore that
\[
g(h,k)=(\sigma^{-1} \otimes 1) v_0 (1 \otimes \sigma)(h,k)=(k,h)
\] 
which is a contradiction since no element of $\widetilde{M}'_{h,k}$
has this property. Therefore $\sigma(h)=h$ and $\sigma(k)=k$.
Similarly one shows that $\sigma(i)=i$ and $\sigma(j)=j$.
So $\sigma=1$. 
Hence $v_0 \in \widetilde{M}'_{h,k}$,
	which is impossible. This shows that the 6 subgroups $\pi(M'_{i,j})$ are all distinct. 
		
	We now show that $\pi(M_{i,j}) \neq \pi(M'_{a,b})$ for all
	$i, j, a,b \in [4]$, $i \neq j$, $a \neq b$. 
	Suppose that $\pi(\lambda_{w'_0}) \in \pi(M'_{a,b})$, i.e. there are $g \in \widetilde{M}'_{a,b}$ and $\sigma \in S_4$ such that $w'_0 = (\sigma \otimes 1) g (1 \otimes \sigma^{-1})$. Since $w'_0$ and $g$ both preserve column $a$ and column $b$, we conclude that $\sigma(a)=a$ and $\sigma(b)=b$. Therefore $\sigma = 1$ or $\sigma = (\alpha,\beta)$ where $\{ \alpha, \beta \}:=[4] \setminus \{ a,b \}$. However, if $\sigma=(\alpha,\beta)$ 
	then, since every element of $\widetilde{M'}_{a,b}$ 
maps $(\alpha, \beta)$ to row $\alpha$, we have that
	 $(\alpha,\alpha) = w'_0(\alpha,\alpha) = (\sigma \otimes 1) g (1 \otimes \sigma^{-1})(\alpha,\alpha) = (\sigma \otimes 1) g (\alpha,\beta) = (\sigma \otimes 1)(\alpha,*) = (\beta,*)$ (where $*$ is some 
	element of $[4]$). This shows that $\sigma=1$. So $w'_0 \in \widetilde{M}'_{a,b}$ and this is impossible since 
	$w'_0 \notin \bigcup_{1 \leq c<d \leq 4} \widetilde{M}'_{c,d}$.
	This is easy to show, the only case requiring a bit of attention being that $w'_0 \notin \widetilde{M}'_{i,j}$. If $w'_0 \in \widetilde{M}'_{i,j}$ then
	$w'_0$ must lie in the subgroup of $\widetilde{M}'_{i,j}$ generated by $c'_0$ 
and $v_0$. 
But this is impossible, noticing that this subgroup is isomorphic to the group of symmetries of the square with vertices $(i,h)$,$(j,k)$,$(i,k)$ and $(j,h)$ ordered clockwise, and $w'_0$ is not a symmetry of this square. 
\end{proof}

\subsection{The fourth line}

We keep notation as in the previous sections so $i,j \in [4]$,
$i<j$, $\{ h,k \}_< := [4] \setminus \{ i,j \}$, and
$$ 
u_0=((i,h),(i,k)), \ v_0 = ((j,h),(j,k)),
$$
$$
w_0 = ((h,i),(k,i)), \ z_0 = ((h,j),(k,j)).
$$
For $i,j \in [4]$, $i \neq j$, we let $G_{i,j}$ be the group generated by 
\[
\{ \lambda_{u_0}, \lambda_{v_0}, \lambda_{w_0}, \lambda_{z_0}, 
\lambda_{(h,k) \otimes 1}, \lambda_{(i,j) \otimes 1} \}.
\]
in ${\rm Aut}(\O_4)$ (so $G_{i,j}=G_{j,i}$, and $G_{i,j}$ is
the subgroup generated by $\lambda(R_{i,j}) \cup \lambda(R_{j,i}) 
\cup \lambda(C_{i,j}) \cup \lambda(C_{j,i}) \cup  \lambda(S_{i,j})
\cup \lambda(S_{h,k})$).
For the reader's convenience we include a graphical representation of the six permutations $u_0$, $v_0$, $w_0$, $z_0$, 
$(h,k) \otimes 1$, and $(i,j) \otimes 1$ in $S([4]^2)$ for $i=1$
and $j=2$.
	\[  \beginpicture
		
	\setcoordinatesystem units <0.25cm,0.25cm>
	\setplotarea x from 10 to 20, y from 2 to 5
	
	\setlinear
	
	\plot 0 0 0 4 /
	\plot 1 0 1 4 /
	\plot 2 0 2 4 /
	\plot 3 0 3 4 /
	\plot 4 0 4 4 /
	
	\plot 0 0 4 0 /
	\plot 0 1 4 1 /
	\plot 0 2 4 2 /
	\plot 0 3 4 3 /
	\plot 0 4 4 4 /
		
	\plot 2.5 3.5 3.5 3.5 /
	\plot 6 0 6 4 /
	\plot 7 0 7 4 /
	\plot 8 0 8 4 /
	\plot 9 0 9 4 /
	\plot 10 0 10 4 /
	
	\plot 6 0 10 0 /
	\plot 6 1 10 1 /
	\plot 6 2 10 2 /
	\plot 6 3 10 3 /
	\plot 6 4 10 4 /
		
	\plot 8.5 2.5 9.5 2.5 /
	\plot 12 0 12 4 /
	\plot 13 0 13 4 /
	\plot 14 0 14 4 /
	\plot 15 0 15 4 /
	\plot 16 0 16 4 /
	
	\plot 12 0 16 0 /
	\plot 12 1 16 1 /
	\plot 12 2 16 2 /
	\plot 12 3 16 3 /
	\plot 12 4 16 4 /
		
	\plot  12.5 0.5 12.5 1.5 /
	
	\plot 18 0 18 4 /
	\plot 19 0 19 4 /
	\plot 20 0 20 4 /
	\plot 21 0 21 4 /
	\plot 22 0 22 4 /
	
	\plot 18 0 22 0 /
	\plot 18 1 22 1 /
	\plot 18 2 22 2 /
	\plot 18 3 22 3 /
	\plot 18 4 22 4 /
		
	\plot  19.5 0.5 19.5 1.5 /	
	
	\plot 24 0 24 4 /
	\plot 25 0 25 4 /
	\plot 26 0 26 4 /
	\plot 27 0 27 4 /
	\plot 28 0 28 4 /
	
	\plot 24 0 28 0 /
	\plot 24 1 28 1 /
	\plot 24 2 28 2 /
	\plot 24 3 28 3 /
	\plot 24 4 28 4 /
		
	\plot 24.5 0.5 24.5 1.5 /
	\plot 25.5 0.5 25.5 1.5 /
	\plot 26.5 0.5 26.5 1.5 /
	\plot 27.5 0.5 27.5 1.5 /
	
	\plot 30 0 30 4 /
	\plot 31 0 31 4 /
	\plot 32 0 32 4 /
	\plot 33 0 33 4 /
	\plot 34 0 34 4 /
	
	\plot 30 0 34 0 /
	\plot 30 1 34 1 /
	\plot 30 2 34 2 /
	\plot 30 3 34 3 /
	\plot 30 4 34 4 /
	
	\plot 30.5 2.5 30.5 3.5 /
	\plot 31.5 2.5 31.5 3.5 /
	\plot 32.5 2.5 32.5 3.5 /
	\plot 33.5 2.5 33.5 3.5 /
	
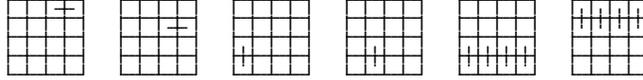
	\put {The six permutations in $S([4]^2)$ indexing the generators of $G_{i,j}$} at 17 -2
	\put{for $i=1, j=2$} at 17 -4
	\endpicture \]

\begin{theorem}
	For $i,j \in [4]$, $i \neq j$, $G_{i,j}$
	is contained in $\lambda^{-1}(\P_4^2)$ and is isomorphic to $\ZZ_2 \times \big((\ZZ_2^{\times 4}) \rtimes \ZZ_2\big)$.
	The kernel of the canonical projection $\pi : G_{i,j} \rightarrow {\rm Out}(\O_4)$
	is trivial and the 6 groups $\pi(G_{i,j})$ are mutually distinct.
\end{theorem}
\begin{proof}
It is clear that the subgroup of $S([4]^2)$ generated
by $\{ u_0, v_0, w_0, z_0 \}$ is isomorphic to
$\ZZ_2^{\times 4}$. Furthermore, it is easy to check that
all these generatore are bicompatible with each other by 
Propositions \ref{horizontal_same_columns},
\ref{horizontal_same_columns2_cycles},
\ref{vertical_same_rows2}, and \ref{vertical_same_rows}.
Therefore, by Proposition \ref{isogroups}, the group generated by 
$ \{ \lambda_{u_0}, \lambda_{v_0}, \lambda_{w_0}, \lambda_{z_0}
\}$ in ${\rm Aut}(\O_4)$
is contained in $\lambda^{-1}(\P_4^2)$ and is also isomorphic to 
$\ZZ_2^{\times 4}$.

Furthermore, it is easy to check that 
$$
\lambda_{(i,j) \otimes 1} \, \lambda_{u_0} = 
\lambda_{v_0} \, \lambda_{(i,j)\otimes 1}, \; \;
\lambda_{(i,j) \otimes 1} \, \lambda_{w_0} = 
\lambda_{z_0} \, \lambda_{(i,j)\otimes 1}
$$
and thus
$$
\lambda_{(i,j) \otimes 1} \, \lambda_{v_0} = 
\lambda_{u_0} \, \lambda_{(i,j)\otimes 1}, \; \;
\lambda_{(i,j) \otimes 1} \, \lambda_{z_0} = 
\lambda_{w_0} \, \lambda_{(i,j)\otimes 1}
$$ 
as well. It readily follows that
the group generated by 
$ \{ \lambda_{u_0}, \lambda_{v_0}, \lambda_{w_0}, \lambda_{z_0},
\lambda_{(i,j) \otimes 1} \}$
in ${\rm Aut}(\O_4)$
 is isomorphic to $\big(\ZZ_2^{\times 4} \big) \rtimes \ZZ_2$.
(The action of $\ZZ_2$ switches the two horizontal generators with each other and similarly for the vertical ones.)

Finally, one can similarly verify that $(h,k) \otimes 1$ 
commutes 
(and is bicompatible) with $u_0, v_0, w_0, z_0$ and 
$(i,j) \otimes 1$, hence $G_{i,j}$ is contained in $\lambda^{-1}(\P_4^2)$, and is isomorphic to 
$\ZZ_2 \times \big((\ZZ_2^{\times 4}) \rtimes \ZZ_2\big)$.

\smallskip
By the above, every element in $G_{i,j}$ can be written as $\lambda_{uvwz (\sigma \sigma' \otimes 1)}$
for some $u \in \langle u_0 \rangle$, $v \in \langle v_0 \rangle$, $w \in \langle w_0 \rangle$, $z \in \langle z_0 \rangle$, $\sigma \in \langle (i,j) \rangle$, and $\sigma' \in \langle (h,k) \rangle$.
We are now ready to show that the image of $G_{i,j}$ in ${\rm Out(\O_4)}$ is isomorphic to $G_{i,j}$ itself.
To this aim, it is enough to verify that none of the nontrivial the elements in $G_{i,j}$ is inner, i.e. of the form $\lambda_{p \otimes p^{-1}}$ for some $p \in S_4 \setminus \{ 1 \}$.
Let $u,v,w,z,\sigma,\sigma'$ and $p$ be as above and such that 
$uvwz (\sigma \sigma' \otimes 1) = p \otimes p^{-1}$. If $\sigma \neq 1$ then 
we have that 
$(i,i)=uvwz (\sigma \sigma' \otimes 1)(j,i)=(p(j),p^{-1}(i))$, which is impossible.
Similarly if $\sigma' \neq 1$ then 
$(k,k)=uvwz (\sigma \sigma' \otimes 1)(h,k)=(p(h),p^{-1}(k))$. On the other hand, 
if $\sigma=\sigma'=1$ then we have that 
$(a,a)=uvwz(a,a)=(p(a),p^{-1}(a))$ for all $a \in [4]$ so $p=1$, which is also
a contradiction.

\smallskip
It remains to show that the six subgroups of ${\rm Out}(\O_4)$ associated to different subsets of $[4]$ of size $2$ are mutually distinct.
Let $\{ a,b \} \in \binom{[4]}{2}$ be such that $\{ a,b \} \neq \{ i,j \}$.

Suppose first that $|\{ a,b \} \cap \{ i,j \}|=1$, say that $i=a$. 
 Then 
$\pi(\lambda_{(i,b) \otimes 1}) \in \pi(G_{i,b})$. We claim that $\pi(\lambda_{(i,b) \otimes 1}) \notin \pi(G_{i,j})$. Indeed, suppose the contrary. Then there are 
$u \in \langle u_0 \rangle$, $v \in \langle v_0 \rangle$, $w \in \langle w_0 \rangle$, $z \in \langle z_0 \rangle$, $\sigma \in \langle (i,j) \rangle$, and $\sigma' \in \langle (h,k) \rangle$,  such that 
$\pi(\lambda_{(i,b) \otimes 1}) = 
\pi(\lambda_{uvwz (\sigma \sigma' \otimes 1)})$. Hence there is $p \in S_4$ such that $uvwz ({\sigma \sigma' \otimes 1})=(p \otimes 1) ((i,b) \otimes 1) (1 \otimes p^{-1})=(p \, (i,b)) \otimes p^{-1}$. 
Let 
$\{ \ell \}:= [4] \setminus \{ i,j,b \}$. If $\sigma' \neq 1$  then $(\sigma'(\ell),\ell)=(uvwz (\sigma \otimes 1) (\sigma' \otimes 1))(\ell,\ell)=(p(\ell),p^{-1}(\ell))$ since 
$\{\sigma'(\ell),\ell \}=\{ h,k \}$, which is a contradiction. On the other hand, if $\sigma' = 1$ then $(b,b)=(uvwz (\sigma \otimes 1))(b,b)=(p(i),p^{-1}(b))$ 
which is again a contradiction.

Suppose now that $\{ a,b \} \cap \{ i,j \}= \emptyset$, say $a=h$ and $b=k$. 
Then $\pi(\lambda_{((k,i),(k,j))})$ $\in \pi(G_{a,b})$. We claim that $\pi(\lambda_{((k,i),(k,j))}) \notin \pi(G_{i,j})$. Indeed, suppose the contrary. Then reasoning as above we conclude that there are 
$u \in \langle u_0 \rangle$, $v \in \langle v_0 \rangle$, $w \in \langle w_0 \rangle$, $z \in \langle z_0 \rangle$, $\sigma \in \langle (i,j) \rangle$, $\sigma' \in \langle (h,k) \rangle$, 
and $p \in S_4$ such that 
$uvwz ({\sigma \otimes 1}) ({\sigma' \otimes 1})=(p \otimes 1) ((k,i),(k,j)) (1 \otimes p^{-1})$. This easily implies that 
$p \neq 1$. 
Note that $(p \otimes 1) ((k,i),(k,j)) (1 \otimes p^{-1})$
maps row $x$ in row $p(x)$ for all $x \in [4]$, so 
$p(\{ i,j \})=\{ i,j \}$ and $p(\{ h,k \})=\{ h,k \}$.
Suppose first that 
either $\sigma'=1$ and $w=1$, or $\sigma' \neq 1$ and $w \neq 1$. Then $(h,i)$
and $(k,i)$ are both fixed points of
$uvwz ({\sigma \otimes 1}) ({\sigma' \otimes 1})$ 
so $p(h)=h$ and $p(k)=k$, and hence $p=(i,j)$. 
But then $(p \otimes 1) ((k,i),(k,j)) (1 \otimes p^{-1})(h,i)=(k,j)$ which is again a contradiction.
Finally, assume that 
either $\sigma' \neq 1$ and $w=1$, or $\sigma' = 1$ and $w \neq 1$. Then 
$uvwz ({\sigma \otimes 1}) ({\sigma' \otimes 1})$ maps $(h,i)$ to $(k,i)$ and $(k,i)$ to $(h,i)$ 
so $p(h)=k$ and $p(k)=h$.
But then $p(i)=j$ (else $(p \otimes 1) ((k,i),(k,j)) (1 \otimes p^{-1})(k,i)=(h,j)$ which is impossible) so $(p \otimes 1) ((k,i),(k,j)) (1 \otimes p^{-1})(h,i)=(k,j)$ which is also a contradiction.
\end{proof}

\subsection{The fifth line}

Let $i \in [n]$ and $(b_1, \ldots , b_r) \in S_n$ be an $r$-cycle
such that $i \notin \{ b_1, \ldots , b_r \}$ (in particular, $2 \leq r \leq n-1$). 
We define $a,b,c \in S([n]^2)$ by
$
a:=((i,b_1), \ldots , (i,b_r))
$,
$
b:=((b_1,i), \ldots , (b_r,i))
$,
\[
c:=\prod_{x \neq i} ((x,b_1), \ldots , (x,b_r)),
\]
and $t:=(b_1,b_r) (b_2,b_{r-1}) \cdots$ (so $t \in S_n$ and $t$ 
is the product of $\lfloor r/2 \rfloor$ $2$-cycles).
We define $L_{i}(b_1, \ldots , b_r)$ to be the subgroup of ${\rm Aut}(\O_n)$ generated by 
$$
\{ \lambda_{a}, \lambda_{b}, \lambda_{c}, \lambda_{t \otimes 1} \}.
$$
Note that $t$ depends on a choice of writing the $r$-cycle
$(b_1, \ldots , b_r)$. However, it can be checked that if we
pick another way of writing the same cycle and take the
corresponding $t$ then the images in ${\rm Out}(\O_n)$
of the two resulting subgroups are the same (see Remark \ref{t=t'}).

\begin{theorem}
	\label{ZrxZrxZrsdZ2}
For $i\in [n]$, and $(b_1, \ldots , b_r) \in S_n$, 
$i \notin \{ b_1, \ldots , b_r \}$, we have that 
$L_{i}(b_1, \ldots , b_r)$ is contained in $\lambda^{-1}(\P_n^2)$, 
and is isomorphic to $(\ZZ_{r}^{ \times 3}) \rtimes \ZZ_2$.
Furthermore,
if $r \geq 3$, the kernel of the canonical projection $\pi : 
L_{i}(b_1, \ldots , b_r) \rightarrow {\rm Out}(\O_n)$
is trivial.
\end{theorem}
\begin{proof}
It is clear that the subgroup $\widetilde{L}_{i}(b_1, \ldots , b_r)$ of $S([n]^2)$ generated by $\{ a,b,c \}$ is isomorphic to $\ZZ_r^{\times 3}$. 
	
	Furthermore, these 3 generators
	are mutually bicompatible. 
	More specifically, by Proposition \ref{horizontal_same_columns}, 
	$a$ is compatible with a permutation $v \in S([n]^2)$ if and only if $v$ leaves $[n] \times ([n] \setminus \{ i \})$
	invariant and this is easily seen to be the case for all 
	these three
	permutations. On the other hand, by Proposition \ref{horizontal_same_columns2_cycles},
a permutation $v \in S([n]^2)$ is compatible with $a$
if and only if it permutes the elements of $\{ b_1, \ldots , b_r\}
	\times [n]$ in the way prescribed by Proposition 
	\ref{horizontal_same_columns2_cycles}, and this holds
	by taking $\sigma=1$, $t(i)=1$,
	and $t(x)=0$ if $x \neq i$ if $v=b$, and 
$\sigma=(b_1, \ldots , b_r)$ and $t(x)=0$ 
	for all $x \in [n]$ if $v=c$.
	Finally, by Proposition \ref{vertical_same_rows2},
$b$ is compatible with itself by taking $\sigma=1$ and $t(x)=0$
for all $x \in [n]$, and 
with $c$ by taking $\sigma=1$ and $t(x):=1$ for all $x \in [n] \setminus \{ i \}$, and $t(i)=0$, 
while the compatibility of $c$ with $b$ and with itself follows easily from Proposition \ref{horizontal_same_columns}.

	Therefore, by Proposition \ref{isogroups}, the subgroup  generated by $\lambda_{a}$, $\lambda_{b}$ and $\lambda_{c}$
	in ${\rm Aut}(\O_4)$ 
	is contained in $\lambda^{-1}(\P_4^2)$  and is isomorphic to $\ZZ_r^{\times 3}$.
	
	Next, we claim that 
	\[
	\lambda_{t \otimes 1} \lambda_x = \lambda^{-1}_x \lambda_{t \otimes 1}
	\]
	for $x \in \{a,b,c\}$
	(notice that $\lambda_{x^{-1}} = \lambda^{-1}_x$ as $x$ is compatible with itself).
	Indeed, it is not difficult to check that $(t \otimes t)x(t \otimes t) = x^{-1}$ and thus, by the composition rules of endomorphisms,
	$$
	\lambda_{t \otimes 1} \lambda_x = \lambda_{(t \otimes 1)(1 \otimes t)x (1 \otimes t)} = \lambda_{(t \otimes t)x (t \otimes t)(t \otimes 1)} \\
	= \lambda_{x^{-1}(t \otimes 1)} = \lambda_{x^{-1}}\lambda_{t \otimes 1} \ . 
	$$
	
	As a consequence, each element in $L_i(b_1,\ldots,b_r)$ can be written in the form $\lambda_u \lambda_v \lambda_w \lambda_z = \lambda_{uvwz}$,
	with $u \in \langle a \rangle$, $v \in \langle b \rangle$, $w \in \langle c \rangle$, $z \in \langle t \otimes 1 \rangle$,
	so that
	$L_i (b_1,\ldots,b_r)$ has $2 r^3$ elements (notice that $t \otimes 1$ does not belong to $\widetilde{L}_{i}(b_1, \ldots , b_r)$) and it has the structure of a semidirect product $(\ZZ_r \times \ZZ_r \times \ZZ_r) \rtimes \ZZ_2$.

\medskip	
Next we show that for $r \geq 3$ none of the elements of $L_i(b_1,\ldots,b_r)$
different from the identity is inner. Suppose that there are 
$u,v,w$, and $z$ as in the previous paragraph such that $uvwz = \sigma \otimes \sigma^{-1}$ for some $\sigma \in S_n$. Let $x \in
[n] \setminus \{ i \}$. If $u \neq 1$ then we have that
\[
(i,(b_1, \ldots , b_r)(x))=(uvwz)(i,x)=
(\sigma \otimes \sigma^{-1})(i,x)=(\sigma(i),\sigma^{-1}(x))
\]
from which we conclude that $\sigma=(b_r, \ldots , b_1)$.
Let now $k \in [r]$. Then $(\sigma \otimes \sigma^{-1})(b_2,b_k)$ is in row $b_1$ while $(uvwz)(b_2,b_k)$ is in row $b_2$ if $z=1$
and in row $b_{r-1}$ if $z \neq 1$, a contradiction. Hence $u=1$. Then we similarly have that
$(i,x)=(uvwz)(i,x)=(\sigma(i),\sigma^{-1}(x))$
from which we deduce that $\sigma=1$, which in turn easily implies 
that $u=v=w=z=1$. 
\end{proof}

Note that, for $r=2$, $t  = (b_1,b_2)$ and $ac = 1 \otimes (b_1,b_2)$ so that $(t \otimes 1)ac = (b_1,b_2) \otimes (b_1,b_2)$ is inner.
\medskip

Regarding the distinctness of the subgroups studied in the 
previous theorem, we have the following result. 
\begin{theorem}
	\label{ZrxZrxZrsdZ2_distinct}
For $i,j\in [n]$, and $(b_1, \ldots , b_r),(c_1, \ldots , c_s) \in S_n$, $i \notin \{ b_1, \ldots , b_r \}$, $j \notin \{ c_1, \ldots , c_s \}$, we have that $\pi(L_i(b_1, \ldots , b_r)) \neq \pi(L_j(c_1, \ldots , c_s))$ if either $r,s \geq 3$ and $r \neq s$, or $i \in \{ c_1, \ldots , c_s \}$.
\end{theorem}
\begin{proof}
The first statement follows immediately from the fact that,
by Theorem \ref{ZrxZrxZrsdZ2}, 
$|\pi(L_i(b_1, \ldots , b_r))|=2 r^3$ if $r \geq 3$.
We show that the images in ${\rm Out}(\O_n)$  of the  groups $L_i(b_1, \ldots , b_r)$ and $L_j(c_1, \ldots , c_s)$ 
are different if
$i \in \{ c_1, \ldots , c_s \}$. Note that this implies that
$i \neq j$. Suppose not. Then $\pi(\lambda_{((j,c_1), \ldots 
	,(j,c_s))}) \in \pi(L_i(b_1, \ldots , b_r))$.
Hence, as observed in the proof of the last result, there are 
$u \in \langle a \rangle$, $v \in \langle b \rangle$, $w \in \langle c \rangle$, and $z \in \langle t \otimes 1 \rangle$
such that $\pi(\lambda_{((j,c_1), \ldots ,(j,c_s))}) = \pi (\lambda_{uvwz})$.
Therefore there is $\sigma \in S_4$ such that 
\begin{equation}
	\label{equal_in_Out}
	((j,c_1), \ldots ,(j,c_s)) = (\sigma \otimes 1) \, uvwz \, (1 \otimes \sigma^{-1}).
\end{equation}
Applying both members of this equation to $(i,i)$
and observing that $u,v,w$, and $z$ all preserve the $i$-th row we conclude that $i=\sigma(i)$. Hence, since $u,v,w$, and $z$ all preserve the $i$-th column, we see that the permutation on the
RHS of (\ref{equal_in_Out}) maps $(j,i)$ to column $i$ while the
one on the LHS maps $(j,i)$ not to the $i$-th column (because $i \in\{ c_1, \ldots , c_s \}$). This proves
that $\pi(L_i(b_1, \ldots , b_r)) \neq \pi(L_j(c_1, \ldots , c_s))$, as desired.
\end{proof}

\begin{remark}
	\label{t=t'}
Let $i$, $(b_1, \ldots , b_r)$, and
$a,b,c,t$ have the same meaning as above,  
$t':=(b_1,b_{r+2})(b_3,b_{r})(b_4,b_{r-1}) \cdots$, 
$L$ be the subgroup of ${\rm Aut}(\O_n)$ generated by 
$
\{ \lambda_{a}, \lambda_{b}, \lambda_{c}, \lambda_{t \otimes 1} \},
$
and $L'$ be the one generated by $\{ \lambda_{a}, \lambda_{b},
 \lambda_{c}, \lambda_{t' \otimes 1} \}$.
We will prove that $\pi(L)=\pi(L')$. It is enough to show
that $\pi(\lambda_{t' \otimes 1}) \in \pi(L)$. Indeed,
we have that 
\begin{align*}
	\pi(\lambda_{t' \otimes 1}) & =
	\pi(\lambda_{1 \otimes t'})=
	\pi(\lambda_{1 \otimes (b_1, \ldots b_r) t
	(b_r, \ldots , b_1)})=
	\pi(\lambda_{1 \otimes (b_1, \ldots b_r)} \lambda_{1 \otimes
t} \lambda_{1 \otimes (b_1, \ldots , b_r)^{-1}}) \\
	& = \pi(\lambda_{ca}) \, 
	\pi(\lambda_{1 \otimes t}) \pi(\lambda_{(ca)^{-1}})=
	\pi(\lambda_{c}) \pi(\lambda_{a})
	\, \pi(\lambda_{t \otimes 1}) \pi((\lambda_a)^{-1})
	\pi((\lambda_c)^{-1})
\end{align*}
(where we have used the facts that $\pi(\lambda_{1 \otimes u})=\pi(\lambda_{u \otimes 1})$, that $1 \otimes u$ and $1 \otimes v$ are bicompatible, for any $u,v \in S_n$, that $ac=(b_1, \ldots , b_r)$, and that $a$ and $c$ are bicompatible)
and the elements at the end of this chain of equalities
are all in $\pi(L)$.
\end{remark}

\bigskip 

We are now ready to discuss the fifth line, by applying the above discussion to the case $n=4$ and $r=3$.
Let $i \in [4]$, and $L_{i}:=L_i(h,k,l)$ where $(h,k,l) \in S_4$ is any $3$-cycle such that 
$\{ h,k,l \} := [4] \setminus \{ i \}$ (note that, since 
$(h,k,l)=(h,l,k)^{-1}$, $L_i(h,k,l)$ only depends on $i$, and
on $t$, cfr. Remark \ref{t=t'}).
For the reader's convenience we include a graphical representation of the four permutations indexing the four generators of $L_1(3,4,2)$. 
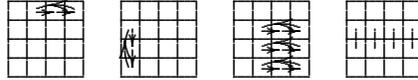
\[  
\beginpicture

\setcoordinatesystem units <0.25cm,0.25cm>
\setplotarea x from 10 to 20, y from 2 to 5

\setlinear

\plot 0 0 0 4 /
\plot 1 0 1 4 /
\plot 2 0 2 4 /
\plot 3 0 3 4 /
\plot 4 0 4 4 /

\plot 0 0 4 0 /
\plot 0 1 4 1 /
\plot 0 2 4 2 /
\plot 0 3 4 3 /
\plot 0 4 4 4 /

\arrow <0.15cm> [0.1,0.5] from 2.6 3.4 to 3.5 3.4 
\arrow <0.15cm> [0.1,0.5] from 1.5 3.4 to 2.4 3.4 

\setquadratic
\plot 3.5 3.6 2.5 3.8 1.5 3.6 /

\arrow <0.235cm> [0.2,0.6] from 2.30 3.8 to 2.20 3.8

\setlinear 

\plot 6 0 6 4 /
\plot 7 0 7 4 /
\plot 8 0 8 4 /
\plot 9 0 9 4 /
\plot 10 0 10 4 /

\plot 6 0 10 0 /
\plot 6 1 10 1 /
\plot 6 2 10 2 /
\plot 6 3 10 3 /
\plot 6 4 10 4 /

\arrow <0.15cm> [0.1,0.5] from 6.6 1.4 to 6.6 0.5 
\arrow <0.15cm> [0.1,0.5] from 6.6 2.5 to 6.6 1.6 

\setquadratic
\plot 6.4 0.5 6.2 1.5 6.4 2.5 /

\arrow <0.235cm> [0.2,0.6] from 6.2 1.8 to 6.2 1.9

\setlinear

\plot 12 0 12 4 /
\plot 13 0 13 4 /
\plot 14 0 14 4 /
\plot 15 0 15 4 /
\plot 16 0 16 4 /

\plot 12 0 16 0 /
\plot 12 1 16 1 /
\plot 12 2 16 2 /
\plot 12 3 16 3 /
\plot 12 4 16 4 /

\arrow <0.15cm> [0.1,0.5] from 14.6 2.4 to 15.5 2.4 
\arrow <0.15cm> [0.1,0.5] from 13.5 2.4 to 14.4 2.4 

\arrow <0.15cm> [0.1,0.5] from 14.6 1.4 to 15.5 1.4 
\arrow <0.15cm> [0.1,0.5] from  13.5 1.4 to 14.4 1.4 

\arrow <0.15cm> [0.1,0.5] from  14.6 0.4 to 15.5 0.4
\arrow <0.15cm> [0.1,0.5] from 13.5 0.4 to 14.4 0.4 

\setquadratic
\plot 15.5 2.6 14.5 2.8 13.5 2.6 /
\arrow <0.235cm> [0.2,0.6] from 14.4 2.8 to 14.2 2.8

\plot 15.5 1.6 14.5 1.8 13.5 1.6 /
\arrow <0.235cm> [0.2,0.6] from 14.4 1.8 to 14.2 1.8

\plot 15.5 0.6 14.5 0.8 13.5 0.6 /
\arrow <0.235cm> [0.2,0.6] from 14.4 0.8 to 14.2 0.8

\setlinear

\plot 18 0 18 4 /
\plot 19 0 19 4 /
\plot 20 0 20 4 /
\plot 21 0 21 4 /
\plot 22 0 22 4 /

\plot 18 0 22 0 /
\plot 18 1 22 1 /
\plot 18 2 22 2 /
\plot 18 3 22 3 /
\plot 18 4 22 4 /

\plot 18.5 1.5 18.5 2.5 /
\plot 19.5 1.5 19.5 2.5 /
\plot 20.5 1.5 20.5 2.5 /
\plot 21.5 1.5 21.5 2.5 /

\put {The four permutations in $S([4]^2)$ indexing the generators of $L_1(3,4,2)$} at 12 -2

\endpicture 
\]

\begin{theorem}
\label{Z3xZ3xZ3sdZ2}
	For $i\in [4]$, $L_{i}$
	is contained in $\lambda^{-1}(\P_4^2)$ and is isomorphic to $(\ZZ_3^{ \times 3}) \rtimes \ZZ_2$.
	The kernel of the canonical projection $\pi : L_{i} \rightarrow {\rm Out}(\O_4)$
	is trivial and the 4 groups $\pi(L_{i})$ are mutually distinct for $i \in [4]$.
\end{theorem}
\begin{proof}
It is an immediate consequence of Theorem \ref{ZrxZrxZrsdZ2} and Theorem \ref{ZrxZrxZrsdZ2_distinct}
(since $i \in [4]\setminus \{j\}$ for all $j \neq i$).
\end{proof}

\subsection{The sixth line}

Throughout this subsection, we assume that $n \geq 3$.
Recall from section 3 that for $i \in [n]$ we let $R_i:=R_{\{ i \}}$ and define similarly $C_i$, $\check{R}_i$,
and $\check{C}_i$ (so $R_i$ consists of all the permutations of $S([n]^2)$
that only permute the elements in row $i$ that are not in 
column $i$, while those of $\check{R}_i$ fix pointwise row $i$,
 and permute all the other rows in the same way
but fixing all the elements in column $i$,
and similarly for $C$'s).
We find it convenient to let
$$
\lambda(S) := \{ \lambda_u : u \in S\},
$$
for any $S \subseteq S([n]^2)$. We let $Q_i$ be the subgroup
of ${\rm Aut}(\O_n)$ generated by 
\[
\lambda(R_i \cup \check{R}_i)
\]
and define $Q'_i$ similarly using $C$'s.
\begin{theorem}
For $i \in [n]$, $Q_i$ is contained in $\lambda^{-1}(\P_n^2)$, 
is isomorphic to $(S_{n-1})^{\times 2}$, and the kernel of the canonical projection $\pi : Q_{i} \rightarrow {\rm Out}(\O_n)$
is trivial.
Similarly for $Q'_i$. The $2n$ subgroups $\pi(Q_i)$ and
$\pi(Q'_i)$ are all distinct .
\end{theorem}
\begin{proof}
We find it convenient to let $\widetilde{Q}_i$ be the
subgroup of $S([n]^2)$ generated by $R_i \cup \check{R}_i$.
The first two statements readily follow from Theorem \ref{R_P} (with $P=\{ i \}$) and Proposition \ref{isogroups}. Furthermore,
if $u \in R_i$ and $v \in \check{R}_i$ are such that 
$\pi(\lambda_{uv})=1$ (it is clear that every element of 
$\widetilde{Q}_i$ is of the form $uv$ for such $u$ and $v$)
then there is $p \in S_n$ such that
$(p \otimes 1) u v (1 \otimes p^{-1})=1$. But all the elements 
of $R_i$ and $\check{R}_i$ leave every row invariant so 
$(p \otimes 1) u v (1 \otimes p^{-1})$ maps row $x$ to row $p(x)$ for all $x \in [n]$, therefore $p=1$. Hence $uv=1$ and this 
implies that $u=v=1$.

We now show that the groups $\pi(Q_i)$ are all
distinct. Suppose that $\pi(Q_i)=\pi(Q_j)$ for some $1 \leq i<j \leq n$. Let $a \in [n] \setminus \{ i,j \}$, and $u:=((i,j),(i,a))$. Then $\pi(\lambda_u) \in \pi(Q_i)$ so 
there is $v \in \widetilde{Q}_j$ such that $\pi(\lambda_u)=\pi(\lambda_v)$. Hence there is 
$p \in S_n$ such that 
$(p \otimes 1) u (1 \otimes p^{-1})=v$. But, as observed in the 
previous paragraph, $(p \otimes 1) u (1 \otimes p^{-1})$ maps
row $x$ to row $p(x)$ for all $x \in [n]$, so $p=1$. 
Therefore $u \in \widetilde{Q}_j$ which is a contradiction since 
every element of $\widetilde{Q}_j$ leaves column $j$ fixed.
Similarly for the groups $\pi(Q'_i)$.

Last, we show that $\pi(Q_i) \neq \pi(Q'_j)$ for all $i,j \in [n]$. Suppose not. Let $a$ and $u$ be as in the previous paragraph. Then $\pi(\lambda_u) \in \pi(Q'_j)$ so reasoning 
as above we conclude that 
there are $v \in \widetilde{Q'}_j$ and $p \in S_n$ such that 
$(p \otimes 1) u (1 \otimes p^{-1})=v$. Let $r \in [n] \setminus \{ i \}$ and $b \in [n]$. Then we have that 
$$
(p(r),b)=(p \otimes 1) u (1 \otimes p^{-1})(r,p(b))=v(r,p(b)).
$$
But every element of $\widetilde{Q'}_j$ leaves every column invariant, so $v(r,p(b))$ is in column $p(b)$. Hence $p(b)=b$
and so $p=1$. 
Therefore $u \in \widetilde{Q'}_j$ which impossible for the same reason.
\end{proof}

Keeping the notation as in the previous proof, we remark that $\widetilde{Q}_i \cap \widetilde{Q}_j = 1 \otimes S_{i,j}$, for all $1 \leq i<j \leq n$, as can be easily checked.

\medskip
For $n=4$ the multiplicity is in agreement with the one computed in  \cite[Table 5]{AJS18}.
According to \cite{AJS18}, for $n=4$ all these subgroups are maximal in $\{[\lambda_u] \in {\rm Out}(\O_4) \ | \ u \in \P_4^2\}$.

\subsection{The seventh line}
By \cite[Table 5]{AJS18} there are 7 maximal subgroups in the image of $\lambda^{-1}(\P_4^2)$ in ${\rm Out}(\O_4)$ that are isomorphic to $S_4$.
It is very easy to spot one of them, namely the image of permutative Bogolubov automorphisms in ${\rm Out}(\O_4)$. The remaining six groups are more difficult to describe, but can be constructed by an in-depth analysis of the {\it bicompatible 
subgroups} isomorphic to $S_4$ in $S([4]^2)$. 
This result is actually a special case of a general study of bicompatible subgroups, which is presented in the next section. It is also worth to
stress that for each of these 7 subgroups one can find a set of Coxeter generators associated to permutations of cycle-type $(2,2,2,2)$.

Let $a,b \in \{ 2,3,4 \}$, $a \neq b$, and $\{c\}:=\{ 2,3,4 \}
\setminus \{ a,b \}$. We define 
$$
s_1:=((1,1),(1,b)) \, ((b,1),(b,b)) \,((1,a),(1,c)) \,((b,a)),(b,c))
$$
$$
s_2:=((1,b),(a,b)) \,((b,b),(c,b)) \,((1,c),(a,c)) \,((b,c),(c,c)) 
$$
$$
s_3:=((a,1),(a,b)) \,((c,1),(c,b)) \,((a,a),(a,c)) \,((c,a),(c,c))
$$
(so each $s_i \in S([4]^2)$ is the product of four $2$-cycles).
We let $T_{a,b}$ be the subgroup of ${\rm Aut}(\O_4)$ generated by 
$$
\{ \lambda_{s_1}, \lambda_{s_2}, \lambda_{s_3} \},
$$
and $\widetilde{T}_{a,b}$ the subgroup of $S([4]^2)$ 
generated by $s_1$, $s_2$, and $s_3$. Note that each 
$\widetilde{T}_{a,b}$ is isomorphic to $S_4$.
For the reader's convenience we include a graphical representation of the three permutations $s_1$, $s_2$, and $s_3$ for $a=2$ and
$b=4$ (see also Figure 1 for another graphical representation
of the six groups $\widetilde{T}_{a,b}$ that should be 
self-explanatory). Note that $\widetilde{T}_{a,b}$ leaves 
the union of rows $1$ and $a$, and of columns $1$ and $b$, invariant.

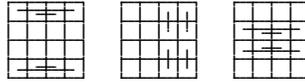
\[  \beginpicture

\setcoordinatesystem units <0.25cm,0.25cm>
\setplotarea x from 10 to 20, y from 2 to 5

\setlinear

\plot 0 0 0 4 /
\plot 1 0 1 4 /
\plot 2 0 2 4 /
\plot 3 0 3 4 /
\plot 4 0 4 4 /

\plot 0 0 4 0 /
\plot 0 1 4 1 /
\plot 0 2 4 2 /
\plot 0 3 4 3 /
\plot 0 4 4 4 /

\plot 0.5 0.4  3.5 0.4 /
\plot 1.5 0.6 2.5 0.6 /
\plot 0.5 3.6  3.5 3.6 /
\plot 1.5 3.4 2.5 3.4 /

\setlinear 

\plot 6 0 6 4 /
\plot 7 0 7 4 /
\plot 8 0 8 4 /
\plot 9 0 9 4 /
\plot 10 0 10 4 /

\plot 6 0 10 0 /
\plot 6 1 10 1 /
\plot 6 2 10 2 /
\plot 6 3 10 3 /
\plot 6 4 10 4 /

\plot 8.5 0.5  8.5 1.5 /
\plot 9.5 0.5 9.5 1.5 /
\plot 8.5 2.5  8.5 3.5 /
\plot 9.5 2.5 9.5 3.5 /

\plot 12 0 12 4 /
\plot 13 0 13 4 /
\plot 14 0 14 4 /
\plot 15 0 15 4 /
\plot 16 0 16 4 /

\plot 12 0 16 0 /
\plot 12 1 16 1 /
\plot 12 2 16 2 /
\plot 12 3 16 3 /
\plot 12 4 16 4 /

\plot 12.5 1.4  15.5 1.4 /
\plot 13.5 1.6 14.5 1.6 /
\plot 12.5 2.6  15.5 2.6 /
\plot 13.5 2.4 14.5 2.4 /

\put {The three permutations $s_1$, $s_2$, and $s_3$ for 
$a=2$ and $b=4$} at 8 -2

\endpicture \]

\begin{theorem}
	\label{S4}
For $a,b \in \{ 2,3,4 \}$, $a \neq b$, $T_{a,b}$
is contained in $\lambda^{-1}(\P_4^2)$ and is isomorphic to 
$S_4$. The kernel of the canonical projection $\pi : T_{a,b} \rightarrow {\rm Out}(\O_4)$
is trivial and the 6 groups $\pi(T_{a,b})$ are mutually distinct.
\end{theorem}
\begin{proof}
One can check that the three Coxeter generators $s_1$, $s_2$ and $s_3$ are mutually bicompatible for all $a$ and $b$ as in the statement.
Indeed, by Proposition \ref{horizontal_same_columns}, $s_1$ is compatible with itself, $s_2$ and $s_3$ 
(since all the $s_i$'s leave $[4] \times \{1,b\}$ invariant).
Similarly, $s_3$ is compatible with $s_1$, $s_2$ and itself 
(since all the $s_i$'s leave $[4] \times \{a,c\}$ invariant).
Also, $s_2$ is compatible with itself by Proposition \ref{vertical_same_rows} (since $s_2$ leaves $\{b,c\}\times [4])$ invariant).
It remains to show that $s_2$ is compatible with both $s_1$ and with $s_3$.
In order to see this, 
we apply Proposition \ref{horizontal_same_columns_2_cycles}
with $t(\alpha,x) = 0$ for all $\alpha \in \{e,f\}$ and $x \in [4]$, and $\sigma = ((e,c),(f,c))((e,b),(f,b)) $.
Therefore, the first claim of the theorem follows immediately from Proposition \ref{isogroups}. 
The remaining two claims are consequences of the next two propositions.
\end{proof}

See the paragraph preceeding Lemma \ref{samerow} for some relevant definition.
\begin{proposition}
The kernel of the canonical projection from any of the
subgroups in Figure 1 to ${\rm Out} (\O_4)$ is trivial.
\end{proposition}
\begin{proof}
	It is enough to show that there is no element of the form
$\sigma \otimes \sigma^{-1}$, with $\sigma \in S_4 \setminus \{ 1 \}$, in any of the subgroups in Figure 1. Suppose that there is.
Then there are $a,b \in [4]$ such that $\sigma(a)=b$. Therefore,
$\{ a,b \}$ is and edge of both $G_R$ and $G_C$. But it is easy to 
check that $G_R \cap G_C = \emptyset$ for all of the six subgroups
depicted in Figure 1. 
\end{proof}

\begin{proposition}
	The images under the canonical projection of the six
	subgroups in Figure 1 to ${\rm Out} (\O_4)$ are distinct.
\end{proposition}
\begin{proof}
Let $s_1^{(1)}$ be the first Coxeter generator of 
$\widetilde{T}_{\{ 1,2 \}, \{ 1,4 \}}$. So
\[
s_1^{(1)} = ((1,1),(1,4))((1,2),(1,3))
((4,1),(4,4))((4,2),(4,3)).
\]
Suppose that there is a $\sigma \in S_4$ such that
$(\sigma \otimes 1) s_1^{(1)} (1 \otimes \sigma^{-1}) \in
\widetilde{T}_{\{ 1,2 \}, \{ 1,3 \}}$.
Since $s_1^{(1)}$ leaves all rows invariant we have that 
$(\sigma \otimes 1) s_1^{(1)} (1 \otimes \sigma^{-1})$ maps
row $i$ to row $\sigma(i)$ for all $i \in [4]$. Since all the
elements of $\widetilde{T}_{\{ 1,2 \}, \{ 1,3 \}}$ leave the union of 
rows $1$ and $2$ invariant we deduce that $\sigma(\{1,2\})=\{1,2\}$
and $\sigma(\{3,4\})=\{ 3,4 \}$. If $\sigma(1)=2$ then
$\sigma(2)=1$ and so 
\[
(\sigma \otimes 1) s_1^{(1)} (1 \otimes \sigma^{-1})(2,1)=(1,2)
\]
which is a contradiction since no element of 
$\widetilde{T}_{\{ 1,2 \}, \{ 1,3 \}}$ has this property ($(1,2)$ and $(2,1)$
are in different orbits of $\widetilde{T}_{\{ 1,2 \}, \{ 1,3 \}}$).
So $\sigma(1)=1$ and $\sigma(2)=2$. Similarly, if $\sigma(3)=4$
then $\sigma(4)=3$ and then 
\[
(\sigma \otimes 1) s_1^{(1)} (1 \otimes \sigma^{-1})(3,4)=(4,3)
\]
which is again a contradiction. Hence $\sigma=1$ so 
$s_1^{(1)} \in \widetilde{T}_{\{ 1,2 \}, \{ 1,3 \}}$ which is also impossible.
This shows that the images of $\widetilde{T}_{\{ 1,2 \}, \{ 1,3 \}}$ and
$\widetilde{T}_{\{ 1,2 \}, \{ 1,4 \}}$ in ${\rm Out} (\O_4)$ are distinct.

Let $s_1^{(4)}$ be the first Coxeter generator of 
$\widetilde{T}_{\{ 1,4 \}, \{ 1,2 \}}$, which is also an element of
$\widetilde{T}_{\{ 1,3 \}, \{ 1,2 \}}$. So
\[
s_1^{(1)} = ((1,1),(1,2))((1,3),(1,4))
((2,1),(2,2))((2,3),(2,4)).
\]
Suppose that there is a $\sigma \in S_4$ such that
$(\sigma \otimes 1) s_1^{(4)} (1 \otimes \sigma^{-1}) \in
\widetilde{T}_{\{ 1,4 \}, \{ 1,3 \}}$.
Since $s_1^{(4)}$ leaves all rows invariant we have that 
$(\sigma \otimes 1) s_1^{(4)} (1 \otimes \sigma^{-1})$ maps
row $i$ to row $\sigma(i)$ for all $i \in [4]$. Since all the
elements of $\widetilde{T}_{\{ 1,4 \}, \{ 1,3 \}}$ leave the union of 
rows $1$ and $4$ invariant we deduce that $\sigma(\{1,4\})=\{1,4\}$
and $\sigma(\{ 2,3 \})=\{ 2,3 \}$. If $\sigma(2)=3$ then
$\sigma(3)=2$ and so 
\[
(\sigma \otimes 1) s_1^{(4)} (1 \otimes \sigma^{-1})(3,2)=(2,3)
\]
which is a contradiction since no element of any of these six
subgroups has this property ($(3,2)$ and $(2,3)$
are in different orbits, as are $(1,4)$ and $(4,1)$).
So $\sigma(2)=2$ and $\sigma(3)=3$. Similarly, if $\sigma(1)=4$
then $\sigma(4)=1$ and then 
\[
(\sigma \otimes 1) s_1^{(4)} (1 \otimes \sigma^{-1})(4,1)=(1,4)
\]
which is again a contradiction. Hence $\sigma=1$ so 
$s_1^{(4)} \in \widetilde{T}_{\{ 1,4 \}, \{ 1,3 \}}$ which is also impossible.
This shows that the images of $\widetilde{T}_{\{ 1,h \}, \{ 1,2 \}}$ and
$\widetilde{T}_{\{ 1,4 \}, \{ 1,3 \}}$ in ${\rm Out} (\O_4)$ are distinct
for any $h \in \{ 3,4 \}$.

The other cases are all similar and are therefore omitted.
\end{proof}

\[  \beginpicture

\setcoordinatesystem units <0.5cm,0.5cm>
\setplotarea x from 7 to 8, y from 4 to 7

\setlinear

\plot -10 10 -10 6 /
\plot -9 10 -9 6 /
\plot -8 10 -8 6 /
\plot -7 10 -7 6 /
\plot -6 10 -6 6 /

\plot -6 6 -10  6 /
\plot -6 7 -10 7 /
\plot -6 8 -10 8 /
\plot -6 9 -10 9 /
\plot -6 10 -10 10 /

\put {\rotatebox{270}{1}} at -9.5 6.5
\put {\rotatebox{270}{2}} at -6.5 6.5
\put {\rotatebox{270}{3}} at -6.5 7.5
\put {\rotatebox{270}{4}} at -9.5 7.5

\color{green} 
\put {\rotatebox{90}{1}} at -8.5 9.5
\put {\rotatebox{90}{2}} at -7.5 9.5
\put {\rotatebox{90}{3}} at -7.5 8.5
\put {\rotatebox{90}{4}} at -8.5 8.5

\color{red} 
\put {1} at -9.5 9.5
\put {2} at -6.5 9.5
\put {3} at -6.5 8.5
\put {4} at -9.5 8.5

\color{blue} 
\put {\rotatebox{180}{1}} at -8.5 6.5
\put {\rotatebox{180}{2}} at -7.5 6.5
\put {\rotatebox{180}{3}} at -7.5 7.5
\put {\rotatebox{180}{4}} at -8.5 7.5

\color{black}
\plot 0 10 0 6 /
\plot 1 10 1 6 /
\plot 2 10 2 6 /
\plot 3 10 3 6 /
\plot 4 10 4 6 /

\plot 4 6 0 6 /
\plot 4 7 0 7 /
\plot 4 8 0 8 /
\plot 4 9 0 9 /
\plot 4 10 0 10 /

\color{black}
\put {\rotatebox{270}{4}} at 0.5 6.5
\color{blue}
\put {\rotatebox{180}{4}} at 1.5 6.5
\color{black}
\put {\rotatebox{270}{3}} at 2.5 6.5
\color{blue}
\put {\rotatebox{180}{3}} at 3.5 6.5

\color{black} 
\put {\rotatebox{270}{1}} at 0.5 7.5
\color{blue}
\put {\rotatebox{180}{1}} at 1.5 7.5
\color{black}
\put {\rotatebox{270}{2}} at 2.5 7.5
\color{blue}
\put {\rotatebox{180}{2}} at 3.5 7.5

\color{red} 
\put {4} at 0.5 8.5
\color{green} 
\put {\rotatebox{90}{4}} at 1.5 8.5
\color{red} 
\put {3} at 2.5 8.5
\color{green} 
\put {\rotatebox{90}{3}} at 3.5 8.5

\color{red} 
\put {1} at 0.5 9.5
\color{green}
\put {\rotatebox{90}{1}} at 1.5 9.5
\color{red}
\put {2} at 2.5 9.5
\color{green} 
\put {\rotatebox{90}{2}} at 3.5 9.5

\color{black} 

\plot 10 10 10 6 /
\plot 11 10 11 6 /
\plot 12 10 12 6 /
\plot 13 10 13 6 /
\plot 14 10 14 6 /

\plot 14 6 10 6 /
\plot 14 7 10 7 /
\plot 14 8 10 8 /
\plot 14 9 10 9 /
\plot 14 10 10 10 /

\color{black} 
\put {\rotatebox{270}{1}} at 10.5 6.5
\color{blue}
\put {\rotatebox{180}{2}} at 11.5 6.5
\put {1} at 12.5 6.5
\color{black}
\put {\rotatebox{270}{2}} at 13.5 6.5

\color{red} 
\put {4} at 10.5 7.5
\color{green}
\put {\rotatebox{90}{3}} at 11.5 7.5
\put {\rotatebox{90}{4}} at 12.5 7.5
\color{red}
\put {3} at 13.5 7.5

\color{black} 
\put {\rotatebox{270}{4}} at 10.5 8.5
\color{blue} 
\put {\rotatebox{180}{3}} at 11.5 8.5 
\put {\rotatebox{180}{4}} at 12.5 8.5
\color{black} 
\put {\rotatebox{270}{3}} at 13.5 8.5

\color{red} 
\put {1} at 10.5 9.5
\color{green}
\put {\rotatebox{90}{2}} at 11.5 9.5
\put {\rotatebox{90}{1}} at 12.5 9.5
\color{red} 
\put {2} at 13.5 9.5

\color{black}
 
\plot -10 4 -10 0 /
\plot -9 4 -9 0 /
\plot -8 4 -8 0 /
\plot -7 4 -7 0 /
\plot -6 4 -6 0 /

\plot -6 0 -10 0 /
\plot -6 1 -10 1 /
\plot -6 2 -10 2 /
\plot -6 3 -10 3 /
\plot -6 4 -10 4 /

\color{red} 
\put {4} at -9.5 0.5
\put {3} at -8.5 0.5
\color{green}
\put {\rotatebox{90}{3}} at -7.5 0.5
\put {\rotatebox{90}{4}} at -6.5 0.5

\color{black} 
\put {\rotatebox{270}{4}} at -9.5 1.5
\put {\rotatebox{270}{3}} at -8.5 1.5
\color{blue}
\put {\rotatebox{180}{3}} at -7.5 1.5
\put {\rotatebox{180}{4}} at -6.5 1.5

\color{black} 
\put {\rotatebox{270}{1}} at -9.5 2.5 
\put {\rotatebox{270}{2}} at -8.5 2.5
\color{blue} 
\put {\rotatebox{180}{2}} at -7.5 2.5
\put {\rotatebox{180}{1}} at -6.5 2.5

\color{red} 
\put {1} at -9.5 3.5
\put {2} at -8.5 3.5
\color{green}
\put {\rotatebox{90}{2}} at -7.5 3.5
\put {\rotatebox{90}{1}} at -6.5 3.5

\color{black}

\plot 0 4 0 0 /
\plot 1 4 1 0 /
\plot 2 4 2 0 /
\plot 3 4 3 0 /
\plot 4 4 4 0 /

\plot 4 0 0 0 /
\plot 4 1 0 1 /
\plot 4 2 0 2 /
\plot 4 3 0 3 /
\plot 4 4 0 4 /

\color{red} 
\put {4} at 0.5 0.5
\color{green} 
\put {\rotatebox{90}{3}} at 1.5 0.5
\color{red}
\put {3} at 2.5 0.5
\color{green} 
\put {\rotatebox{90}{4}} at 3.5 0.5

\color{black} 
\put {\rotatebox{270}{1}} at 0.5 1.5
\color{blue} 
\put {\rotatebox{180}{2}} at 1.5 1.5
\color{black}
\put {\rotatebox{270}{2}} at 2.5 1.5
\color{blue} 
\put {\rotatebox{180}{1}} at 3.5 1.5

\color{black} 
\put {\rotatebox{270}{4}} at 0.5 2.5 
\color{blue} 
\put {\rotatebox{180}{3}} at 1.5 2.5
\color{black} 
\put {\rotatebox{270}{3}} at 2.5 2.5
\color{blue} 
\put {\rotatebox{180}{4}} at 3.5 2.5

\color{red} 
\put {1} at 0.5 3.5
\color{green} 
\put {\rotatebox{90}{2}} at 1.5 3.5
\color{red}
\put {2} at 2.5 3.5
\color{green} 
\put {\rotatebox{90}{1}} at 3.5 3.5

\color{black} 

\plot 10 4 10 0 /
\plot 11 4 11 0 /
\plot 12 4 12 0 /
\plot 13 4 13 0 /
\plot 14 4 14 0 /

\plot 14 0 10 0 /
\plot 14 1 10 1 /
\plot 14 2 10 2 /
\plot 14 3 10 3 /
\plot 14 4 10 4 /

\color{black} 
\put {\rotatebox{270}{4}} at 10.5 0.5
\put {\rotatebox{270}{3}} at 11.5 0.5
\color{blue}
\put {\rotatebox{180}{4}} at 12.5 0.5
\put {\rotatebox{180}{3}} at 13.5 0.5

\color{red} 
\put {4} at 10.5 1.5
\put {3} at 11.5 1.5
\color{green}
\put {\rotatebox{90}{4}} at 12.5 1.5
\put {\rotatebox{90}{3}} at 13.5 1.5

\color{black} 
\put {\rotatebox{270}{1}} at 10.5 2.5 
\put {\rotatebox{270}{2}} at 11.5 2.5
\color{blue} 
\put {\rotatebox{180}{1}} at 12.5 2.5
\put {\rotatebox{180}{2}} at 13.5 2.5

\color{red} 
\put {1} at 10.5 3.5
\put {2} at 11.5 3.5
\color{green}
\put {\rotatebox{90}{1}} at 12.5 3.5
\put {\rotatebox{90}{2}} at 13.5 3.5

\color{black}
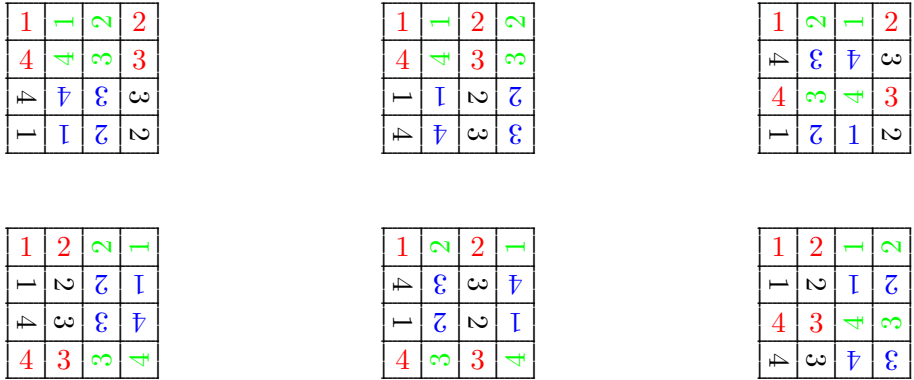
 \put {The six bicompatible subgroups of $S([4]^2)$ isomorphic to $S_4$} at 2 -2.5
 
 \put{\bf{Fig.1}} at 2, -4

\endpicture \]

\section{Bicompatible subgroups and embeddings}
\label{bicomp_subgs}
In this next to last section we derive a series of results
about bicompatible subgroups and then we apply them to
the classification of all the bicompatible subgroups of 
$S([4]^2)$ that are isomorphic to $S_4$. We do this because 
we feel that 
they may be of independent interest, and because they show
how some of the subgroups presented and studied in Section 5
were found. The results in the first subsection hold for 
any $n$, while those in the second subsection are specific
to $n=4$.

\subsection{Bicompatible subgroups}

Let $V$ be a bicompatible subgroup of $S([n]^2)$. We attach two
graphs to $V$ as follows. We let $G_R :=([n],E_R)$ where, for any
$a,b \in [n]$, $\{ a,b \} \in E_R$ if and only if
there are $x,y \in [n]$ and $v \in V$ such that $v(a,x)=(b,y)$.
Thus, there is an edge in $G_R$ between $a$ and $b$ if and only 
if there is some element of $V$ that maps some element in row
$a$ to some element in row $b$. Similarly, we define
$G_C :=([n],E_C)$ by letting $\{ a,b \} \in E_C$ if and only if
there is some element of $V$ that maps some element in column
$a$ to some element in column $b$.

\begin{lemma} \label{samerow}
	Let $a,b \in [n]$ be in the same connected component of
	$G_R$, $c \in [n]$, and $v \in V$. Then
	$v(c,a)$ and $v(c,b)$ are in the same row.
\end{lemma}
\begin{proof}
	We may assume that $\{ a,b \} \in E_R$. Then there are 
	$x,y \in [n]$ and $u \in V$ such that $u(b,x)=(a,y)$.
	Since $V$ is bicompatible we therefore have that
	\[
	(v \otimes 1)(c,a,y)= (v \otimes 1)(1 \otimes u)(c,b,x)=
	(1 \otimes u)(v \otimes 1)(c,b,x).
	\]
	Hence the first coordinates of $(v \otimes 1)(c,a,y)$
	and $(v \otimes 1)(c,b,x)$ are the same.
\end{proof}

\noindent
There is an analogous result for $G_C$ and columns.
\begin{lemma}
	\label{samecol}
	Let $a,b \in [n]$ be in the same connected component of
	$G_C$, $r \in [n]$, and $v \in V$. Then
	$v(a,r)$ and $v(b,r)$ are in the same column.
\end{lemma}

As a consequence of the the last two lemmas we obtain the following result which shows that, on each connected component of both $G_R$ and $G_C$, each element of $V$ acts as a tensor product.
\begin{corollary}
	\label{conn_comp_of_both}
	Let $A \subseteq [n]$ be a connected 
	component of both $G_R$ and $G_C$, and $v \in V$.
	Then there are $\tau_v, \sigma_v \in S(A)$ such that
	\[
	v(x,y)=(\sigma_v(x),\tau_v(y))
	\]
	for all $x,y \in A$
\end{corollary}
\begin{proof}
	Let $a \in A$. We define a map $\sigma_v : A \rightarrow A$ by letting
	\[
	\sigma_v(x) := v(x,a)_1
	\]  
	(the first coordinate of $v(x,a)$) for all $x \in A$.
	Then $\sigma_v$ is injective since, if $\sigma_v(x)=\sigma_v(y)$
	for some $x,y \in A$ then $v(x,a)$ and $v(y,a)$ are in the same 
	row. But, by Lemma \ref{samecol}, 
	$v(x,a)$ and $v(y,a)$ are in the same
	column. So $v(x,a)=v(y,a)$ and hence $(x,a)=(y,a)$. Hence
	$\sigma_v$ is injective and therefore a permutation. 
	Similarly for $\tau_v$.
\end{proof}

The next lemma shows that isolated vertices in one of the two 
graphs often force other vertices to be isolated as well.
\begin{lemma} \label{isoconn}
	Let $t \in [n]$ be an isolated vertex of $G_R$. Then the connected
	component of $t$ in $G_C$ contains only isolated vertices of $G_R$.
\end{lemma}
\begin{proof}
	By our hypotheses, and the definition of $G_R$, $V$ leaves row $t$ 
	invariant. Let $s \in [n]$ be in the same connected component as $t$ in $G_C$. We may assume that $\{ s,t \} \in E_C$. Then there 
	are $x,y \in [n]$ and $u \in V$ such that $u(x,t)=(y,s)$. We wish
	to show that $V$ leaves row $s$ invariant. Let $z \in [n]$ and
	$v \in V$. Since $V$ is bicompatible we have that
	\begin{align*}
		(1 \otimes v)(y,s,z) & =(1 \otimes v)(u \otimes 1)(x,t,z)
		=(u \otimes 1)(1 \otimes v)(x,t,z) \\
		& =(u \otimes 1)(x,t,z')=(y,s,z')
	\end{align*}
	for some $z' \in [n]$, so $v(s,z)=(s,z')$, as claimed.
\end{proof}

A similar statement holds if one interchanges ``$R$'' and 
``$C$''.
We also note the following consequence of Lemma \ref{samerow}.
\begin{lemma}
	\label{connected->complete}
	If $G_R$ is connected, then	$G_R$ is a complete graph.
\end{lemma}
\begin{proof}
	Let $v \in V$. Then, by Lemma \ref{samerow},
	there is $\sigma_v \in S_n$ such that 
	$v(a,b)$ is in row $\sigma_v(a)$ for all $(a,b) \in [n]^2$.
	If $\{a,b\} \in E_R$, there are $g,h \in [n]$ and $w \in V$ such that $w(a,g)=(b,h)$, so that $\sigma_w(a) = b$.
	Similarly, if $\{b,c\} \in E_R$, there is $u \in V$ such that $\sigma_u(b) = c$.
	Hence the first component of $(uw)(a,g) = u(b,h)$ is $\sigma_u(b) = c$ and $\{a,c\} \in E_R$.
\end{proof}

There is of course an analogous statement for $G_C$.

\medskip 
Assume now that $G_R$ is a complete graph.
Let $v \in V$. Then, by Lemma \ref{samerow},
there is $\sigma_v \in S_n$ such that 
$v(a,b)$ is in row $\sigma_v(a)$ for all $(a,b) \in [n]^2$.
We point out the following simple observation.
\begin{proposition}
	\label{homsigma}
	Suppose that $G_R$ is a complete graph. Then
	the map $\sigma: V \to S_n$ mapping $v$ to $\sigma_v$ is a group homomorphism.
\end{proposition}
\begin{proof}
	Let $u,w \in V$. Then, for every $a,b \in [n]$, 
	$$
	\sigma_{uw}(a) = (uw)(a,b)_1 = u(w(a,b))_1 = \sigma_u(w(a,b)_1) = \sigma_u(\sigma_w(a)). 
	$$
\end{proof}
There is, again, an analogous statement for $G_C$.

\smallskip
We conclude with a further result that gives yet more 
constraints on the elements of a bicompatible subgroup of
$S([n]^2)$.
\begin{proposition} 
	\label{change_col}
	Let $a,b \in [n]$ be in the same connected component of
	$G_R$, $c \in [n]$, and $v \in V$ . Then
	$v(c,a)_2=a$ if and only if $v(c,b)_2=b$. 
	Similarly, if
	$a$ and $b$ are in the same connected component of $G_C$ 
	then $v(a,c)_1=a$ if and only if $v(b,c)_1=b$.
\end{proposition}
\begin{proof}
	We only prove the first statement, the other one being
	entirely analogous.
	We may assume that $\{ a,b \} \in E_R$. Then there are 
	$x,y \in [n]$ and $u \in V$ such that $u(a,x)=(b,y)$.
	Since $V$ is bicompatible we therefore have that
	\begin{align*}
		(v \otimes 1)(c,b,y) & = (v \otimes 1)(1 \otimes u)(c,a,x)=
		(1 \otimes u)(v \otimes 1)(c,a,x) \\ 
		& =(1 \otimes u)(d,a,x)=(d,b,y)
	\end{align*}
	for some $d \in [n]$. So $v(c,b)$ is in column $b$.
\end{proof}

\subsection{Bicompatible embeddings of $S_4$ in $S([4]^2)$}

In this subsection, by looking at the possible 
configurations of $G_R$ and $G_C$, we classify all the bicompatible subgroups of $S([4]^2)$ that are isomorphic to $S_4$, and
see that these are exactly the 4 subgroups of $N_{i,j}$ 
that appeared in the discussion of the second line
(so 24 in all), the 
groups $S_4 \otimes 1$ and $1 \otimes S_4$,  
and the 6 groups $\widetilde{T}_{a,b}$.

\begin{theorem}
Let $V$ be a subgroup of $S([4]^2)$. Then the following conditions are equivalent:
\begin{itemize}
\item[i)] $V$ is bicompatible and isomorphic to $S_4$;
\item[ii)] $V$ is one of the following groups:
\begin{itemize}
\item $V$ is a copy of $S_4$ inside any of the groups $N_{i,j}$ defined in the proof of Theorem \ref{S4xZ2xZ2alt} (24 copies in total); 
\item $V = S_4 \otimes 1$, or $V = 1 \otimes S_4$;
\item $V$ is one of the 6 copies of $S_4$ described in Theorem \ref{S4}.
\end{itemize}
\end{itemize}	
\end{theorem}
\begin{proof}
Assume that $V$ is bicompatible and isomorphic to $S_4$.
We distinguish two main cases, according as to whether 
$G_R \cup G_C$ is connected or not (where $G_R \cup G_C :=
([4], E_R \cup E_C)$).

\bigskip
\noindent
{\bf A): $G_R \cup G_C$ is disconnected.}

This implies that both $G_R$ and $G_C$ are disconnected. 
Therefore each of $G_R$ and $G_C$ has either 2, 3, or 4 
connected components. We distinguish these resulting 9 cases.
However, due to the symmetry between ``R" and ``C", we only treat
6 of them.

\medskip
\noindent
{\bf A1): $G_R$ and $G_C$ both have 2 connected components.}

Here again we have to consider two separate possibilities.

\smallskip 
\noindent 
{\bf A1i): $G_R$ and $G_C$ have no isolated vertices.}
Let $E_1:=\{ x,y \}$ and $E_2:=\{ z,t \}$ be the connected components 
of $G_R$, so
 $E_R = \{ \{ x,y \},$ $\{ z,t \} \}$. Therefore, since 
 $G_R \cup G_C$ is
disconnected, and $G_C$ has two connected components, $G_C=G_R$.
This implies that $V$ has at least 4 orbits, namely 
$E_1 \times E_1$, $E_1 \times E_2$,
$E_2 \times E_1$, and $E_2 \times E_2$.

Let $v \in V$.
Then, by Corollary \ref {conn_comp_of_both}, there are 
$\sigma , \tau \in S(E_1)$ such that $v=\sigma \otimes \tau$
in $(E_1)^{2}$ and hence $v^2(a,b)=(a,b)$ for all 
$a,b \in E_1$. Similarly, $v^2(a,b)=(a,b)$ for all 
$a,b \in E_2$. Furthermore, by Lemma \ref{samerow},
$v(c,z)$ and $v(c,t)$ are in the same row for $c \in 
E_1$, and by Lemma 
\ref{samecol} $v(x,c)$ and $v(y,c)$ are in the same column if $c \in E_2$. 
Therefore, $v^2$ restricted to $E_1 \times E_2$ preserves both rows and columns.
Thus $v^2=Id$ on $E_1 \times E_2$. Similarly for $E_2 \times E_1$.
Hence $v^2=1$. Since this holds for any 
$v \in V$, this contradicts our hypothesis i). 

\smallskip
\noindent 
{\bf A1ii): $G_R$ or $G_C$ have isolated vertices.}
Let $t$ be the isolated vertex of $G_R$, so that $E := \{x,y,z\}$ is the other connected component of $G_R$.
Then $t$ is isolated in $G_C$ (otherwise $G_R \cup G_C$ would be connected) so $E$ is also a connected component of $G_C$. Thus
$V$ has at least 4 orbits, namely 
$E \times E$, $E \times \{ t \}$,
$\{ t \} \times E$, and $\{ t \} \times \{ t \}$.

Let $v \in V$.
Then, by Corollary \ref {conn_comp_of_both}, there are 
$\sigma , \tau \in S(E)$ such that $v=\sigma \otimes \tau$
in $E^{2}$ and hence $v^6(a,b)=(a,b)$ for all 
$a,b \in E$. Furthermore, since $|E \times \{ t \}|=
|\{ t \} \times E|=3$ we have that $v^6=Id$ on $E \times \{ t \}$
and on $\{ t \} \times E$. Since $v(t,t)=(t,t)$ we conclude that 
$v^6=1$. This holds for any 
$v \in V$, and so contradicts our hypothesis i). 

{\bf A2): $G_R$ has 3 connected components and $G_C$ has 2 connected components.}

Let $x$ and $y$ be the two isolated vertices of $G_R$, and $E = \{ z,t \}$. 
Then it is easy to see, by Lemma \ref{isoconn} and the running hypothesis A),
that  $E_C = \{\{x,y\}, \{z,t\}\}$. 
Thus $V$ has at least 6 orbits $\{x\} \times \{x,y\}$, $\{x\} \times E$, $\{y\} \times \{x,y\}$, $\{y\} \times E$,
$E \times \{x,y\}$, $E \times E$.
By Corollary \ref{conn_comp_of_both}, there exist $\sigma, \tau \in S(E)$ such that $v = \sigma \otimes \tau$
on $E^2$ for every $v \in V$, and thus $v^2 = Id$ on $E^2$.
By Lemma \ref{samecol}, $v(z,x)$ and $v(t,x)$ lie in the same column, and hence $v^2$ restricted to $E \times \{x,y\}$, leaves columns invariant.
Therefore, $v^4 = Id$ on $E \times \{x,y\}$. It readily follows that $v^4 = 1$ for all $v \in V$, contradicting hypothesis i).

\medskip
{\bf A3): $G_R$ has 3 connected components and $G_C$ has 3 connected components.}

Let $z$ and $t$ be the two isolated vertices of $G_R$. By Lemma \ref{isoconn}, either both $x$ and $y$ are isolated in $G_C$, or none of them.

\smallskip
\noindent 
{\bf A3i): $x$ and $y$ are isolated in $G_C$.}
Then $E = \{z,d\}$ is the only edge of $G_C$. We have at least 9 orbits $\{t\} \times \{x\}$, $\{t\} \times \{y\}$, $\{z\} \times \{x\}$, $\{z\} \times \{y\}$, $\{x,y\} \times \{x\}$,
$\{x,y\} \times \{y\}$, $\{x,y\} \times E$, $\{z\} \times E$, $\{t\} \times E$.
By Lemma \ref{samecol}, $v(z,z)$ and $v(t,z)$ are in the same column, therefore $v$ restricted to $E \times E$ is either the identity or the double transposition
$((z,z),(z,t))((t,z),(t,t))$ (recall that row $z$ and row $t$ are invariant by our assumption A3)).
Similarly, $v$ restricted to $\{x,y\} \times \{x,y\}$ is either the identity or the double transposition $((x,x),(y,x))((x,y),(y,y))$.
On the other hand, we have no restriction on the square $\{x,y\} \times \{z,t\}$. 
It follows that $V$ is a subgroup of $K_{x,y}$, see Theorem \ref{S4xZ2xZ2alt}.
(Indeed there are 4 such subgroups, and $24=6 \cdot 4$ in total because of the choices for $x$ and $y$.)

\smallskip
\noindent 
{\bf A3ii): neither $x$ nor $y$ are isolated in $G_C$.} Here, $G_R = G_C$. There are at least 9 orbits 
$\{z\} \times \{z\}$, $\{z\} \times \{t\}$, $\{t\} \times \{z\}$, $\{t\} \times \{t\}$, $\{z\} \times E$, $\{t\} \times E$, $E \times \{z\}$, $E \times \{t\}$, $E \times E$, 
where $E = \{x,y\}$. By Corollary \ref{conn_comp_of_both}, there exist $\sigma , \tau \in S(E)$ such that $v=\sigma \otimes \tau$ in $E^{2}$.
Hence $v^2 = 1$ on $E \times E$. It readily follows that $v^2 = 1$ for all $v \in V$, which contradicts hypothesis i).

\smallskip
\noindent 
{\bf A4): $G_R$ has 3 connected components and $G_C$ has 4 connected components.}

Then $V$ has at least 12 orbits of sizes either $1$ or $2$,
so $V$ is not isomorphic to $S_4$.

\smallskip
\noindent 
{\bf A5): $G_R$ has 4 connected components and $G_C$ has 2 connected components.}

We then have to distinguish two possibilities.

\medskip
\noindent 
{\bf A5i): $G_C$ has an isolated vertex.}
We may then assume that $x$ is the isolated vertex of $G_C$
so that $E:=\{ y,z,t \}$ is the other connected component.
Therefore $V$ has at least 8 orbits, namely, $\{ x \} \times E$,
$\{ y \} \times E$, $\{ z \} \times E$, $\{ t \} \times E$, 
$\{ (z,x) \}$, $\{ (t,x) \}$, $\{ (x,x) \}$, and $\{ (y,x) \}$.
Then, since the sizes of these orbits are all $ \leq 3$, 
we have that $v^6=1$ for all  $v \in V$, so $V$ is not 
isomorphic to $S_4$.

\smallskip
\noindent 
{\bf A5ii): $G_C$ has no isolated vertices.}
We may assume that $E_C= \{ \{ x,y \}, \{ z,t \} \}$.
Then $V$ has at least 8 orbits, all of sizes $\leq 2$,
so $v^2=1$ for all $v \in V$, so $V$ is not 
isomorphic to $S_4$.

\smallskip
\noindent 
{\bf A6): $G_R$ has 4 connected components as well as $G_C$.}

The only possible case here is clearly $V = \{e\}$.

\bigskip
\noindent
{\bf B): $G_R \cup G_C$ is connected.}

We distinguish here two possibilities.

\medskip \noindent
{\bf B1): Either $G_R$ or $G_C$ are connected.}

We treat the case $G_R$ connected, the other one being 
entirely symmetrical. Then, by Lemma \ref{connected->complete},
$G_R$ is a complete graph and so Proposition \ref{homsigma} applies.
We now have two cases to consider.

\smallskip
\noindent 
{\bf B1i): $G_C$ has isolated vertices.}
Then, by Lemma \ref{isoconn}, $G_C$ is empty, and this, by
what has just been observed, implies that $v = \sigma_v \otimes 1$.
Therefore $V$ is a subgroup of $S_4 \otimes \{ 1 \}$, so
$V=S_4 \otimes \{ 1 \}$, as desired.

\smallskip
\noindent 
{\bf B1ii): $G_C$ has no isolated vertices.}
We distinguish two further subcases.

\smallskip
\noindent 
{\bf B1iia): $G_C$ is connected.}
Then, by Corollary \ref{conn_comp_of_both}, there is $\tau_v \in 
S_4$ such that $v = \sigma_v \otimes \tau_v$. 
Let 
\[
V_1 := \{  \sigma_v : v \in V \}, \; \;
V_2 := \{  \tau_v : v \in V \}.
\]
Let $w \in V$. Since $V$ is bicompatible, $v$ and $w$ are 
bicompatible, and this is easily seen to imply that
$\sigma_v \tau_w = \tau_w \sigma_v$. Therefore any element of
$V_1$ commutes with any element of $V_2$. Since $V \simeq S_4$
there are $v,w \in V$ such that $vw \neq wv$. Hence, either
$\sigma_v \sigma_w \neq \sigma_w \sigma_v$ or
$\tau_v \tau_w \neq \tau_w \tau_v$. Say 
$\sigma_v \sigma_w \neq \sigma_w \sigma_v$. Let
$H := \langle \sigma_v, \sigma_w \rangle $. Then $H \leq S_4$
and $H$ is not abelian. Therefore $H$ is isomorphic to either
$S_3$, or to the dihedral group of order $8$, $D_8$, or to
the alternating group $A_4$, or to $S_4$. Since $V_2$ commutes
with every element of $H$ we have that $V_2 \subseteq C_{S_4}(H)$
(the centralizer of $H$ in $S_4$). But, since $G_C$ is not empty, 
$V_2 \neq \{ 1 \}$. On the other hand, it is well known, and 
easy to see, that, $C_{S_4}(S_3)=C_{S_4}(S_4)=C_{S_4}(A_4)=\{ 1 \}$,
while $C_{S_4}(D_8) = \{ 1,g \}$, where $g \in S_4$ is a double
transposition. Hence $V_2 = \{ 1,g \}$. Therefore,
every element of $V_1$ commutes with $g$. But it easy to check 
that there are exactly $8$ elements in $S_4$ that commute with
a double transposition, so $|V_1| \leq 8$. Hence $|V| \leq |V_1| \cdot |V_2| \leq 16$, which contradicts our hypothesis i).
Similarly if $\tau_v \tau_w \neq \tau_w \tau_v$.

\medskip
\noindent 
{\bf B1iib): $G_C$ is not connected.}
Then $E_C$ is either $\{\{x,y\}, \{z,t\}\}$, or $\{\{x,z\},$ $\{y,t\}\}$ or $\{\{x,t\}, \{y,z\}\}$.
Let $v \in V$. Recall that, by Lemma \ref{samerow},
there is $\sigma_v \in S_4$ such that 
$v(a,b)$ is in row $\sigma_v(a)$ for all $(a,b) \in [4]^2$. 
Without loss of generality we may assume that $E_C = \{\{x,y\}, \{z,t\}\}$. Hence $V$ has at least two orbits, namely 
$[4] \times \{ x,y \}$, and $[4] \times \{ z,t \}$.
By Proposition \ref{homsigma}
the map $\sigma: V \to S_4$ mapping $v$ to $\sigma_v$ is a group homomorphism. 
It follows that $\ker\sigma$ is a normal subgroup of $S_4$, namely
$$
\ker\sigma \in \{\{e\}, S_4, A_4, V_4\} \ , 
$$
where $V_4 = \{e, (1,2)(3,4), (1,3)(2,4), (1,4)(2,3)\}$.
But $G_R$ is complete, hence for any $a,b \in [4]$ there are
$c,d \in [4]$ and $u \in V$ such that $u(a,c)=(b,d)$ and 
therefore $\sigma_u(a)=b$. In particular, this shows that
$|\sigma(V)| \geq 4$ and so implies that 
$\ker\sigma \in \{\{e\}, V_4\}$. 

Assume that $\ker \sigma =
\{ e \}$. Then $\sigma$ is an isomorphism between $V$ and $S_4$.
In particular, there are $v_i \in V$, for $i=1,2,3$, such that
$\sigma_{v_1}=(x,y)$, $\sigma_{v_2}=(y,z)$, and $\sigma_{v_3}=(z,t)$. 
In particular, the $v_i$'s are involutions. 

Since $v_1$ maps row $x$ to row $y$,
and $E_C = \{\{x,y\}, \{z,t\} \}$ we have that 
$v_1(x,x) \in \{ (y,x), (y,y) \}$. Similarly, since $v_1$ leaves
row $z$ invariant, $v_1(z,x) \in \{ (z,x), (z,y) \}$. If $v_1(x,x)=(y,x)$ then,
by Proposition \ref{change_col}, $v_1(x,b)=(y,b)$ for all $b 
\in [4]$, and therefore also $v_1(y,b)=(x,b)$ for all $b \in [4]$.
If $v_1(x,x)=(y,y)$ then,
by Proposition \ref{change_col}, $v(x,y)=(y,x)$, $v(x,z)=(y,t)$,
and $v(x,t)=(y,z)$ and therefore also, since $v_1$ is an involution, $v_1(y,y)=(x,x)$, $v_1(y,x)=(x,y)$, $v_1(y,t)=(x,z)$,
and $v_1(y,z)=(x,t)$. Similarly, we conclude that if $v_1(z,x)=
(z,x)$ then $v_1(z,b)=(z,b)$ for all $b \in [4]$, and hence
by Lemma \ref{samecol}, also $v_1(t,b)=(t,b)$ for all $b \in [4]$.
Finally, if $v_1(z,x)=(z,y)$ then by the same reasonings we 
conclude that $v_1(z,y)=(z,x)$, $v_1(z,z)=(z,t)$, and $v_1(z,t)=(z,z)$, and that $v_1(t,x)=(t,y)$, $v_1(t,y)=(t,x)$, $v_1(t,z)=(t,t)$, and $v_1(t,t)=(t,z)$. Hence 
there are only four possibilities for $v_1$ (combining the two cases for lines $x$, $y$ with those two for lines $z$, $t$), including $(x,y) \otimes 1$.

Similarly, there are four possibilities for $v_3$ (combining two cases for lines $x$, $y$ with other two for lines $z$, $t$), including $(z,t) \otimes 1$.

Finally, $v_2(y,x) \in \{ (z,x), (z,y) \}$. 
If $v_2(y,x)=(z,x)$, by Proposition \ref{change_col}, $v_2(y,b) = (z,b)$ for all $b \in [4]$ and, $v_2$ being an involution, $v_2(z,b) = (y,b)$ for all $b \in [4]$.
Moreover, by Lemma \ref{samecol}, $v_2(x,b) = (x,b)$ and $v_2(t,b)=(t,b)$, for all $b \in [4]$, i.e. $v_2 = (y,z) \otimes 1$.
If $v_2(y,x) = (z,y)$, then by Proposition \ref{change_col}, $v_2(y,y) = (z,x)$, $v_2 (y,z) = (z,t)$ and $v_2(y,t) = (z,z)$ and also 
$v_2(z,y)=(y,x)$, $v_2(z,x) = (y,y)$, $v_2 (z,t) = (y,z)$ and $v_2(z,z) = (y,t)$.
Moreover, by Lemma \ref{samecol}, $v_2(a,x) = (a,y)$, $v_2(a,y) = (a,x)$, $v_2(a,z) = (a,t)$ and $v_2(a,t) = (a,z)$, for $a \in \{x,t\}$.

Now, it is not difficult to check that each one of the two possible choices for $v_2$ just obtained is bicompatible only 
with $(x,y) \otimes 1$ (among the 4 possible $v_1$ as above).
Picking $v_2 = (y,z) \otimes1$ (and accordingly $v_1 = (x,y) \otimes 1$) would eventually lead, by similar arguments, to the unique possibility $v_3=(z,t) \otimes 1$, 
giving $V = S_4 \otimes 1$, which does not match with the running assumption that $G_C$ is nonempty.
However, the other possible choice for $v_2$ does not work either, as it does not satisfy the braid relation with $v_1 = (x,y) \otimes 1$. Therefore $\ker \sigma \neq \{ e \}$.

\medskip
Hence $\ker\sigma = V_2$.
Then it follows that $\sigma(V) \simeq V/V_2  \simeq S_3$.
But $\sigma(V) \leq S_4 = S(\{x,y,z,t\})$ and the only subgroups of $S_4$ which are isomorphic to $S_3$ are ${\rm Stab}(\{a\})$ for some $a \in \{x,y,z,t\} = [4]$.
But this would imply that $a$ is an isolated vertex of $G_R$, contradicting our assumption.

\medskip
\medskip \noindent
{\bf B2): Both $G_R$ and $G_C$ are disconnected.}
\medskip

\noindent 
{\bf B2a): Either $G_R$ or $G_C$ has isolated vertices.}
We may assume that $G_R$ has isolated vertices, the other
case being entirely analogous.
Then, since $G_R \cup G_C$ is connected by our running 
hypothesis B), no vertex can be isolated in both 
$G_R$ and $G_C$ so $G_R$ has at most two isolated vertices
(if $G_R$ had four isolated vertices then $G_C \cup G_R = G_C$
so $G_C$ would be connected).
Suppose that $G_R$ has two isolated vertices. We may assume 
that $E_R = \{  \{ z,t \}\}$. Then, by Lemma \ref{isoconn},
$\{ x,z \}, \{ x,t \}, \{ y,z \}, \{ y,t \} \notin E_C$ so
$G_R \cup G_C$ is not connected, a contradiction. Therefore
$G_R$ has only one isolated vertex. Say that $x$ is this 
unique isolated vertex. Then, by Lemma \ref{isoconn}, 
$\{ x \}$ is a connected component of $G_C$, but this is
impossible, as already observed above, because $G_R \cup G_C$ is
connected. 
\medskip

\noindent 
{\bf B2b): Both $G_R$ and $G_C$ have no isolated vertices.}
Then $G_R$ and $G_C$ each have two connected components of
size $2$. We may assume that $E_R = \{ \{ x,y \}, \{ z,t \}\}$. 
Then, since $G_R \cup G_C$ is connected, we may further assume 
that $E_C = \{ \{ x,t \}, \{ z,y \}\}$. Therefore $V$ has at 
least four orbits, each contained in either $\{ x,y \} \times 
\{ x,t \}$, or $ \{ x,y \} \times \{ z,y \}$, or $ \{ z,t \} \times \{ x,t \}$, or $ \{ z,t \} \times \{ z,y \}$.
Let $v \in V$. Then $v(x,x) \in \{ x,y \} \times 
\{ x,t \}$. We distinguish these four cases.

\medskip
\noindent 
{\bf $\alpha$): $v(x,x)=(x,x)$.}
Then, by Lemma \ref{samerow} (with $c=a=x$ and $b=y$)
$v(x,y)$ is in row $x$. Furthermore, by Proposition
\ref{change_col} (with $c=a=x$ and $b=y$) $v(x,y)$ is in column
$y$, so $v(x,y)=(x,y)$. Similarly, using Lemma \ref{samecol} and
Proposition \ref{change_col}, we conclude that
$v(t,x)=(t,x)$ and $v(t,y)=(t,y)$. Therefore $v(x,z) \in \{ (x,z),(y,y), (y,z) \}$.
We discuss these three further subcases separately.

\smallskip\noindent
{\bf $\alpha1)$: $v(x,z)=(x,z).$}
Then by Lemmas \ref{samerow} and \ref{samecol}, and 
Proposition \ref{change_col}, we conclude that $v$ 
leaves fixed (pointwise) rows $x$ and $t$. In particular,
$v(y,x) \in \{ (y,x), (y,t) \}$. If $v(y,x)=(y,x)$ then,
using similarly Lemmas \ref{samerow} and \ref{samecol}, and 
Proposition \ref{change_col}, we conclude that $v=1$.
If $v(y,x)=(y,t)$ we analogously conclude that 
$$v=((y,x),(y,t))((y,y),(y,z))((z,x),(z,t))((z,y),(z,z)) . $$

\smallskip\noindent
{\bf $\alpha2)$: $v(x,z)=(y,z).$}
Then, by arguments analogous to the ones in the previous case,
we conclude that $v(x,t)=(y,t)$, $v(t,z)=(z,z)$, and $v(t,t)=
(z,t)$. Hence $v(y,z) \in \{ (x,z), (y,y) \}$. 
If $v(y,z)=(x,z)$ then we conclude that 
$$v=((x,z),(y,z))((x,t),(y,t))((t,z),(z,z))((t,t),(z,t)) . $$
If $v(y,z)=(y,y)$
then we conclude that 
\begin{align*}
v=((y,x),(x,t),(y,t)) ((y,y),(x,z),(y,z)) ((z,x),(t,t),(z,t)) ((z,y),(t,z),(z,z)).
\end{align*}

\smallskip\noindent
{\bf $\alpha3)$: $v(x,z)=(y,y).$}
Then, by arguments similar to those presented above,
we conclude that $v(x,t)=(y,x)$, $v(t,z)=(z,y)$, and $v(t,t)=
(y,x)$. Hence $v(y,y) \in \{ (x,z), (y,z) \}$.
If $v(y,y)=(x,z)$ then we conclude that 
$$v=((x,z),(y,y)) ((x,t),(y,x))((z,x),(t,t))((z,y),(t,z)).$$
If $v(y,y)=(y,z)$
then we conclude that 
$$v=((y,x),(y,t),(x,t)) ((y,y),(y,z),(x,z)) ((z,x),(z,t),(t,t))((z,y),(z,z),(t,z)).$$

All in all, we have thus found 6 possible $v$'s in case $\alpha)$.

\medskip

\bigskip 
\medskip
\noindent 
{\bf $\beta$): $v(x,x)=(y,x)$.}
By arguing similarly to the above, $v(y,x) \in \{(y,t),(x,x),$ 
$(x,t)\}$.
We discuss these three further subcases separately.

\smallskip\noindent
{\bf $\beta1)$: $v(y,x)=(y,t).$}
Then $v(y,t) \in \{(x,t),(x,x)\}$.
If $v(y,t)=(x,t)$ we conclude that 
\begin{align*}
v& =((y,x),(y,t),(x,t),(x,x))((y,y),(y,z),(x,z),(x,y))((z,x),(z,t),(t,t),(t,x)) \\
& \qquad \times ((z,y),(z,z),(t,z),(t,y)).
\end{align*}
If $v(y,t)=(x,x)$ we conclude that 
\begin{align*}
v=((y,x),(y,t),(x,x))((y,y),(y,z),(x,y))((z,x),(z,t),(t,x))((z,y),(z,t),(t,x)).
\end{align*}

\smallskip\noindent
{\bf $\beta2)$: $v(y,x)=(x,x).$} 
Then $v(x,z) \in \{(y,z),(x,z)\}$.
If $v(x,z)=(y,z)$ we conclude that 
$$v=((x,y)(z,t)) \otimes 1 . $$
If $v(x,z)=(x,z)$ we conclude that 
$$v= ((y,x),(x,x))((y,y),(x,y))((z,x),(t,x))((z,y),(t,y)) . $$

\smallskip\noindent
{\bf $\beta3)$: $v(y,x)=(x,t).$}
Then $v(x,t) \in \{(x,x),(y,t)\}$.
If $v(x,t) = (x,x)$ we conclude that 
\begin{align*}
v=((y,x),(x,t),(x,x))((y,y),(x,z),(x,y)) ((z,x),(t,t),(t,x)) ((t,y),(z,y),(t,z)).
\end{align*}
If $v(x,t) = (y,t)$ we conclude that 
\begin{align*}
v &=((y,x),(x,t),(y,t),(x,x))((y,y),(x,z),(y,z),(x,y))((z,x),(t,t),(z,t),(t,x))\\
& \qquad \times ((z,y),(t,z),(z,z),(t,y)).
\end{align*}

So again, we have found 6 possible $v$'s in case $\beta)$.

\medskip
\noindent 
{\bf $\gamma$): $v(x,x)=(x,t)$.}
Then in the same way as in the two previous cases we conclude that
$v(x,t) \in \{ (x,x),(y,t),(y,x) \}$ and we obtain that either
$$v=1 \otimes ((x,t)(y,z)) , $$ or 
$$v=((x,x),(x,t))((x,y),(x,z)) ((t,x),(t,t))((t,y),(t,z))$$ 
(if $v(x,t)=(x,x)$), or
\begin{align*}
v & =((x,x),(x,t),(y,t),(y,x))((x,y),(x,z),(y,z),(y,y))
((t,x),(t,t),(z,t),(z,x)) \\
& \qquad \times ((t,y),(t,z),(z,z),(z,y)) , 
\end{align*}
or 
$$v=((x,x),(x,t),(y,t))((x,y),(x,z),(y,z))
((t,x),(t,t),(z,t))((t,y),(t,z),(z,z))$$ 
(if $v(x,t)=(y,t)$), or
\begin{align*}
v & =((x,x),(x,t),(y,x),(y,t))((x,y),(x,z),(y,y),(y,z))
((t,x),(t,t),(z,x),(z,t)) \\
& \qquad \times ((t,y),(t,z),(z,y),(z,z)),
\end{align*}
or
$$v=((x,x),(x,t),(y,x))((x,y),(x,z),(y,y)) ((t,x),(t,t),(z,x))((t,y),(t,z),(z,y))$$ 
(if $v(x,t)=(y,x)$).

\medskip
\noindent 
{\bf $\delta$): $v(x,x)=(y,t)$.}
In this case, one concludes analogously that $v$ is in
\begin{align*}
& \Big\{ ((x,x),(y,t))((x,y),(y,z))((t,x),(z,t))((t,y),(z,z)), \\
& ((x,x),(y,t))((x,y),(y,z))((x,z),(y,y))((x,t),(y,x))
((t,x),(z,t))((t,y),(z,z)) \\
& \qquad \times ((t,z),(z,y))((t,t),(z,x)), \\
& ((x,x),(y,t),(x,t))((x,y),(y,z),(x,z))
((t,x),(z,t),(t,t))((t,y),(z,z),(t,z)),  \\
& ((x,x),(y,t),(x,t),(y,x))((x,y),(y,z),(x,z),(y,y)) 
((t,x),(z,t),(t,t),(z,x)) \\
& \qquad \times ((t,y),(z,z),(t,z),(z,y)), \\
& ((x,x),(y,t),(y,x)) ((x,y),(y,z),(y,y))
((t,x),(z,t),(z,x)) ((t,y),(z,z),(z,y)), \\
& ((x,x),(y,t),(y,x),(x,t)) ((x,y),(y,z),(y,y),(x,z))
((t,x),(z,t),(z,x),(t,t)) \\
& \qquad \times ((t,y),(z,z),(z,y),(t,z)) \Big\}.
\end{align*}

We have therefore found $24$ possibilities for $v$
(see Fig.2) 
and they correspond exactly to the elements of $V$ as described 
in the third line of our thesis ii) (the six copies mentioned there correspond to the three possible choices for $E_R$ and the
two corresponding further choices for $E_C$).
So ii) again holds.

\[  
\beginpicture

\setcoordinatesystem units <0.28cm,0.28cm>
\setplotarea x from 6 to 7, y from 2 to 5

\setlinear

\plot -12 0 -12 4 /
\plot -11 0 -11 4 /
\plot -10 0 -10 4 /
\plot -9 0 -9 4 /
\plot -8 0 -8 4 /

\plot -12 0 -8 0 /
\plot -12 1 -8 1 /
\plot -12 2 -8 2 /
\plot -12 3 -8 3 /
\plot -12 4 -8 4 /

\plot -6 0 -6 4 /
\plot -5 0 -5 4 /
\plot -4 0 -4 4 /
\plot -3 0 -3 4 /
\plot -2 0 -2 4 /

\plot -6 0 -2 0 /
\plot -6 1 -2 1 /
\plot -6 2 -2 2 /
\plot -6 3 -2 3 /
\plot -6 4 -2 4 /

\plot -5.5 2.6 -2.5 2.6 /
\plot -4.5 2.4 -3.5 2.4 /
\plot -5.5 1.6 -2.5 1.6 /
\plot -4.5 1.4 -3.5 1.4 /

\plot 0 0 0 4 /
\plot 1 0 1 4 /
\plot 2 0 2 4 /
\plot 3 0 3 4 /
\plot 4 0 4 4 /

\plot 0 0 4 0 /
\plot 0 1 4 1 /
\plot 0 2 4 2 /
\plot 0 3 4 3 /
\plot 0 4 4 4 /

\plot 2.5 0.5 2.5 1.5 /
\plot 3.5 1.5 3.5 0.5 /
\plot 2.5 2.5 2.5 3.5 /
\plot 3.5 3.5 3.5 2.5 /

\plot 6 0 6 4 /
\plot 7 0 7 4 /
\plot 8 0 8 4 /
\plot 9 0 9 4 /
\plot 10 0 10 4 /

\plot 6 0 10 0 /
\plot 6 1 10 1 /
\plot 6 2 10 2 /
\plot 6 3 10 3 /
\plot 6 4 10 4 /

\plot 6.5 2.4 9.5 3.6 /
\arrow <0.15cm> [0.1,0.5] from  9.5 3.6 to 9.5 2.4 
\plot 9.5 2.4 6.5 2.4 /
\plot 7.5 2.6 8.5 3.4 /
\plot 8.5 3.4 8.5 2.6 /
\arrow <0.15cm> [0.1,0.5]  from 8.5 2.6 to 7.5 2.6 

\plot 6.5 1.6 9.5 0.4 /
\arrow <0.15cm> [0.1,0.5] from  9.5 0.4 to 9.5 1.6 
\plot 9.5 1.6 6.5 1.6 /
\plot 7.5 1.4 8.5 0.6 /
\plot 8.5 0.6 8.5 1.4 /
\arrow <0.15cm> [0.1,0.5]  from 8.5 1.4 to 7.5 1.4 

\plot 12 0 12 4 /
\plot 13 0 13 4 /
\plot 14 0 14 4 /
\plot 15 0 15 4 /
\plot 16 0 16 4 /

\plot 12 0 16 0 /
\plot 12 1 16 1 /
\plot 12 2 16 2 /
\plot 12 3 16 3 /
\plot 12 4 16 4 /

\plot 12.5 2.5 15.5 3.5 /
\plot 13.5 2.5 14.5 3.5 /
\plot 13.5 1.5 14.5 0.5 /
\plot 12.5 1.5 15.5 0.5 /

\plot 18 0 18 4 /
\plot 19 0 19 4 /
\plot 20 0 20 4 /
\plot 21 0 21 4 /
\plot 22 0 22 4 /

\plot 18 0 22 0 /
\plot 18 1 22 1 /
\plot 18 2 22 2 /
\plot 18 3 22 3 /
\plot 18 4 22 4 /

\plot 18.5 2.4 21.5 2.4 /
\arrow <0.15cm> [0.1,0.5]  from 21.5 2.4 to 21.5 3.6
\plot 21.5 3.6 18.5 2.4 /
\plot 19.5 2.6 20.5 2.6 /
\arrow <0.15cm> [0.1,0.5]  from 20.5 2.6 to 20.5 3.4
\plot 20.5 3.4 19.5 2.6 /

\plot 18.5 1.6 21.5 1.6 /
\arrow <0.15cm> [0.1,0.5]  from 21.5 1.6 to 21.5 0.4
\plot 21.5 0.4 18.5 1.6 /
\plot 19.5 1.4 20.5 1.4 /
\arrow <0.15cm> [0.1,0.5]  from 20.5 1.4 to 20.5 0.6
\plot 20.5 0.6 19.5 1.4 /

\plot -12 -6 -12 -2 /
\plot -11 -6 -11 -2 /
\plot -10 -6 -10 -2 /
\plot -9 -6 -9 -2 /
\plot -8 -6 -8 -2 /

\plot -12 -6 -8 -6 /
\plot -12 -5 -8 -5 /
\plot -12 -4 -8 -4 /
\plot -12 -3 -8 -3 /
\plot -12 -2 -8 -2 /

\plot -11.5 -3.6 -8.5 -3.6 /
\arrow <0.15cm> [0.1,0.5]  from -8.5 -3.6 to -8.5 -2.4 
\plot -8.5 -2.4 -11.5 -2.4 /
\plot -11.5 -2.4 -11.5 -3.5 /
\plot -10.5 -2.6 -9.5 -2.6 /
\arrow <0.15cm> [0.1,0.5]  from -9.5 -3.4 to -9.5 -2.6 
\plot -9.5 -3.4 -10.5 -3.4 /
\plot -10.5 -3.4 -10.5 -2.6 /

\plot -11.5 -4.4 -8.5 -4.4 /
\arrow <0.15cm> [0.1,0.5]  from -8.5 -4.4 to -8.5 -5.6 
\plot -8.5 -5.6 -11.5 -5.6 /
\plot -11.5 -5.6 -11.5 -4.4 /
\plot -10.5 -5.4 -9.5 -5.4 /
\arrow <0.15cm> [0.1,0.5]  from -9.5 -4.6 to -9.5 -5.4 
\plot -9.5 -4.6 -10.5 -4.6 /
\plot -10.5 -4.6 -10.5 -5.4 /

\plot -6 -6 -6 -2 /
\plot -5 -6 -5 -2 /
\plot -4 -6 -4 -2 /
\plot -3 -6 -3 -2 /
\plot -2 -6 -2 -2 /

\plot -6 -6 -2 -6 /
\plot -6 -5 -2 -5 /
\plot -6 -4 -2 -4 /
\plot -6 -3 -2 -3 /
\plot -6 -2 -2 -2 /

\plot -5.5 -3.6 -2.5 -3.6 /
\plot -2.5 -3.6 -5.5 -2.4 /
\arrow <0.15cm> [0.1,0.5]  from -5.5 -2.4 to -5.5 -3.6 
\arrow <0.15cm> [0.1,0.5]  from -4.5 -3.4 to -3.5 -3.4 
\plot -3.5 -3.4 -4.5 -2.6 /
\plot -4.5 -2.6 -4.5 -3.4 /

\plot -5.5 -4.4 -2.5 -4.4 /
\plot -2.5 -4.4 -5.5 -5.6 /
\arrow <0.15cm> [0.1,0.5]  from -5.5 -5.6 to -5.5 -4.4 
\arrow <0.15cm> [0.1,0.5]  from -4.5 -4.6 to -3.5 -4.6 
\plot -3.5 -4.6 -4.5 -5.6 /
\plot -4.5 -5.6 -4.5 -4.6 /

\plot 0 -6 0 -2 /
\plot 1 -6 1 -2 /
\plot 2 -6 2 -2 /
\plot 3 -6 3 -2 /
\plot 4 -6 4 -2 /

\plot 0 -6 4 -6 /
\plot 0 -5 4 -5 /
\plot 0 -4 4 -4 /
\plot 0 -3 4 -3 /
\plot 0 -2 4 -2 /

\plot 0.5 -5.5 0.5 -4.5 /
\plot 1.5 -5.5 1.5 -4.5 /
\plot 2.5 -5.5 2.5 -4.5 /
\plot 3.5 -5.5 3.5 -4.5 /

\plot 0.5 -3.5 0.5 -2.5 /
\plot 1.5 -3.5 1.5 -2.5 /
\plot 2.5 -3.5 2.5 -2.5 /
\plot 3.5 -3.5 3.5 -2.5 /

\plot 6 -6 6 -2 /
\plot 7 -6 7 -2 /
\plot 8 -6 8 -2 /
\plot 9 -6 9 -2 /
\plot 10 -6 10 -2 /

\plot 6 -6 10 -6 /
\plot 6 -5 10 -5 /
\plot 6 -4 10 -4 /
\plot 6 -3 10 -3 /
\plot 6 -2 10 -2 /

\plot 6.5 -2.5 6.5 -3.5 /
\plot 7.5 -2.5 7.5 -3.5 /
\plot 6.5 -4.5 6.5 -5.5 /
\plot 7.5 -4.5 7.5 -5.5 /

\plot 12 -6 12 -2 /
\plot 13 -6 13 -2 /
\plot 14 -6 14 -2 /
\plot 15 -6 15 -2 /
\plot 16 -6 16 -2 /

\plot 12 -6 16 -6 /
\plot 12 -5 16 -5 /
\plot 12 -4 16 -4 /
\plot 12 -3 16 -3 /
\plot 12 -2 16 -2 /

\plot 12.5 -3.6 15.5 -2.4 /
\plot 15.5 -2.4 12.5 -2.4 /
\arrow <0.15cm> [0.1,0.5] from 12.5 -2.4 to 12.5 -3.6
\plot 13.5 -3.4 14.5 -2.6 /
\arrow <0.15cm> [0.1,0.5] from 14.5 -2.6 to 13.5 -2.6 
\plot 13.5 -2.6 13.5 -3.4 /

\plot 12.5 -4.4 15.5 -5.6 /
\plot 15.5 -5.6 12.5 -5.6 /
\arrow <0.15cm> [0.1,0.5] from 12.5 -5.6 to 12.5 -4.4
\plot 13.5 -4.6 14.5 -5.4 /
\arrow <0.15cm> [0.1,0.5] from 14.5 -5.4 to 13.5 -5.4 
\plot 13.5 -5.4 13.5 -4.6 /

\plot 18 -6 18 -2 /
\plot 19 -6 19 -2 /
\plot 20 -6 20 -2 /
\plot 21 -6 21 -2 /
\plot 22 -6 22 -2 /

\plot 18 -6 22 -6 /
\plot 18 -5 22 -5 /
\plot 18 -4 22 -4 /
\plot 18 -3 22 -3 /
\plot 18 -2 22 -2 /

\plot 18.5 -2.5 18.5 -3.5 /
\plot 18.5 -3.5 21.5 -2.5 /
\arrow <0.15cm> [0.1,0.5] from 21.5 -2.5 to 21.5 -3.5 
\plot 21.5 -3.5  18.5 -2.5 /
\plot 19.5 -2.4 19.5 -3.6 /
\plot 19.5 -3.6 20.5 -2.4 /
\arrow <0.15cm> [0.1,0.5] from 20.5 -2.4 to 20.5 -3.6 
\plot 20.5 -3.6 19.5 -2.4 /

\plot 18.5 -5.5 18.5 -4.5 /
\plot 18.5 -4.5 21.5 -5.5 /
\arrow <0.15cm> [0.1,0.5] from 21.5 -5.5 to 21.5 -4.5 
\plot 21.5 -4.5  18.5 -5.5 /
\plot 19.5 -5.6 19.5 -4.4 /
\plot 19.5 -4.4 20.5 -5.6 /
\arrow <0.15cm> [0.1,0.5] from 20.5 -5.6 to 20.5 -4.4 
\plot 20.5 -4.4 19.5 -5.6 /

\plot -12 -12 -12 -8 /
\plot -11 -12 -11 -8 /
\plot -10 -12 -10 -8 /
\plot -9 -12 -9 -8 /
\plot -8 -12 -8 -8 /

\plot -12 -12 -8 -12 /
\plot -12 -11 -8 -11 /
\plot -12 -10 -8 -10 /
\plot -12 -9 -8 -9 /
\plot -12 -8 -8 -8 /

\plot -11.5 -8.4 -8.5 -8.4 /
\plot -10.5 -8.6 -9.5 -8.6 /
\plot -11.5 -9.4 -8.5 -9.4 /
\plot -10.5 -9.6 -9.5 -9.6 /
\plot -11.5 -10.4 -8.5 -10.4 /
\plot -10.5 -10.6 -9.5 -10.6 /
\plot -11.5 -11.4 -8.5 -11.4 /
\plot -10.5 -11.6 -9.5 -11.6 /

\plot -6 -12 -6 -8 /
\plot -5 -12 -5 -8 /
\plot -4 -12 -4 -8 /
\plot -3 -12 -3 -8 /
\plot -2 -12 -2 -8 /

\plot -6 -12 -2 -12 /
\plot -6 -11 -2 -11 /
\plot -6 -10 -2 -10 /
\plot -6 -9 -2 -9 /
\plot -6 -8 -2 -8 /

\plot -5.5 -8.4 -2.5 -8.4 /
\plot -4.5 -8.6 -3.5 -8.6 /
\plot -5.5 -11.4 -2.5 -11.4 /
\plot -4.5 -11.6 -3.5 -11.6 /

\plot 0 -12 0 -8 /
\plot 1 -12 1 -8 /
\plot 2 -12 2 -8 /
\plot 3 -12 3 -8 /
\plot 4 -12 4 -8 /

\plot 0 -12 4 -12 /
\plot 0 -11 4 -11 /
\plot 0 -10 4 -10 /
\plot 0 -9 4 -9 /
\plot 0 -8 4 -8 /

\plot 0.5 -8.4 3.5 -8.4 /
\plot 3.5 -8.4 3.5 -9.6 /
\plot 3.5 -9.6 0.5 -9.6 /
\arrow <0.15cm> [0.1,0.5] from 0.5 -9.6 to 0.5 -8.4 
\plot 1.5 -8.6 2.5 -8.6 /
\plot 2.5 -8.6 2.5 -9.4 /
\plot 2.5 -9.4 1.5 -9.4 /
\arrow <0.15cm> [0.1,0.5] from 1.5 -9.4 to 1.5 -8.6 

\arrow <0.15cm> [0.1,0.5] from 0.5 -10.4 to 0.5 -11.6 
\plot 0.5 -11.6 3.5 -11.6 /
\plot 3.5 -11.6 3.5 -10.4 /
\plot 3.5 -10.4 0.5 -10.4 /
\arrow <0.15cm> [0.1,0.5] from 1.5 -10.6 to 1.5 -11.4 
\plot 1.5 -11.4 2.5 -11.4 /
\plot 2.5 -11.4 2.5 -10.6 /
\plot 2.5 -10.6 1.5 -10.6 /

\plot 6 -12 6 -8 /
\plot 7 -12 7 -8 /
\plot 8 -12 8 -8 /
\plot 9 -12 9 -8 /
\plot 10 -12 10 -8 /

\plot 6 -12 10 -12 /
\plot 6 -11 10 -11 /
\plot 6 -10 10 -10 /
\plot 6 -9 10 -9 /
\plot 6 -8 10 -8 /

\plot 6.5 -8.4 9.5 -8.4 /
\arrow <0.15cm> [0.1,0.5] from 9.5 -8.4 to 9.5 -9.6 
\plot 9.5 -9.6 6.5 -8.4 /
\arrow <0.15cm> [0.1,0.5] from 7.5 -8.6 to 8.5 -8.6 
\plot 8.5 -8.6 8.5 -9.4 /
\plot 8.5 -9.4 7.5 -8.6 /

\plot 6.5 -11.6 9.5 -11.6 /
\arrow <0.15cm> [0.1,0.5] from 9.5 -11.6 to 9.5 -10.4 
\plot 9.5 -10.4 6.5 -11.6 /
\arrow <0.15cm> [0.1,0.5] from 7.5 -11.4 to 8.5 -11.4 
\plot 8.5 -11.4 8.5 -10.6 /
\plot 8.5 -10.6 7.5 -11.4 /

\plot 12 -12 12 -8 /
\plot 13 -12 13 -8 /
\plot 14 -12 14 -8 /
\plot 15 -12 15 -8 /
\plot 16 -12 16 -8 /

\plot 12 -12 16 -12 /
\plot 12 -11 16 -11 /
\plot 12 -10 16 -10 /
\plot 12 -9 16 -9 /
\plot 12 -8 16 -8 /

\arrow <0.15cm> [0.1,0.5] from 12.5 -8.6 to 15.5 -8.6 
\plot 15.5 -8.6 12.5 -9.4 /
\plot 12.5 -9.4 15.5 -9.4 /
\plot 15.5 -9.4 12.5 -8.6 /
\arrow <0.15cm> [0.1,0.5] from 13.5 -8.4 to 14.5 -8.4 
\plot 14.5 -8.4 13.5 -9.6 /
\plot 13.5 -9.6 14.5 -9.6 /
\plot 14.5 -9.6 13.5 -8.4 /

\arrow <0.15cm> [0.1,0.5] from 12.5 -11.4 to 15.5 -11.4 
\plot 15.5 -11.4 12.5 -10.6 /
\plot 12.5 -10.6 15.5 -10.6 /
\plot 15.5 -10.6 12.5 -11.4 /
\arrow <0.15cm> [0.1,0.5] from 13.5 -11.6 to 14.5 -11.6 
\plot 14.5 -11.6 13.5 -10.4 /
\plot 13.5 -10.4 14.5 -10.4 /
\plot 14.5 -10.4 13.5 -11.6 /

\plot 18 -12 18 -8 /
\plot 19 -12 19 -8 /
\plot 20 -12 20 -8 /
\plot 21 -12 21 -8 /
\plot 22 -12 22 -8 /

\plot 18 -12 22 -12 /
\plot 18 -11 22 -11 /
\plot 18 -10 22 -10 /
\plot 18 -9 22 -9 /
\plot 18 -8 22 -8 /

\plot 18.5 -8.4 21.5 -8.4 / 
\plot 21.5 -8.4 18.5 -9.6 /
\arrow <0.15cm> [0.1,0.5] from 18.5 -9.6 to 18.5 -8.4 
\plot 19.5 -8.6 20.5 -8.6 / 
\plot 20.5 -8.6 19.5 -9.4 /
\arrow <0.15cm> [0.1,0.5] from 19.5 -9.4 to 19.5 -8.6  

\plot 18.5 -11.6 21.5 -11.6 / 
\plot 21.5 -11.6 18.5 -10.4 /
\arrow <0.15cm> [0.1,0.5] from 18.5 -10.4 to 18.5 -11.6 
\plot 19.5 -11.4 20.5 -11.4 / 
\plot 20.5 -11.4 19.5 -10.6 /
\arrow <0.15cm> [0.1,0.5] from 19.5 -10.6 to 19.5 -11.4  

\plot -12 -18 -12 -14 /
\plot -11 -18 -11 -14 /
\plot -10 -18 -10 -14 /
\plot -9 -18 -9 -14 /
\plot -8 -18 -8 -14 /

\plot -12 -18 -8 -18 /
\plot -12 -17 -8 -17 /
\plot -12 -16 -8 -16 /
\plot -12 -15 -8 -15 /
\plot -12 -14 -8 -14 /

\plot -11.5 -14.5 -8.5 -15.5 /
\plot -10.5 -14.5 -9.5 -15.5 /
\plot -11.5 -17.5 -8.5 -16.5 /
\plot -10.5 -17.5 -9.5 -16.5 /

\plot -6 -18 -6 -14 /
\plot -5 -18 -5 -14 /
\plot -4 -18 -4 -14 /
\plot -3 -18 -3 -14 /
\plot -2 -18 -2 -14 /

\plot -6 -18 -2 -18 /
\plot -6 -17 -2 -17 /
\plot -6 -16 -2 -16 /
\plot -6 -15 -2 -15 /
\plot -6 -14 -2 -14 /

\plot -5.5 -14.5 -2.5 -15.5 /
\plot -4.5 -14.5 -3.5 -15.5 /
\plot -5.5 -15.5 -2.5 -14.5 /
\plot -4.5 -15.5 -3.5 -14.5 /
\plot -5.5 -16.5 -2.5 -17.5 /
\plot -4.5 -16.5 -3.5 -17.5 /
\plot -5.5 -17.5 -2.5 -16.5 /
\plot -4.5 -17.5 -3.5 -16.5 /

\plot 0 -18 0 -14 /
\plot 1 -18 1 -14 /
\plot 2 -18 2 -14 /
\plot 3 -18 3 -14 /
\plot 4 -18 4 -14 /

\plot 0 -18 4 -18 /
\plot 0 -17 4 -17 /
\plot 0 -16 4 -16 /
\plot 0 -15 4 -15 /
\plot 0 -14 4 -14 /

\plot 0.5 -14.4 3.5 -15.6 /
\arrow <0.15cm> [0.1,0.5] from 3.5 -15.6 to 3.5 -14.4 
\plot 3.5 -14.4 0.5 -14.4 /
\plot 1.5 -14.6 2.5 -15.4 /
\arrow <0.15cm> [0.1,0.5] from 2.5 -15.4 to 2.5 -14.6 
\plot 2.5 -14.6 1.5 -14.6 /

\plot 0.5 -17.6 3.5 -17.6 /
\arrow <0.15cm> [0.1,0.5] from 3.5 -16.4 to 3.5 -17.6 
\plot 3.5 -16.4 0.5 -17.6 /
\plot 1.5 -17.4 2.5 -17.4 /
\arrow <0.15cm> [0.1,0.5] from 2.5 -16.6 to 2.5 -17.6 
\plot 2.5 -16.6 1.5 -17.4 /

\plot 6 -18 6 -14 /
\plot 7 -18 7 -14 /
\plot 8 -18 8 -14 /
\plot 9 -18 9 -14 /
\plot 10 -18 10 -14 /

\plot 6 -18 10 -18 /
\plot 6 -17 10 -17 /
\plot 6 -16 10 -16 /
\plot 6 -15 10 -15 /
\plot 6 -14 10 -14 /

\plot 6.5 -14.5 9.5 -15.5 /
\arrow <0.15cm> [0.1,0.5] from 9.5 -15.5 to 9.5 -14.5 
\plot 9.5 -14.5 6.5 -15.5 /
\plot 6.5 -15.5 6.5 -14.5 /
\plot 7.5 -14.4 8.5 -15.6 /
\arrow <0.15cm> [0.1,0.5] from 8.5 -15.6 to 8.5 -14.4 
\plot 8.5 -14.4 7.5 -15.6 /
\plot 7.5 -15.6 7.5 -14.4 /

\plot 6.5 -16.5 9.5 -17.5 /
\arrow <0.15cm> [0.1,0.5] from 9.5 -16.5 to 9.5 -17.5 
\plot 9.5 -16.5 6.5 -17.5 /
\plot 6.5 -17.5 6.5 -16.5 /
\plot 7.5 -16.4 8.5 -17.6 /
\arrow <0.15cm> [0.1,0.5] from 8.5 -16.4 to 8.5 -17.6 
\plot 8.5 -16.4 7.5 -17.6 /
\plot 7.5 -17.6 7.5 -16.4 /

\plot 12 -18 12 -14 /
\plot 13 -18 13 -14 /
\plot 14 -18 14 -14 /
\plot 15 -18 15 -14 /
\plot 16 -18 16 -14 /

\plot 12 -18 16 -18 /
\plot 12 -17 16 -17 /
\plot 12 -16 16 -16 /
\plot 12 -15 16 -15 /
\plot 12 -14 16 -14 /

\plot 12.5 -14.4 15.5 -15.6 /
\plot 15.5 -15.6 12.5 -15.6 /
\arrow <0.15cm> [0.1,0.5] from 12.5 -15.6 to 12.5 -14.4 
\plot 13.5 -14.6 14.5 -15.4 /
\arrow <0.15cm> [0.1,0.5] from 14.5 -15.4 to 13.5 -15.4 
\plot 13.5 -15.4 13.5 -14.6 /

\plot 12.5 -17.6 15.5 -16.4 /
\plot 15.5 -16.4 12.5 -16.4 /
\arrow <0.15cm> [0.1,0.5] from 12.5 -16.4 to 12.5 -17.6 
\plot 13.5 -17.4 14.5 -16.6 /
\arrow <0.15cm> [0.1,0.5] from 14.5 -16.6 to 13.5 -16.6 
\plot 13.5 -16.6 13.5 -17.4 /

\plot 18 -18 18 -14 /
\plot 19 -18 19 -14 /
\plot 20 -18 20 -14 /
\plot 21 -18 21 -14 /
\plot 22 -18 22 -14 /

\plot 18 -18 22 -18 /
\plot 18 -17 22 -17 /
\plot 18 -16 22 -16 /
\plot 18 -15 22 -15 /
\plot 18 -14 22 -14 /

\plot 18.5 -14.6 21.5 -15.4 /
\plot 21.5 -15.4 18.5 -15.4 /
\plot 18.5 -15.4 21.5 -14.6 /
\arrow <0.15cm> [0.1,0.5] from 21.5 -14.6 to 18.5 -14.6 
\plot 19.5 -14.4 20.5 -15.6 /
\plot 20.5 -15.6 19.5 -15.6 /
\plot 19.5 -15.6 20.5 -14.4 /
\arrow <0.15cm> [0.1,0.5] from 20.5 -14.4 to 19.5 -14.4 

\plot 18.5 -16.6 21.5 -17.4 /
\arrow <0.15cm> [0.1,0.5] from  21.5 -17.4 to 18.5 -17.4
\plot 18.5 -17.4 21.5 -16.6 /
\plot 21.5 -16.6 18.5 -16.6 / 
\plot 19.5 -16.4 20.5 -17.6 /
\arrow <0.15cm> [0.1,0.5] from  20.5 -17.6 to 19.5 -17.6 
\plot 19.5 -17.6 20.5 -16.4 /
\plot 20.5 -16.4 19.5 -16.4 /

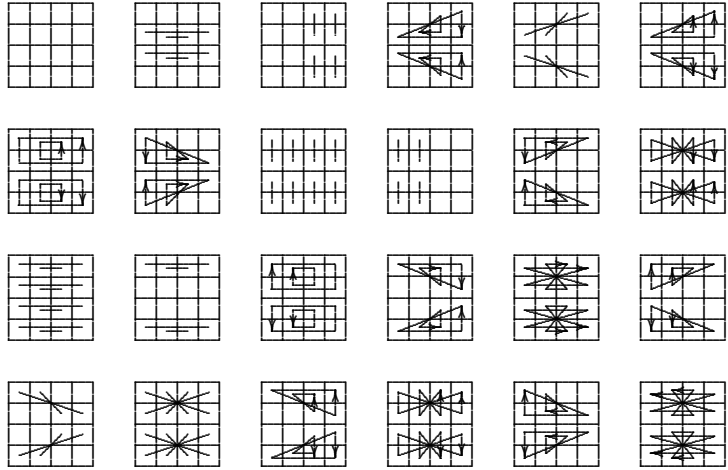
\put {The 24 elements $v \in V \subset S([4]^2)$ described in case B2b.} at 5 -22
\put{Rows and columns are indexed by $x,y,z,t$ in this order.} at 5 -24

\put{\bf{Fig.2}} at 5, -27

\endpicture 
\]

\newpage 

Conversely, assume that ii) holds. All the groups in part ii)
are isomorphic to $S_4$ by definition. It is easy to see that
both $1 \otimes S_4$ and $S_4 \otimes 1$ are bicompatible
subgroups. The bicompatibility of
the groups $N_{i,j}$ is established in the proof of
Theorem \ref{S4xZ2xZ2alt} while that of the groups 
$\widetilde{T}_{a,b}$ in the proof of Theorem \ref{S4}.
\end{proof}

\section{Outlook}\label{Outlook}

We present in this brief final section some directions for
future research and some open problems that arise naturally
from the results presented in this work.

In \cite{AJS18} other tables are presented. In particular, for 
the subgroups of ${\rm Out}(\O_2)$ which are maximal in $\pi(\lambda^{-1}(\P_2^5))$, and those in ${\rm Out}(\O_3)$
which are maximal in $\pi(\lambda^{-1}(\P_3^3))$ 
(see \cite[Tables 3 and 4]{AJS18}).
We feel that a natural next step would be to find explicitly
these subgroups. One can notice that the groups described in
Section 5 are not all Coxeter groups. Thus it seems to be difficult
to predict, given $n$ and $t$, which groups appear as subgroups 
of ${\rm Out}(\O_n)$ which are maximal in $\pi(\lambda^{-1}(\P_n^t))$. Related to this last question
is the one of classifying all the bicompatible 
subgroups of $S([n]^2)$ which are isomorphic to $S_n$
and also to their image in $\pi(\lambda^{-1}(\P_n^2))$,
which we feel might be doable.

Many of the generators of the groups studied in this work
arise from stable permutations of $[n]^2$ of rank $1$ and 
cycle-type $(2,...,2)$. The computations that we have carried out
for $n=4$ and cycle-types $(2,2)$ and $(2,2,2)$ (see Figures
3 and 4) show that there are 60 and 144 such
permutations and they seem to have quite an orderly 
structure. Thus it could be a promising avenue of research 
to try to classify all such stable permutations of rank $1$
(as was done for stable permutations of rank $1$ and cycle-type
$(r)$ in \cite[Thm. 3.1]{BCN}).

Finally, we feel that many of the results in Section 
\ref{horiz_vert} should be generalizable, both to $[n]^t$ 
for $t>2$, as well as to more general kind of permutations of
$[n]^2$.

\newpage

	\[  \beginpicture
		
	\setcoordinatesystem units <0.6cm,0.6cm>
	\setplotarea x from 7 to 8, y from 4 to 7
	
	\setlinear
	
	\plot 0 0 0 4 /
	\plot 1 0 1 4 /
	\plot 2 0 2 4 /
	\plot 3 0 3 4 /
	\plot 4 0 4 4 /
	
	\plot 0 0 4 0 /
	\plot 0 1 4 1 /
	\plot 0 2 4 2 /
	\plot 0 3 4 3 /
	\plot 0 4 4 4 /

	\plot 1.5 3.6 2.4 3.6 /
	\plot 2.6 3.6 3.5 3.6 /
	\plot 1.5 3.4 3.5 3.4 /
	
	\plot 0.5 2.6 2.4 2.6 /
	\plot 2.6 2.6 3.5 2.6 /
	\plot 0.5 2.4 3.5 2.4 /
	
	\plot 0.5 1.6 1.4 1.6 /
	\plot 1.6 1.6 3.5 1.6 /
	\plot 0.5 1.4 3.5 1.4 /
	
	\plot 0.5 0.6 1.4 0.6 /
	\plot 1.6 0.6 2.5 0.6 /
	\plot 0.5 0.4 2.5 0.4 /
	
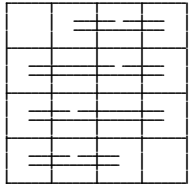
	\put {The 12 horizontal stable transpositions of rank 1 in $S([4]^2)$} at 2 -1.5
	
	\endpicture \]

\[ 
\beginpicture
\setcoordinatesystem units <0.6cm,0.6cm>
\setplotarea x from 7 to 8, y from 4 to 7

\setlinear

\plot 0 0 0 4 /
\plot 1 0 1 4 /
\plot 2 0 2 4 /
\plot 3 0 3 4 /
\plot 4 0 4 4 /

\plot 0 0 4 0 /
\plot 0 1 4 1 /
\plot 0 2 4 2 /
\plot 0 3 4 3 /
\plot 0 4 4 4 /

\plot 3.4 1.5 3.4 2.4 /
\plot 3.4 2.6 3.4 3.5 /
\plot 3.6 1.5 3.6 3.5 /

\plot 2.4 0.5 2.4 2.4 /
\plot 2.4 2.6 2.4 3.5 /
\plot 2.6 0.5 2.6 3.5 /

\plot 1.4 0.5 1.4 1.4 /
\plot 1.4 1.6 1.4 3.5 /
\plot 1.6 0.5 1.6 3.5 /

\plot 0.4 0.5 0.4 1.4 /
\plot 0.4 1.6 0.4 2.5 /
\plot 0.6 0.5 0.6 2.5 /

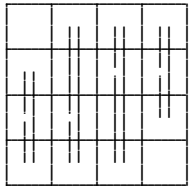
\put {The 12 vertical stable transpositions of rank 1 in $S([4]^2)$} at 2 -1.5

\endpicture \]

\[ 
\beginpicture
\setcoordinatesystem units <0.6cm,0.6cm>
\setplotarea x from 7 to 8, y from 4 to 7

\setlinear

\color{black} 

\plot 0 0 0 4 /
\plot 1 0 1 4 /
\plot 2 0 2 4 /
\plot 3 0 3 4 /
\plot 4 0 4 4 /

\plot 0 0 4 0 /
\plot 0 1 4 1 /
\plot 0 2 4 2 /
\plot 0 3 4 3 /
\plot 0 4 4 4 /

\plot 0.5 0.6 1.4 1.5 /
\plot 1.6 1.7 3.5 3.6 /
\plot 0.5 0.3 2.5 2.3 /
\plot 2.6 2.4 3.5 3.3 /

\plot 0.6 1.4 1.4 0.6 /
\plot 0.6 2.4 2.4 0.6 /
\plot 1.6 3.6 3.6 1.6 /
\plot 2.6 3.6 3.6 2.6 /

\plot 0.5 2.6 3.5 1.5 /
\plot 2.6 0.5 1.5 3.5 /
\plot 0.5 1.5 3.5 2.5 /
\plot 1.5 0.5 2.5 3.5 /

\color{black} 
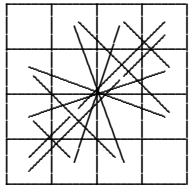
\put {The 12 remaining stable transpositions of rank 1 in $S([4]^2)$} at 2 -1.5

\put{\bf{Fig.3}} at 2, -4

\endpicture \]

\newpage

	\[  \beginpicture
	
	\setcoordinatesystem units <0.28cm,0.28cm>
	\setplotarea x from 6 to 7, y from 2 to 5
	
	\setlinear
	
	\plot -12 0 -12 4 /
	\plot -11 0 -11 4 /
	\plot -10 0 -10 4 /
	\plot -9 0 -9 4 /
	\plot -8 0 -8 4 /
	
	\plot -12 0 -8 0 /
	\plot -12 1 -8 1 /
	\plot -12 2 -8 2 /
	\plot -12 3 -8 3 /
	\plot -12 4 -8 4 /
	
	\plot -9.5 0.5 -8.5 0.5 /
	\plot -9.5 1.5 -8.5 1.5 /
	
	\plot -6 0 -6 4 /
	\plot -5 0 -5 4 /
	\plot -4 0 -4 4 /
	\plot -3 0 -3 4 /
	\plot -2 0 -2 4 /
	
	\plot -6 0 -2 0 /
	\plot -6 1 -2 1 /
	\plot -6 2 -2 2 /
	\plot -6 3 -2 3 /
	\plot -6 4 -2 4 /
	
	\plot -3.5 0.5 -3.5 1.5 /
	\plot -2.5 0.5 -2.5 1.5 /
	
	\plot 0 0 0 4 /
	\plot 1 0 1 4 /
	\plot 2 0 2 4 /
	\plot 3 0 3 4 /
	\plot 4 0 4 4 /
	
	\plot 0 0 4 0 /
	\plot 0 1 4 1 /
	\plot 0 2 4 2 /
	\plot 0 3 4 3 /
	\plot 0 4 4 4 /
	
	\plot 2.5 0.5 3.5 1.5 /
	\plot 2.5 1.5 3.5 0.5 /
	
	\plot 6 0 6 4 /
	\plot 7 0 7 4 /
	\plot 8 0 8 4 /
	\plot 9 0 9 4 /
	\plot 10 0 10 4 /
	
	\plot 6 0 10 0 /
	\plot 6 1 10 1 /
	\plot 6 2 10 2 /
	\plot 6 3 10 3 /
	\plot 6 4 10 4 /
	
	\plot 6.5 0.5 7.5 0.5 /
	\plot 6.5 1.5 7.5 1.5 /
	
	\plot 12 0 12 4 /
	\plot 13 0 13 4 /
	\plot 14 0 14 4 /
	\plot 15 0 15 4 /
	\plot 16 0 16 4 /
	
	\plot 12 0 16 0 /
	\plot 12 1 16 1 /
	\plot 12 2 16 2 /
	\plot 12 3 16 3 /
	\plot 12 4 16 4 /
	
	\plot 12.5 0.5 12.5 1.5 /
	\plot 13.5 0.5 13.5 1.5 /
	
	\plot 18 0 18 4 /
	\plot 19 0 19 4 /
	\plot 20 0 20 4 /
	\plot 21 0 21 4 /
	\plot 22 0 22 4 /
	
	\plot 18 0 22 0 /
	\plot 18 1 22 1 /
	\plot 18 2 22 2 /
	\plot 18 3 22 3 /
	\plot 18 4 22 4 /
	
	\plot 18.5 0.5 19.5 1.5 /
	\plot 18.5 1.5 19.5 0.5 /
	
	\plot -12 -6 -12 -2 /
	\plot -11 -6 -11 -2 /
	\plot -10 -6 -10 -2 /
	\plot -9 -6 -9 -2 /
	\plot -8 -6 -8 -2 /
	
	\plot -12 -6 -8 -6 /
	\plot -12 -5 -8 -5 /
	\plot -12 -4 -8 -4 /
	\plot -12 -3 -8 -3 /
	\plot -12 -2 -8 -2 /
	
	\plot -10.5 -5.5 -9.5 -5.5 /
	\plot -8.5 -4.5 -8.5 -3.5 /
	
	\plot -6 -6 -6 -2 /
	\plot -5 -6 -5 -2 /
	\plot -4 -6 -4 -2 /
	\plot -3 -6 -3 -2 /
	\plot -2 -6 -2 -2 /
	
	\plot -6 -6 -2 -6 /
	\plot -6 -5 -2 -5 /
	\plot -6 -4 -2 -4 /
	\plot -6 -3 -2 -3 /
	\plot -6 -2 -2 -2 /
	
	\plot -4.5 -5.5 -4.5 -4.5 /
	\plot -3.5 -3.5 -2.5 -3.5 /
	
	\plot 0 -6 0 -2 /
	\plot 1 -6 1 -2 /
	\plot 2 -6 2 -2 /
	\plot 3 -6 3 -2 /
	\plot 4 -6 4 -2 /
	
	\plot 0 -6 4 -6 /
	\plot 0 -5 4 -5 /
	\plot 0 -4 4 -4 /
	\plot 0 -3 4 -3 /
	\plot 0 -2 4 -2 /
	
	\plot 0.5 -5.5 0.5 -4.5 /
	\plot 2.5 -3.5 3.5 -3.5 /
	
	\plot 6 -6 6 -2 /
	\plot 7 -6 7 -2 /
	\plot 8 -6 8 -2 /
	\plot 9 -6 9 -2 /
	\plot 10 -6 10 -2 /
	
	\plot 6 -6 10 -6 /
	\plot 6 -5 10 -5 /
	\plot 6 -4 10 -4 /
	\plot 6 -3 10 -3 /
	\plot 6 -2 10 -2 /
	
	\plot 8.5 -5.5 8.5 -3.5 /
	\plot 7.5 -4.5 9.5 -4.5 /
	
	\plot 12 -6 12 -2 /
	\plot 13 -6 13 -2 /
	\plot 14 -6 14 -2 /
	\plot 15 -6 15 -2 /
	\plot 16 -6 16 -2 /
	
	\plot 12 -6 16 -6 /
	\plot 12 -5 16 -5 /
	\plot 12 -4 16 -4 /
	\plot 12 -3 16 -3 /
	\plot 12 -2 16 -2 /
	
	\plot 13.5 -4.5 14.5 -4.5 /
	\plot 13.5 -3.5 14.5 -3.5 /
	
	\plot 18 -6 18 -2 /
	\plot 19 -6 19 -2 /
	\plot 20 -6 20 -2 /
	\plot 21 -6 21 -2 /
	\plot 22 -6 22 -2 /
	
	\plot 18 -6 22 -6 /
	\plot 18 -5 22 -5 /
	\plot 18 -4 22 -4 /
	\plot 18 -3 22 -3 /
	\plot 18 -2 22 -2 /
	
	\plot 19.5 -4.5 19.5 -3.5 /
	\plot 20.5 -4.5 20.5 -3.5 /
	
	\plot -12 -12 -12 -8 /
	\plot -11 -12 -11 -8 /
	\plot -10 -12 -10 -8 /
	\plot -9 -12 -9 -8 /
	\plot -8 -12 -8 -8 /
	
	\plot -12 -12 -8 -12 /
	\plot -12 -11 -8 -11 /
	\plot -12 -10 -8 -10 /
	\plot -12 -9 -8 -9 /
	\plot -12 -8 -8 -8 /
	
	\plot -10.5 -10.5 -9.5 -9.5 /
	\plot -10.5 -9.5 -9.5 -10.5 /
	
	\plot -6 -12 -6 -8 /
	\plot -5 -12 -5 -8 /
	\plot -4 -12 -4 -8 /
	\plot -3 -12 -3 -8 /
	\plot -2 -12 -2 -8 /
	
	\plot -6 -12 -2 -12 /
	\plot -6 -11 -2 -11 /
	\plot -6 -10 -2 -10 /
	\plot -6 -9 -2 -9 /
	\plot -6 -8 -2 -8 /
	
	\plot -4.5 -11.5 -2.5 -11.5 /
	\plot -4.5 -9.5 -2.5 -9.5 /
	
	\plot 0 -12 0 -8 /
	\plot 1 -12 1 -8 /
	\plot 2 -12 2 -8 /
	\plot 3 -12 3 -8 /
	\plot 4 -12 4 -8 /
	
	\plot 0 -12 4 -12 /
	\plot 0 -11 4 -11 /
	\plot 0 -10 4 -10 /
	\plot 0 -9 4 -9 /
	\plot 0 -8 4 -8 /
	
	\plot 1.5 -11.5 1.5 -9.5 /
	\plot 3.5 -11.5 3.5 -9.5 /
	
	\plot 6 -12 6 -8 /
	\plot 7 -12 7 -8 /
	\plot 8 -12 8 -8 /
	\plot 9 -12 9 -8 /
	\plot 10 -12 10 -8 /
	
	\plot 6 -12 10 -12 /
	\plot 6 -11 10 -11 /
	\plot 6 -10 10 -10 /
	\plot 6 -9 10 -9 /
	\plot 6 -8 10 -8 /
	
	\plot 7.5 -11.5 9.5 -9.5 /
	\plot 7.5 -9.5 9.5 -11.5 /
	
	\plot 12 -12 12 -8 /
	\plot 13 -12 13 -8 /
	\plot 14 -12 14 -8 /
	\plot 15 -12 15 -8 /
	\plot 16 -12 16 -8 /
	
	\plot 12 -12 16 -12 /
	\plot 12 -11 16 -11 /
	\plot 12 -10 16 -10 /
	\plot 12 -9 16 -9 /
	\plot 12 -8 16 -8 /
	
	\plot 12.5 -9.5 14.5 -9.5 /
	\plot 15.5 -8.5 15.5 -10.5 /
	
	\plot 18 -12 18 -8 /
	\plot 19 -12 19 -8 /
	\plot 20 -12 20 -8 /
	\plot 21 -12 21 -8 /
	\plot 22 -12 22 -8 /
	
	\plot 18 -12 22 -12 /
	\plot 18 -11 22 -11 /
	\plot 18 -10 22 -10 /
	\plot 18 -9 22 -9 /
	\plot 18 -8 22 -8 /
	
	\plot 18.5 -11.5 20.5 -11.5 /
	\plot 18.5 -9.5 20.5 -9.5 /
	
	\plot -12 -18 -12 -14 /
	\plot -11 -18 -11 -14 /
	\plot -10 -18 -10 -14 /
	\plot -9 -18 -9 -14 /
	\plot -8 -18 -8 -14 /
	
	\plot -12 -18 -8 -18 /
	\plot -12 -17 -8 -17 /
	\plot -12 -16 -8 -16 /
	\plot -12 -15 -8 -15 /
	\plot -12 -14 -8 -14 /
	
	\plot -11.5 -16.5 -8.5 -16.5 /
	\plot -11.5 -15.5 -8.5 -15.5 /
	
	\plot -6 -18 -6 -14 /
	\plot -5 -18 -5 -14 /
	\plot -4 -18 -4 -14 /
	\plot -3 -18 -3 -14 /
	\plot -2 -18 -2 -14 /
	
	\plot -6 -18 -2 -18 /
	\plot -6 -17 -2 -17 /
	\plot -6 -16 -2 -16 /
	\plot -6 -15 -2 -15 /
	\plot -6 -14 -2 -14 /
	
	\plot -5.5 -16.5 -5.5 -15.5 /
	\plot -4.5 -17.5 -3.5 -17.5 /
	
	\plot 0 -18 0 -14 /
	\plot 1 -18 1 -14 /
	\plot 2 -18 2 -14 /
	\plot 3 -18 3 -14 /
	\plot 4 -18 4 -14 /
	
	\plot 0 -18 4 -18 /
	\plot 0 -17 4 -17 /
	\plot 0 -16 4 -16 /
	\plot 0 -15 4 -15 /
	\plot 0 -14 4 -14 /
	
	\plot 0.5 -16.5 0.5 -15.5 /
	\plot 3.5 -16.5 3.5 -15.5 /
	
	\plot 6 -18 6 -14 /
	\plot 7 -18 7 -14 /
	\plot 8 -18 8 -14 /
	\plot 9 -18 9 -14 /
	\plot 10 -18 10 -14 /
	
	\plot 6 -18 10 -18 /
	\plot 6 -17 10 -17 /
	\plot 6 -16 10 -16 /
	\plot 6 -15 10 -15 /
	\plot 6 -14 10 -14 /
	
	\plot 6.5 -16.5 9.5 -15.5 /
	\plot 6.5 -15.5 9.5 -16.5 /
	
	\plot 12 -18 12 -14 /
	\plot 13 -18 13 -14 /
	\plot 14 -18 14 -14 /
	\plot 15 -18 15 -14 /
	\plot 16 -18 16 -14 /
	
	\plot 12 -18 16 -18 /
	\plot 12 -17 16 -17 /
	\plot 12 -16 16 -16 /
	\plot 12 -15 16 -15 /
	\plot 12 -14 16 -14 /
	
	\plot 12.5 -17.5 12.5 -15.5 /
	\plot 13.5 -16.5 15.5 -16.5 /
	
	\plot 18 -18 18 -14 /
	\plot 19 -18 19 -14 /
	\plot 20 -18 20 -14 /
	\plot 21 -18 21 -14 /
	\plot 22 -18 22 -14 /
	
	\plot 18 -18 22 -18 /
	\plot 18 -17 22 -17 /
	\plot 18 -16 22 -16 /
	\plot 18 -15 22 -15 /
	\plot 18 -14 22 -14 /
	
	\plot 18.5 -17.5 18.5 -15.5 /
	\plot 20.5 -17.5 20.5 -15.5 /
	
	\plot -12 -24 -12 -20 /
	\plot -11 -24 -11 -20 /
	\plot -10 -24 -10 -20 /
	\plot -9 -24 -9 -20 /
	\plot -8 -24 -8 -20 /
	
	\plot -12 -24 -8 -24 /
	\plot -12 -23 -8 -23 /
	\plot -12 -22 -8 -22 /
	\plot -12 -21 -8 -21 /
	\plot -12 -20 -8 -20 /
	
	\plot -11.5 -23.5 -9.5 -21.5 /
	\plot -11.5 -21.5 -9.5 -23.5 /
	
	\plot -6 -24 -6 -20 /
	\plot -5 -24 -5 -20 /
	\plot -4 -24 -4 -20 /
	\plot -3 -24 -3 -20 /
	\plot -2 -24 -2 -20 /
	
	\plot -6 -24 -2 -24 /
	\plot -6 -23 -2 -23 /
	\plot -6 -22 -2 -22 /
	\plot -6 -21 -2 -21 /
	\plot -6 -20 -2 -20 /
	
	\plot -5.5 -23.5 -4.5 -23.5 /
	\plot -2.5 -21.5 -2.5 -20.5 /
	
	\plot 0 -24 0 -20 /
	\plot 1 -24 1 -20 /
	\plot 2 -24 2 -20 /
	\plot 3 -24 3 -20 /
	\plot 4 -24 4 -20 /
	
	\plot 0 -24 4 -24 /
	\plot 0 -23 4 -23 /
	\plot 0 -22 4 -22 /
	\plot 0 -21 4 -21 /
	\plot 0 -20 4 -20 /
	
	\plot 0.5 -22.5 1.5 -22.5 /
	\plot 3.5 -21.5 3.5 -20.5 /
	
	\plot 6 -24 6 -20 /
	\plot 7 -24 7 -20 /
	\plot 8 -24 8 -20 /
	\plot 9 -24 9 -20 /
	\plot 10 -24 10 -20 /
	
	\plot 6 -24 10 -24 /
	\plot 6 -23 10 -23 /
	\plot 6 -22 10 -22 /
	\plot 6 -21 10 -21 /
	\plot 6 -20 10 -20 /
	
	\plot 6.5 -23.5 8.5 -23.5 /
	\plot 9.5 -22.5 9.5 -20.5 /
	
	\plot 12 -24 12 -20 /
	\plot 13 -24 13 -20 /
	\plot 14 -24 14 -20 /
	\plot 15 -24 15 -20 /
	\plot 16 -24 16 -20 /
	
	\plot 12 -24 16 -24 /
	\plot 12 -23 16 -23 /
	\plot 12 -22 16 -22 /
	\plot 12 -21 16 -21 /
	\plot 12 -20 16 -20 /
	
	\plot 13.5 -23.5 13.5 -22.5 /
	\plot 14.5 -20.5 15.5 -20.5 /
	
	\plot 18 -24 18 -20 /
	\plot 19 -24 19 -20 /
	\plot 20 -24 20 -20 /
	\plot 21 -24 21 -20 /
	\plot 22 -24 22 -20 /
	
	\plot 18 -24 22 -24 /
	\plot 18 -23 22 -23 /
	\plot 18 -22 22 -22 /
	\plot 18 -21 22 -21 /
	\plot 18 -20 22 -20 /
	
	\plot 18.5 -23.5 18.5 -22.5 /
	\plot 20.5 -20.5 21.5 -20.5 /
	
	\plot -12 -30 -12 -26 /
	\plot -11 -30 -11 -26 /
	\plot -10 -30 -10 -26 /
	\plot -9 -30 -9 -26 /
	\plot -8 -30 -8 -26 /
	
	\plot -12 -30 -8 -30 /
	\plot -12 -29 -8 -29 /
	\plot -12 -28 -8 -28 /
	\plot -12 -27 -8 -27 /
	\plot -12 -26 -8 -26 /
	
	\plot -9.5 -26.5 -8.5 -26.5 /
	\plot -9.5 -27.5 -8.5 -27.5 /
	
	\plot -6 -30 -6 -26 /
	\plot -5 -30 -5 -26 /
	\plot -4 -30 -4 -26 /
	\plot -3 -30 -3 -26 /
	\plot -2 -30 -2 -26 /
	
	\plot -6 -30 -2 -30 /
	\plot -6 -29 -2 -29 /
	\plot -6 -28 -2 -28 /
	\plot -6 -27 -2 -27 /
	\plot -6 -26 -2 -26 /
	
	\plot -5.5 -29.5 -4.5 -29.5 /
	\plot -3.5 -27.5 -3.5 -26.5 /
	
	\plot 0 -30 0 -26 /
	\plot 1 -30 1 -26 /
	\plot 2 -30 2 -26 /
	\plot 3 -30 3 -26 /
	\plot 4 -30 4 -26 /
	
	\plot 0 -30 4 -30 /
	\plot 0 -29 4 -29 /
	\plot 0 -28 4 -28 /
	\plot 0 -27 4 -27 /
	\plot 0 -26 4 -26 /
	
	\plot 0.5 -28.5 1.5 -28.5 /
	\plot 2.5 -27.5 2.5 -26.5 /
	
	\plot 6 -30 6 -26 /
	\plot 7 -30 7 -26 /
	\plot 8 -30 8 -26 /
	\plot 9 -30 9 -26 /
	\plot 10 -30 10 -26 /
	
	\plot 6 -30 10 -30 /
	\plot 6 -29 10 -29 /
	\plot 6 -28 10 -28 /
	\plot 6 -27 10 -27 /
	\plot 6 -26 10 -26 /
	
	\plot 8.5 -27.5 8.5 -26.5 /
	\plot 9.5 -27.5 9.5 -26.5 /
	
	\plot 12 -30 12 -26 /
	\plot 13 -30 13 -26 /
	\plot 14 -30 14 -26 /
	\plot 15 -30 15 -26 /
	\plot 16 -30 16 -26 /
	
	\plot 12 -30 16 -30 /
	\plot 12 -29 16 -29 /
	\plot 12 -28 16 -28 /
	\plot 12 -27 16 -27 /
	\plot 12 -26 16 -26 /
	
	\plot 14.5 -27.5 15.5 -26.5 /
	\plot 14.5 -26.5 15.5 -27.5 /
	
	\plot 18 -30 18 -26 /
	\plot 19 -30 19 -26 /
	\plot 20 -30 20 -26 /
	\plot 21 -30 21 -26 /
	\plot 22 -30 22 -26 /
	
	\plot 18 -30 22 -30 /
	\plot 18 -29 22 -29 /
	\plot 18 -28 22 -28 /
	\plot 18 -27 22 -27 /
	\plot 18 -26 22 -26 /
	
	\plot 18.5 -28.5 21.5 -28.5 /
	\plot 20.5 -29.5 20.5 -26.5 /
	
	\plot -12 -36 -12 -32 /
	\plot -11 -36 -11 -32 /
	\plot -10 -36 -10 -32 /
	\plot -9 -36 -9 -32 /
	\plot -8 -36 -8 -32 /
	
	\plot -12 -36 -8 -36 /
	\plot -12 -35 -8 -35 /
	\plot -12 -34 -8 -34 /
	\plot -12 -33 -8 -33 /
	\plot -12 -32 -8 -32 /
	
	\plot -11.5 -33.5 -8.5 -33.5 /
	\plot -9.5 -32.5 -9.5 -35.5 /
	
	\plot -6 -36 -6 -32 /
	\plot -5 -36 -5 -32 /
	\plot -4 -36 -4 -32 /
	\plot -3 -36 -3 -32 /
	\plot -2 -36 -2 -32 /
	
	\plot -6 -36 -2 -36 /
	\plot -6 -35 -2 -35 /
	\plot -6 -34 -2 -34 /
	\plot -6 -33 -2 -33 /
	\plot -6 -32 -2 -32 /
	
	\plot -4.5 -32.5 -3.5 -32.5 /
	\plot -4.5 -35.5 -3.5 -35.5 /
	
	\plot 0 -36 0 -32 /
	\plot 1 -36 1 -32 /
	\plot 2 -36 2 -32 /
	\plot 3 -36 3 -32 /
	\plot 4 -36 4 -32 /
	
	\plot 0 -36 4 -36 /
	\plot 0 -35 4 -35 /
	\plot 0 -34 4 -34 /
	\plot 0 -33 4 -33 /
	\plot 0 -32 4 -32 /
	
	\plot 0.5 -35.5 0.5 -33.5 /
	\plot 1.5 -32.5 3.5 -32.5 /
	
	\plot 6 -36 6 -32 /
	\plot 7 -36 7 -32 /
	\plot 8 -36 8 -32 /
	\plot 9 -36 9 -32 /
	\plot 10 -36 10 -32 /
	
	\plot 6 -36 10 -36 /
	\plot 6 -35 10 -35 /
	\plot 6 -34 10 -34 /
	\plot 6 -33 10 -33 /
	\plot 6 -32 10 -32 /
	
	\plot 7.5 -32.5 8.5 -32.5 /
	\plot 9.5 -34.5 9.5 -33.5 /
	
	\plot 12 -36 12 -32 /
	\plot 13 -36 13 -32 /
	\plot 14 -36 14 -32 /
	\plot 15 -36 15 -32 /
	\plot 16 -36 16 -32 /
	
	\plot 12 -36 16 -36 /
	\plot 12 -35 16 -35 /
	\plot 12 -34 16 -34 /
	\plot 12 -33 16 -33 /
	\plot 12 -32 16 -32 /
	
	\plot 12.5 -34.5 12.5 -33.5 /
	\plot 13.5 -32.5 14.5 -32.5 /
	
	\plot 18 -36 18 -32 /
	\plot 19 -36 19 -32 /
	\plot 20 -36 20 -32 /
	\plot 21 -36 21 -32 /
	\plot 22 -36 22 -32 /
	
	\plot 18 -36 22 -36 /
	\plot 18 -35 22 -35 /
	\plot 18 -34 22 -34 /
	\plot 18 -33 22 -33 /
	\plot 18 -32 22 -32 /
	
	\plot 19.5 -32.5 21.5 -32.5 /
	\plot 19.5 -34.5 21.5 -34.5 /
	
	\plot -12 -42 -12 -38 /
	\plot -11 -42 -11 -38 /
	\plot -10 -42 -10 -38 /
	\plot -9 -42 -9 -38 /
	\plot -8 -42 -8 -38 /
	
	\plot -12 -42 -8 -42 /
	\plot -12 -41 -8 -41 /
	\plot -12 -40 -8 -40 /
	\plot -12 -39 -8 -39 /
	\plot -12 -38 -8 -38 /
	
	\plot -8.5 -38.5 -10.5 -38.5 /
	\plot -9.5 -41.5 -9.5 -39.5 /
	
	\plot -6 -42 -6 -38 /
	\plot -5 -42 -5 -38 /
	\plot -4 -42 -4 -38 /
	\plot -3 -42 -3 -38 /
	\plot -2 -42 -2 -38 /
	
	\plot -6 -42 -2 -42 /
	\plot -6 -41 -2 -41 /
	\plot -6 -40 -2 -40 /
	\plot -6 -39 -2 -39 /
	\plot -6 -38 -2 -38 /
	
	\plot -4.5 -40.5 -4.5 -38.5 /
	\plot -5.5 -41.5 -3.5 -41.5 /
	
	\plot 0 -42 0 -38 /
	\plot 1 -42 1 -38 /
	\plot 2 -42 2 -38 /
	\plot 3 -42 3 -38 /
	\plot 4 -42 4 -38 /
	
	\plot 0 -42 4 -42 /
	\plot 0 -41 4 -41 /
	\plot 0 -40 4 -40 /
	\plot 0 -39 4 -39 /
	\plot 0 -38 4 -38 /
	
	\plot 1.5 -40.5 1.5 -38.5 /
	\plot 0.5 -39.5 2.5 -39.5 /
	
	\plot 6 -42 6 -38 /
	\plot 7 -42 7 -38 /
	\plot 8 -42 8 -38 /
	\plot 9 -42 9 -38 /
	\plot 10 -42 10 -38 /
	
	\plot 6 -42 10 -42 /
	\plot 6 -41 10 -41 /
	\plot 6 -40 10 -40 /
	\plot 6 -39 10 -39 /
	\plot 6 -38 10 -38 /
	
	\plot 7.5 -40.5 7.5 -38.5 /
	\plot 9.5 -40.5 9.5 -38.5 /
	
	\plot 12 -42 12 -38 /
	\plot 13 -42 13 -38 /
	\plot 14 -42 14 -38 /
	\plot 15 -42 15 -38 /
	\plot 16 -42 16 -38 /
	
	\plot 12 -42 16 -42 /
	\plot 12 -41 16 -41 /
	\plot 12 -40 16 -40 /
	\plot 12 -39 16 -39 /
	\plot 12 -38 16 -38 /
	
	\plot 13.5 -38.5 15.5 -40.5 /
	\plot 13.5 -40.5 15.5 -38.5 /
	
	\plot 18 -42 18 -38 /
	\plot 19 -42 19 -38 /
	\plot 20 -42 20 -38 /
	\plot 21 -42 21 -38 /
	\plot 22 -42 22 -38 /
	
	\plot 18 -42 22 -42 /
	\plot 18 -41 22 -41 /
	\plot 18 -40 22 -40 /
	\plot 18 -39 22 -39 /
	\plot 18 -38 22 -38 /
	
	\plot 18.5 -40.5 21.5 -40.5 /
	\plot 19.5 -41.5 19.5 -38.5 /
	
	\plot -12 -48 -12 -44 /
	\plot -11 -48 -11 -44 /
	\plot -10 -48 -10 -44 /
	\plot -9 -48 -9 -44 /
	\plot -8 -48 -8 -44 /
	
	\plot -12 -48 -8 -48 /
	\plot -12 -47 -8 -47 /
	\plot -12 -46 -8 -46 /
	\plot -12 -45 -8 -45 /
	\plot -12 -44 -8 -44 /
	
	\plot -10.5 -47.5 -10.5 -44.5 /
	\plot -11.5 -45.5 -8.5 -45.5 /

	\plot -6 -48 -6 -44 /
	\plot -5 -48 -5 -44 /
	\plot -4 -48 -4 -44 /
	\plot -3 -48 -3 -44 /
	\plot -2 -48 -2 -44 /
	
	\plot -6 -48 -2 -48 /
	\plot -6 -47 -2 -47 /
	\plot -6 -46 -2 -46 /
	\plot -6 -45 -2 -45 /
	\plot -6 -44 -2 -44 /
	
	\plot -4.5 -47.5 -4.5 -44.5 /
	\plot -3.5 -47.5 -3.5 -44.5 /
	
	\plot 0 -48 0 -44 /
	\plot 1 -48 1 -44 /
	\plot 2 -48 2 -44 /
	\plot 3 -48 3 -44 /
	\plot 4 -48 4 -44 /
	
	\plot 0 -48 4 -48 /
	\plot 0 -47 4 -47 /
	\plot 0 -46 4 -46 /
	\plot 0 -45 4 -45 /
	\plot 0 -44 4 -44 /
	
	\plot 1.5 -47.5 2.5 -44.5 /
	\plot 1.5 -44.5 2.5 -47.5 /

	\plot 6 -48 6 -44 /
	\plot 7 -48 7 -44 /
	\plot 8 -48 8 -44 /
	\plot 9 -48 9 -44 /
	\plot 10 -48 10 -44 /
	
	\plot 6 -48 10 -48 /
	\plot 6 -47 10 -47 /
	\plot 6 -46 10 -46 /
	\plot 6 -45 10 -45 /
	\plot 6 -44 10 -44 /
	
	\plot 6.5 -44.5 7.5 -44.5 /
	\plot 6.5 -45.5 7.5 -45.5 /
	
	\plot 12 -48 12 -44 /
	\plot 13 -48 13 -44 /
	\plot 14 -48 14 -44 /
	\plot 15 -48 15 -44 /
	\plot 16 -48 16 -44 /
	
	\plot 12 -48 16 -48 /
	\plot 12 -47 16 -47 /
	\plot 12 -46 16 -46 /
	\plot 12 -45 16 -45 /
	\plot 12 -44 16 -44 /
	
	\plot 12.5 -44.5 14.5 -44.5 /
	\plot 12.5 -46.5 14.5 -46.5 /
	
	\plot 18 -48 18 -44 /
	\plot 19 -48 19 -44 /
	\plot 20 -48 20 -44 /
	\plot 21 -48 21 -44 /
	\plot 22 -48 22 -44 /
	
	\plot 18 -48 22 -48 /
	\plot 18 -47 22 -47 /
	\plot 18 -46 22 -46 /
	\plot 18 -45 22 -45 /
	\plot 18 -44 22 -44 /
	
	\plot 18.5 -44.5 21.5 -44.5 /
	\plot 18.5 -47.5 21.5 -47.5 /
	
	\plot -12 -54 -12 -50 /
	\plot -11 -54 -11 -50 /
	\plot -10 -54 -10 -50 /
	\plot -9 -54 -9 -50 /
	\plot -8 -54 -8 -50 /
	
	\plot -12 -54 -8 -54 /
	\plot -12 -53 -8 -53 /
	\plot -12 -52 -8 -52 /
	\plot -12 -51 -8 -51 /
	\plot -12 -50 -8 -50 /
	
	\plot -11.5 -50.5 -11.5 -51.5 /
	\plot -10.5 -50.5 -10.5 -51.5 /
	
	\plot -6 -54 -6 -50 /
	\plot -5 -54 -5 -50 /
	\plot -4 -54 -4 -50 /
	\plot -3 -54 -3 -50 /
	\plot -2 -54 -2 -50 /
	
	\plot -6 -54 -2 -54 /
	\plot -6 -53 -2 -53 /
	\plot -6 -52 -2 -52 /
	\plot -6 -51 -2 -51 /
	\plot -6 -50 -2 -50 /
	
	\plot -5.5 -50.5 -4.5 -51.5 /
	\plot -4.5 -50.5 -5.5 -51.5 /
	
	\plot 0 -54 0 -50 /
	\plot 1 -54 1 -50 /
	\plot 2 -54 2 -50 /
	\plot 3 -54 3 -50 /
	\plot 4 -54 4 -50 /
	
	\plot 0 -54 4 -54 /
	\plot 0 -53 4 -53 /
	\plot 0 -52 4 -52 /
	\plot 0 -51 4 -51 /
	\plot 0 -50 4 -50 /
	
	\plot 0.5 -52.5 0.5 -50.5 /
	\plot 2.5 -52.5 2.5 -50.5 /
	
	\plot 6 -54 6 -50 /
	\plot 7 -54 7 -50 /
	\plot 8 -54 8 -50 /
	\plot 9 -54 9 -50 /
	\plot 10 -54 10 -50 /
	
	\plot 6 -54 10 -54 /
	\plot 6 -53 10 -53 /
	\plot 6 -52 10 -52 /
	\plot 6 -51 10 -51 /
	\plot 6 -50 10 -50 /
	
	\plot 6.5 -52.5 8.5 -50.5 /
	\plot 8.5 -52.5 6.5 -50.5 /
	
	\plot 12 -54 12 -50 /
	\plot 13 -54 13 -50 /
	\plot 14 -54 14 -50 /
	\plot 15 -54 15 -50 /
	\plot 16 -54 16 -50 /
	
	\plot 12 -54 16 -54 /
	\plot 12 -53 16 -53 /
	\plot 12 -52 16 -52 /
	\plot 12 -51 16 -51 /
	\plot 12 -50 16 -50 /
	
	\plot 12.5 -53.5 12.5 -50.5 /
	\plot 15.5 -53.5 15.5 -50.5 /
	
	\plot 18 -54 18 -50 /
	\plot 19 -54 19 -50 /
	\plot 20 -54 20 -50 /
	\plot 21 -54 21 -50 /
	\plot 22 -54 22 -50 /
	
	\plot 18 -54 22 -54 /
	\plot 18 -53 22 -53 /
	\plot 18 -52 22 -52 /
	\plot 18 -51 22 -51 /
	\plot 18 -50 22 -50 /
	
	\plot 18.5 -53.5 21.5 -50.5 /
	\plot 21.5 -53.5 18.5 -50.5 /

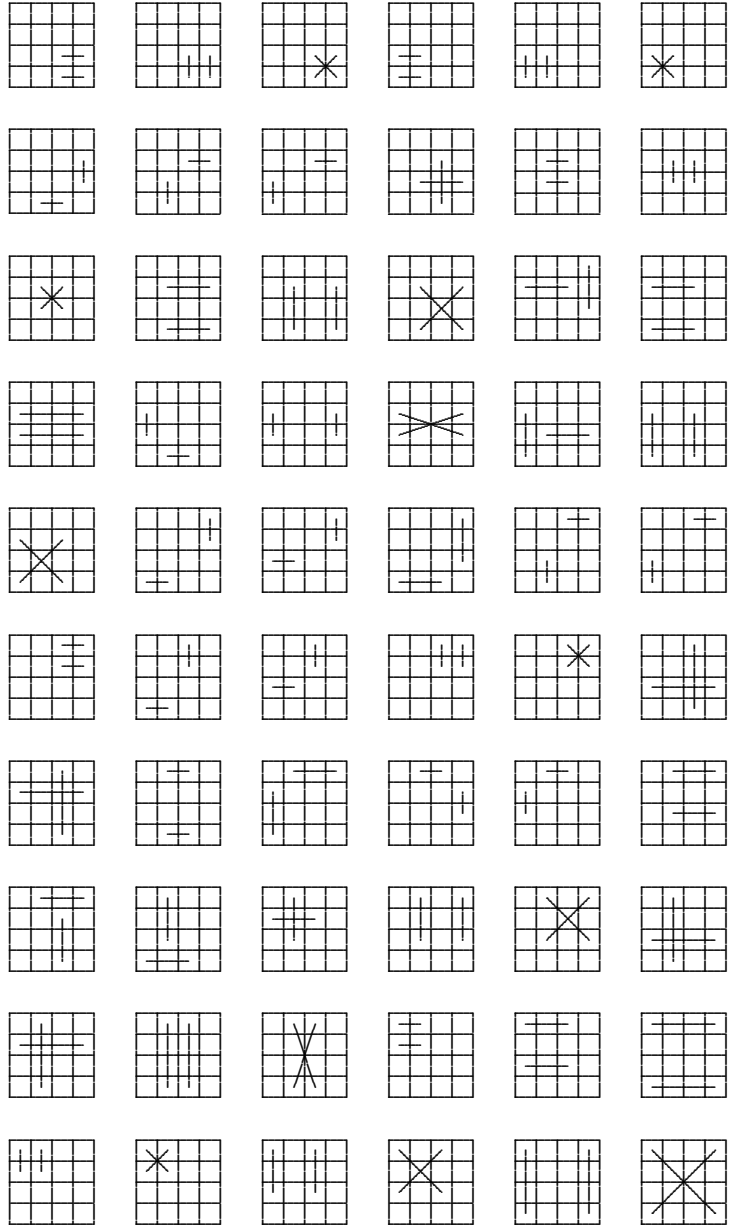
	\put {The 60 stable permutation of rank 1 of cycle-type $(2,2)$ in $S([4]^2)$} at 5 -57.5
	
	\put{\bf{Fig.4}} at 5, -61
		
	\endpicture \]

\newpage

	\[  \beginpicture
		
	\setcoordinatesystem units <0.28cm,0.28cm>
	\setplotarea x from 4 to 8, y from 2 to 5
	
	\setlinear
	
	\plot -12 0 -12 4 /
	\plot -11 0 -11 4 /
	\plot -10 0 -10 4 /
	\plot -9 0 -9 4 /
	\plot -8 0 -8 4 /
	
	\plot -12 0 -8 0 /
	\plot -12 1 -8 1 /
	\plot -12 2 -8 2 /
	\plot -12 3 -8 3 /
	\plot -12 4 -8 4 /
	
	\plot -11.5 2.5 -10.5 2.5 /
	\plot -11.5 3.5 -10.5 3.5 /
	\plot -9.5 2.5 -9.5 3.5 /
	
	\plot -6 0 -6 4 /
	\plot -5 0 -5 4 /
	\plot -4 0 -4 4 /
	\plot -3 0 -3 4 /
	\plot -2 0 -2 4 /
	
	\plot -6 0 -2 0 /
	\plot -6 1 -2 1 /
	\plot -6 2 -2 2 /
	\plot -6 3 -2 3 /
	\plot -6 4 -2 4 /
	
	\plot -5.5 3.5 -4.5 3.5 /
	\plot -5.5 2.5 -4.5 2.5 /
	\plot -5.5 1.5 -4.5 1.5 /
	
	\plot 0 0 0 4 /
	\plot 1 0 1 4 /
	\plot 2 0 2 4 /
	\plot 3 0 3 4 /
	\plot 4 0 4 4 /
	
	\plot 0 0 4 0 /
	\plot 0 1 4 1 /
	\plot 0 2 4 2 /
	\plot 0 3 4 3 /
	\plot 0 4 4 4 /
	
	\plot 0.5 3.5 1.5 3.5 /
	\plot 0.5 2.5 1.5 2.5 /
	\plot 0.5 1.5 1.5 0.5 /
	
	\plot 6 0 6 4 /
	\plot 7 0 7 4 /
	\plot 8 0 8 4 /
	\plot 9 0 9 4 /
	\plot 10 0 10 4 /
	
	\plot 6 0 10 0 /
	\plot 6 1 10 1 /
	\plot 6 2 10 2 /
	\plot 6 3 10 3 /
	\plot 6 4 10 4 /
	
	\plot 6.5 2.5 7.5 2.5 /
	\plot 6.5 3.5 7.5 3.5 /
	\plot 7.5 0.5 7.5 1.5 /
	
	\plot 12 0 12 4 /
	\plot 13 0 13 4 /
	\plot 14 0 14 4 /
	\plot 15 0 15 4 /
	\plot 16 0 16 4 /
	
	\plot 12 0 16 0 /
	\plot 12 1 16 1 /
	\plot 12 2 16 2 /
	\plot 12 3 16 3 /
	\plot 12 4 16 4 /
	
	\plot 12.5 3.5 14.5 3.5 /
	\plot 12.5 1.5 14.5 1.5 /
	\plot 13.5 3.5 13.5 1.5 /
	
	\plot 18 0 18 4 /
	\plot 19 0 19 4 /
	\plot 20 0 20 4 /
	\plot 21 0 21 4 /
	\plot 22 0 22 4 /
	
	\plot 18 0 22 0 /
	\plot 18 1 22 1 /
	\plot 18 2 22 2 /
	\plot 18 3 22 3 /
	\plot 18 4 22 4 /
	
	\plot 18.5 1.5 20.5 1.5 /
	\plot 18.5 2.5 20.5 2.5 /
	\plot 18.5 3.5 20.5 3.5 /
	
	\plot 24 0 24 4 /
	\plot 25 0 25 4 /
	\plot 26 0 26 4 /
	\plot 27 0 27 4 /
	\plot 28 0 28 4 /
	
	\plot 24 0 28 0 /
	\plot 24 1 28 1 /
	\plot 24 2 28 2 /
	\plot 24 3 28 3 /
	\plot 24 4 28 4 /
	
	\plot 24.5 3.5 26.5 3.5 /
	\plot 24.5 1.5 26.5 1.5 /
	\plot 24.5 2.5 26.5 0.5 /
	
	\plot 30 0 30 4 /
	\plot 31 0 31 4 /
	\plot 32 0 32 4 /
	\plot 33 0 33 4 /
	\plot 34 0 34 4 /
	
	\plot 30 0 34 0 /
	\plot 30 1 34 1 /
	\plot 30 2 34 2 /
	\plot 30 3 34 3 /
	\plot 30 4 34 4 /
	
	\plot 30.5 3.5 32.5 3.5 /
	\plot 30.5 1.5 32.5 1.5 /
	\plot 32.5 2.5 32.5 0.5 /
	
	\plot -12 -6 -12 -2 /
	\plot -11 -6 -11 -2 /
	\plot -10 -6 -10 -2 /
	\plot -9 -6 -9 -2 /
	\plot -8 -6 -8 -2 /
	
	\plot -12 -6 -8 -6 /
	\plot -12 -5 -8 -5 /
	\plot -12 -4 -8 -4 /
	\plot -12 -3 -8 -3 /
	\plot -12 -2 -8 -2 /
	
	\plot -11.5 -5.5 -8.5 -5.5 /
	\plot -11.5 -2.5 -8.5 -2.5 /
	\plot -10.5 -5.5 -10.5 -2.5 /
	
	\plot -6 -6 -6 -2 /
	\plot -5 -6 -5 -2 /
	\plot -4 -6 -4 -2 /
	\plot -3 -6 -3 -2 /
	\plot -2 -6 -2 -2 /
	
	\plot -6 -6 -2 -6 /
	\plot -6 -5 -2 -5 /
	\plot -6 -4 -2 -4 /
	\plot -6 -3 -2 -3 /
	\plot -6 -2 -2 -2 /
	
	\plot -5.5 -2.5 -2.5 -2.5 /
	\plot -5.5 -3.5 -2.5 -3.5 /
	\plot -5.5 -5.5 -2.5 -5.5 /
	
	\plot 0 -6 0 -2 /
	\plot 1 -6 1 -2 /
	\plot 2 -6 2 -2 /
	\plot 3 -6 3 -2 /
	\plot 4 -6 4 -2 /
	
	\plot 0 -6 4 -6 /
	\plot 0 -5 4 -5 /
	\plot 0 -4 4 -4 /
	\plot 0 -3 4 -3 /
	\plot 0 -2 4 -2 /
	
	\plot 0.5 -3.5 1.5 -3.5 /
	\plot 0.5 -2.5 1.5 -2.5 /
	\plot 3.5 -2.5 3.5 -3.5 /
	
	\plot 6 -6 6 -2 /
	\plot 7 -6 7 -2 /
	\plot 8 -6 8 -2 /
	\plot 9 -6 9 -2 /
	\plot 10 -6 10 -2 /
	
	\plot 6 -6 10 -6 /
	\plot 6 -5 10 -5 /
	\plot 6 -4 10 -4 /
	\plot 6 -3 10 -3 /
	\plot 6 -2 10 -2 /
	
	\plot 6.5 -2.5 7.5 -2.5 /
	\plot 6.5 -3.5 7.5 -3.5 /
	\plot 6.5 -4.5 6.5 -5.5 /
	
	\plot 12 -6 12 -2 /
	\plot 13 -6 13 -2 /
	\plot 14 -6 14 -2 /
	\plot 15 -6 15 -2 /
	\plot 16 -6 16 -2 /
	
	\plot 12 -6 16 -6 /
	\plot 12 -5 16 -5 /
	\plot 12 -4 16 -4 /
	\plot 12 -3 16 -3 /
	\plot 12 -2 16 -2 /
	
	\plot 12.5 -2.5 13.5 -2.5 /
	\plot 12.5 -3.5 13.5 -3.5 /
	\plot 12.5 -5.5 13.5 -4.5 /
	
	\plot 18 -6 18 -2 /
	\plot 19 -6 19 -2 /
	\plot 20 -6 20 -2 /
	\plot 21 -6 21 -2 /
	\plot 22 -6 22 -2 /
	
	\plot 18 -6 22 -6 /
	\plot 18 -5 22 -5 /
	\plot 18 -4 22 -4 /
	\plot 18 -3 22 -3 /
	\plot 18 -2 22 -2 /
	
	\plot 18.5 -2.5 19.5 -2.5 /
	\plot 18.5 -3.5 19.5 -3.5 /
	\plot 18.5 -5.5 19.5 -5.5 /
	
	\plot 24 -6 24 -2 /
	\plot 25 -6 25 -2 /
	\plot 26 -6 26 -2 /
	\plot 27 -6 27 -2 /
	\plot 28 -6 28 -2 /
	
	\plot 24 -6 28 -6 /
	\plot 24 -5 28 -5 /
	\plot 24 -4 28 -4 /
	\plot 24 -3 28 -3 /
	\plot 24 -2 28 -2 /
	
	\plot 24.5 -2.5 26.5 -2.5 /
	\plot 24.5 -4.5 26.5 -4.5 /
	\plot 27.5 -2.5 27.5 -4.5 /
	
	\plot 30 -6 30 -2 /
	\plot 31 -6 31 -2 /
	\plot 32 -6 32 -2 /
	\plot 33 -6 33 -2 /
	\plot 34 -6 34 -2 /
	
	\plot 30 -6 34 -6 /
	\plot 30 -5 34 -5 /
	\plot 30 -4 34 -4 /
	\plot 30 -3 34 -3 /
	\plot 30 -2 34 -2 /
	
	\plot 30.5 -2.5 32.5 -2.5 /
	\plot 30.5 -4.5 32.5 -4.5 /
	\plot 30.5 -5.5 30.5 -3.5 /
	
	\plot -12 -12 -12 -8 /
	\plot -11 -12 -11 -8 /
	\plot -10 -12 -10 -8 /
	\plot -9 -12 -9 -8 /
	\plot -8 -12 -8 -8 /
	
	\plot -12 -12 -8 -12 /
	\plot -12 -11 -8 -11 /
	\plot -12 -10 -8 -10 /
	\plot -12 -9 -8 -9 /
	\plot -12 -8 -8 -8 /
	
	\plot -11.5 -8.5 -9.5 -8.5 /
	\plot -11.5 -10.5 -9.5 -10.5 /
	\plot -11.5 -11.5 -9.5 -9.5 /
	
	\plot -6 -12 -6 -8 /
	\plot -5 -12 -5 -8 /
	\plot -4 -12 -4 -8 /
	\plot -3 -12 -3 -8 /
	\plot -2 -12 -2 -8 /
	
	\plot -6 -12 -2 -12 /
	\plot -6 -11 -2 -11 /
	\plot -6 -10 -2 -10 /
	\plot -6 -9 -2 -9 /
	\plot -6 -8 -2 -8 /
	
	\plot -5.5 -11.5 -3.5 -11.5 /
	\plot -5.5 -10.5 -3.5 -10.5 /
	\plot -5.5 -8.5 -3.5 -8.5 /
	
	\plot 0 -12 0 -8 /
	\plot 1 -12 1 -8 /
	\plot 2 -12 2 -8 /
	\plot 3 -12 3 -8 /
	\plot 4 -12 4 -8 /
	
	\plot 0 -12 4 -12 /
	\plot 0 -11 4 -11 /
	\plot 0 -10 4 -10 /
	\plot 0 -9 4 -9 /
	\plot 0 -8 4 -8 /
	
	\plot 0.5 -11.5 3.5 -11.5 /
	\plot 0.5 -8.5 3.5 -8.5 /
	\plot 2.5 -11.5 2.5 -8.5 /
	
	\plot 6 -12 6 -8 /
	\plot 7 -12 7 -8 /
	\plot 8 -12 8 -8 /
	\plot 9 -12 9 -8 /
	\plot 10 -12 10 -8 /
	
	\plot 6 -12 10 -12 /
	\plot 6 -11 10 -11 /
	\plot 6 -10 10 -10 /
	\plot 6 -9 10 -9 /
	\plot 6 -8 10 -8 /
	
	\plot 6.5 -11.5 9.5 -11.5 /
	\plot 6.5 -8.5 9.5 -8.5 /
	\plot 6.5 -9.5 6.5 -10.5 /
	
	\plot 12 -12 12 -8 /
	\plot 13 -12 13 -8 /
	\plot 14 -12 14 -8 /
	\plot 15 -12 15 -8 /
	\plot 16 -12 16 -8 /
	
	\plot 12 -12 16 -12 /
	\plot 12 -11 16 -11 /
	\plot 12 -10 16 -10 /
	\plot 12 -9 16 -9 /
	\plot 12 -8 16 -8 /
	
	\plot 12.5 -8.5 15.5 -8.5 /
	\plot 12.5 -11.5 15.5 -11.5 /
	\plot 12.5 -9.5 15.5 -10.5 /
	
	\plot 18 -12 18 -8 /
	\plot 19 -12 19 -8 /
	\plot 20 -12 20 -8 /
	\plot 21 -12 21 -8 /
	\plot 22 -12 22 -8 /
	
	\plot 18 -12 22 -12 /
	\plot 18 -11 22 -11 /
	\plot 18 -10 22 -10 /
	\plot 18 -9 22 -9 /
	\plot 18 -8 22 -8 /
	
	\plot 18.5 -11.5 21.5 -11.5 /
	\plot 18.5 -8.5 21.5 -8.5 /
	\plot 21.5 -9.5 21.5 -10.5 /
	
	\plot 24 -12 24 -8 /
	\plot 25 -12 25 -8 /
	\plot 26 -12 26 -8 /
	\plot 27 -12 27 -8 /
	\plot 28 -12 28 -8 /
	
	\plot 24 -12 28 -12 /
	\plot 24 -11 28 -11 /
	\plot 24 -10 28 -10 /
	\plot 24 -9 28 -9 /
	\plot 24 -8 28 -8 /
	
	\plot 24.5 -8.5 25.5 -9.5 /
	\plot 24.5 -9.5 25.5 -8.5 /
	\plot 26.5 -9.5 26.5 -8.5 /
	
	\plot 30 -12 30 -8 /
	\plot 31 -12 31 -8 /
	\plot 32 -12 32 -8 /
	\plot 33 -12 33 -8 /
	\plot 34 -12 34 -8 /
	
	\plot 30 -12 34 -12 /
	\plot 30 -11 34 -11 /
	\plot 30 -10 34 -10 /
	\plot 30 -9 34 -9 /
	\plot 30 -8 34 -8 /
	
	\plot 30.5 -8.5 33.5 -8.5 /
	\plot 30.5 -11.5 33.5 -11.5 /
	\plot 30.5 -10.5 33.5 -9.5 /
	
	\plot -12 -18 -12 -14 /
	\plot -11 -18 -11 -14 /
	\plot -10 -18 -10 -14 /
	\plot -9 -18 -9 -14 /
	\plot -8 -18 -8 -14 /
	
	\plot -12 -18 -8 -18 /
	\plot -12 -17 -8 -17 /
	\plot -12 -16 -8 -16 /
	\plot -12 -15 -8 -15 /
	\plot -12 -14 -8 -14 /
	
	\plot -11.5 -14.5 -8.5 -14.5 /
	\plot -11.5 -16.5 -8.5 -16.5 /
	\plot -11.5 -17.5 -8.5 -17.5 /
	
	\plot -6 -18 -6 -14 /
	\plot -5 -18 -5 -14 /
	\plot -4 -18 -4 -14 /
	\plot -3 -18 -3 -14 /
	\plot -2 -18 -2 -14 /
	
	\plot -6 -18 -2 -18 /
	\plot -6 -17 -2 -17 /
	\plot -6 -16 -2 -16 /
	\plot -6 -15 -2 -15 /
	\plot -6 -14 -2 -14 /
	
	\plot -5.5 -14.5 -4.5 -15.5 /
	\plot -5.5 -15.5 -4.5 -14.5 /
	\plot -2.5 -15.5 -2.5 -14.5 /
	
	\plot 0 -18 0 -14 /
	\plot 1 -18 1 -14 /
	\plot 2 -18 2 -14 /
	\plot 3 -18 3 -14 /
	\plot 4 -18 4 -14 /
	
	\plot 0 -18 4 -18 /
	\plot 0 -17 4 -17 /
	\plot 0 -16 4 -16 /
	\plot 0 -15 4 -15 /
	\plot 0 -14 4 -14 /
	
	\plot 0.5 -16.5 2.5 -14.5 /
	\plot 0.5 -14.5 2.5 -16.5 /
	\plot 1.5 -14.5 1.5 -16.5 /
	
	\plot 6 -18 6 -14 /
	\plot 7 -18 7 -14 /
	\plot 8 -18 8 -14 /
	\plot 9 -18 9 -14 /
	\plot 10 -18 10 -14 /
	
	\plot 6 -18 10 -18 /
	\plot 6 -17 10 -17 /
	\plot 6 -16 10 -16 /
	\plot 6 -15 10 -15 /
	\plot 6 -14 10 -14 /
	
	\plot 6.5 -17.5 9.5 -14.5 /
	\plot 6.5 -14.5 9.5 -17.5 /
	\plot 7.5 -14.5 7.5 -17.5 /
	
	\plot 12 -18 12 -14 /
	\plot 13 -18 13 -14 /
	\plot 14 -18 14 -14 /
	\plot 15 -18 15 -14 /
	\plot 16 -18 16 -14 /
	
	\plot 12 -18 16 -18 /
	\plot 12 -17 16 -17 /
	\plot 12 -16 16 -16 /
	\plot 12 -15 16 -15 /
	\plot 12 -14 16 -14 /
	
	\plot 12.5 -14.5 14.5 -16.5 /
	\plot 12.5 -16.5 14.5 -14.5 /
	\plot 15.5 -16.5 15.5 -14.5 /
	
	\plot 18 -18 18 -14 /
	\plot 19 -18 19 -14 /
	\plot 20 -18 20 -14 /
	\plot 21 -18 21 -14 /
	\plot 22 -18 22 -14 /
	
	\plot 18 -18 22 -18 /
	\plot 18 -17 22 -17 /
	\plot 18 -16 22 -16 /
	\plot 18 -15 22 -15 /
	\plot 18 -14 22 -14 /
	
	\plot 18.5 -17.5 21.5 -14.5 /
	\plot 18.5 -14.5 21.5 -17.5 /
	\plot 20.5 -14.5 20.5 -17.5 /
	
	\plot 24 -18 24 -14 /
	\plot 25 -18 25 -14 /
	\plot 26 -18 26 -14 /
	\plot 27 -18 27 -14 /
	\plot 28 -18 28 -14 /
	
	\plot 24 -18 28 -18 /
	\plot 24 -17 28 -17 /
	\plot 24 -16 28 -16 /
	\plot 24 -15 28 -15 /
	\plot 24 -14 28 -14 /
	
	\plot 25.5 -14.5 26.5 -14.5 /
	\plot 24.5 -15.5 24.5 -16.5 /
	\plot 27.5 -15.5 27.5 -16.5 /
	
	\plot 30 -18 30 -14 /
	\plot 31 -18 31 -14 /
	\plot 32 -18 32 -14 /
	\plot 33 -18 33 -14 /
	\plot 34 -18 34 -14 /
	
	\plot 30 -18 34 -18 /
	\plot 30 -17 34 -17 /
	\plot 30 -16 34 -16 /
	\plot 30 -15 34 -15 /
	\plot 30 -14 34 -14 /
	
	\plot 31.5 -14.5 32.5 -14.5 /
	\plot 31.5 -15.5 32.5 -15.5 /
	\plot 31.5 -16.5 32.5 -16.5 /
	
	\plot -12 -24 -12 -20 /
	\plot -11 -24 -11 -20 /
	\plot -10 -24 -10 -20 /
	\plot -9 -24 -9 -20 /
	\plot -8 -24 -8 -20 /
	
	\plot -12 -24 -8 -24 /
	\plot -12 -23 -8 -23 /
	\plot -12 -22 -8 -22 /
	\plot -12 -21 -8 -21 /
	\plot -12 -20 -8 -20 /
	
	\plot -10.5 -21.5 -9.5 -22.5 /
	\plot -10.5 -22.5 -9.5 -21.5 /
	\plot -10.5 -20.5 -9.5 -20.5 /
	
	\plot -6 -24 -6 -20 /
	\plot -5 -24 -5 -20 /
	\plot -4 -24 -4 -20 /
	\plot -3 -24 -3 -20 /
	\plot -2 -24 -2 -20 /
	
	\plot -6 -24 -2 -24 /
	\plot -6 -23 -2 -23 /
	\plot -6 -22 -2 -22 /
	\plot -6 -21 -2 -21 /
	\plot -6 -20 -2 -20 /
	
	\plot -5.5 -23.5 -5.5 -21.5 /
	\plot -3.5 -23.5 -3.5 -21.5 /
	\plot -4.5 -20.5 -2.5 -20.5 /
	
	\plot 0 -24 0 -20 /
	\plot 1 -24 1 -20 /
	\plot 2 -24 2 -20 /
	\plot 3 -24 3 -20 /
	\plot 4 -24 4 -20 /
	
	\plot 0 -24 4 -24 /
	\plot 0 -23 4 -23 /
	\plot 0 -22 4 -22 /
	\plot 0 -21 4 -21 /
	\plot 0 -20 4 -20 /
	
	\plot 1.5 -20.5 3.5 -20.5 /
	\plot 1.5 -21.5 3.5 -21.5 /
	\plot 1.5 -23.5 3.5 -23.5 /
	
	\plot 6 -24 6 -20 /
	\plot 7 -24 7 -20 /
	\plot 8 -24 8 -20 /
	\plot 9 -24 9 -20 /
	\plot 10 -24 10 -20 /
	
	\plot 6 -24 10 -24 /
	\plot 6 -23 10 -23 /
	\plot 6 -22 10 -22 /
	\plot 6 -21 10 -21 /
	\plot 6 -20 10 -20 /
	
	\plot 7.5 -23.5 9.5 -21.5 /
	\plot 7.5 -21.5 9.5 -23.5 /
	\plot 7.5 -20.5 9.5 -20.5 /
	
	\plot 12 -24 12 -20 /
	\plot 13 -24 13 -20 /
	\plot 14 -24 14 -20 /
	\plot 15 -24 15 -20 /
	\plot 16 -24 16 -20 /
	
	\plot 12 -24 16 -24 /
	\plot 12 -23 16 -23 /
	\plot 12 -22 16 -22 /
	\plot 12 -21 16 -21 /
	\plot 12 -20 16 -20 /
	
	\plot 12.5 -21.5 14.5 -21.5 /
	\plot 13.5 -20.5 13.5 -22.5 /
	\plot 15.5 -20.5 15.5 -22.5 /
	
	\plot 18 -24 18 -20 /
	\plot 19 -24 19 -20 /
	\plot 20 -24 20 -20 /
	\plot 21 -24 21 -20 /
	\plot 22 -24 22 -20 /
	
	\plot 18 -24 22 -24 /
	\plot 18 -23 22 -23 /
	\plot 18 -22 22 -22 /
	\plot 18 -21 22 -21 /
	\plot 18 -20 22 -20 /
	
	\plot 19.5 -23.5 21.5 -23.5 /
	\plot 19.5 -21.5 21.5 -21.5 /
	\plot 19.5 -20.5 21.5 -22.5 /
	
	\plot 24 -24 24 -20 /
	\plot 25 -24 25 -20 /
	\plot 26 -24 26 -20 /
	\plot 27 -24 27 -20 /
	\plot 28 -24 28 -20 /
	
	\plot 24 -24 28 -24 /
	\plot 24 -23 28 -23 /
	\plot 24 -22 28 -22 /
	\plot 24 -21 28 -21 /
	\plot 24 -20 28 -20 /
	
	\plot 24.5 -22.5 27.5 -22.5 /
	\plot 25.5 -20.5 25.5 -23.5 /
	\plot 26.5 -20.5 26.5 -23.5 /
	
	\plot 30 -24 30 -20 /
	\plot 31 -24 31 -20 /
	\plot 32 -24 32 -20 /
	\plot 33 -24 33 -20 /
	\plot 34 -24 34 -20 /
	
	\plot 30 -24 34 -24 /
	\plot 30 -23 34 -23 /
	\plot 30 -22 34 -22 /
	\plot 30 -21 34 -21 /
	\plot 30 -20 34 -20 /
	
	\plot 31.5 -20.5 32.5 -20.5 /
	\plot 31.5 -21.5 31.5 -22.5 /
	\plot 32.5 -21.5 32.5 -22.5 /
	
	\plot -12 -30 -12 -26 /
	\plot -11 -30 -11 -26 /
	\plot -10 -30 -10 -26 /
	\plot -9 -30 -9 -26 /
	\plot -8 -30 -8 -26 /
	
	\plot -12 -30 -8 -30 /
	\plot -12 -29 -8 -29 /
	\plot -12 -28 -8 -28 /
	\plot -12 -27 -8 -27 /
	\plot -12 -26 -8 -26 /
	
	\plot -10.5 -26.5 -9.5 -26.5 /
	\plot -10.5 -29.5 -9.5 -29.5 /
	\plot -8.5 -27.5 -8.5 -28.5 /
	
	\plot -6 -30 -6 -26 /
	\plot -5 -30 -5 -26 /
	\plot -4 -30 -4 -26 /
	\plot -3 -30 -3 -26 /
	\plot -2 -30 -2 -26 /
	
	\plot -6 -30 -2 -30 /
	\plot -6 -29 -2 -29 /
	\plot -6 -28 -2 -28 /
	\plot -6 -27 -2 -27 /
	\plot -6 -26 -2 -26 /
	
	\plot -4.5 -26.5 -2.5 -26.5 /
	\plot -4.5 -27.5 -4.5 -29.5 /
	\plot -2.5 -27.5 -2.5 -29.5 /
	
	\plot 0 -30 0 -26 /
	\plot 1 -30 1 -26 /
	\plot 2 -30 2 -26 /
	\plot 3 -30 3 -26 /
	\plot 4 -30 4 -26 /
	
	\plot 0 -30 4 -30 /
	\plot 0 -29 4 -29 /
	\plot 0 -28 4 -28 /
	\plot 0 -27 4 -27 /
	\plot 0 -26 4 -26 /
	
	\plot 1.5 -26.5 3.5 -26.5 /
	\plot 1.5 -28.5 3.5 -28.5 /
	\plot 2.5 -27.5 2.5 -29.5 /
	
	\plot 6 -30 6 -26 /
	\plot 7 -30 7 -26 /
	\plot 8 -30 8 -26 /
	\plot 9 -30 9 -26 /
	\plot 10 -30 10 -26 /
	
	\plot 6 -30 10 -30 /
	\plot 6 -29 10 -29 /
	\plot 6 -28 10 -28 /
	\plot 6 -27 10 -27 /
	\plot 6 -26 10 -26 /
	
	\plot 6.5 -29.5 8.5 -29.5 /
	\plot 7.5 -28.5 7.5 -26.5 /
	\plot 9.5 -28.5 9.5 -26.5 /
	
	\plot 12 -30 12 -26 /
	\plot 13 -30 13 -26 /
	\plot 14 -30 14 -26 /
	\plot 15 -30 15 -26 /
	\plot 16 -30 16 -26 /
	
	\plot 12 -30 16 -30 /
	\plot 12 -29 16 -29 /
	\plot 12 -28 16 -28 /
	\plot 12 -27 16 -27 /
	\plot 12 -26 16 -26 /
	
	\plot 13.5 -29.5 15.5 -29.5 /
	\plot 13.5 -27.5 15.5 -27.5 /
	\plot 13.5 -28.5 13.5 -26.5 /
	
	\plot 18 -30 18 -26 /
	\plot 19 -30 19 -26 /
	\plot 20 -30 20 -26 /
	\plot 21 -30 21 -26 /
	\plot 22 -30 22 -26 /
	
	\plot 18 -30 22 -30 /
	\plot 18 -29 22 -29 /
	\plot 18 -28 22 -28 /
	\plot 18 -27 22 -27 /
	\plot 18 -26 22 -26 /
	
	\plot 18.5 -27.5 21.5 -27.5 /
	\plot 19.5 -29.5 19.5 -26.5 /
	\plot 20.5 -29.5 20.5 -26.5 /
	
	\plot 24 -30 24 -26 /
	\plot 25 -30 25 -26 /
	\plot 26 -30 26 -26 /
	\plot 27 -30 27 -26 /
	\plot 28 -30 28 -26 /
	
	\plot 24 -30 28 -30 /
	\plot 24 -29 28 -29 /
	\plot 24 -28 28 -28 /
	\plot 24 -27 28 -27 /
	\plot 24 -26 28 -26 /
	
	\plot 25.5 -28.5 26.5 -28.5 /
	\plot 25.5 -27.5 26.5 -27.5 /
	\plot 25.5 -29.5 25.5 -26.5 /
	
	\plot 30 -30 30 -26 /
	\plot 31 -30 31 -26 /
	\plot 32 -30 32 -26 /
	\plot 33 -30 33 -26 /
	\plot 34 -30 34 -26 /
	
	\plot 30 -30 34 -30 /
	\plot 30 -29 34 -29 /
	\plot 30 -28 34 -28 /
	\plot 30 -27 34 -27 /
	\plot 30 -26 34 -26 /
	
	\plot 32.5 -26.5 33.5 -26.5 /
	\plot 32.5 -27.5 33.5 -27.5 /
	\plot 30.5 -29.5 30.5 -28.5 /
	
	\plot -12 -36 -12 -32 /
	\plot -11 -36 -11 -32 /
	\plot -10 -36 -10 -32 /
	\plot -9 -36 -9 -32 /
	\plot -8 -36 -8 -32 /
	
	\plot -12 -36 -8 -36 /
	\plot -12 -35 -8 -35 /
	\plot -12 -34 -8 -34 /
	\plot -12 -33 -8 -33 /
	\plot -12 -32 -8 -32 /
	
	\plot -9.5 -35.5 -9.5 -34.5 /
	\plot -8.5 -35.5 -8.5 -34.5 /
	\plot -8.5 -32.5 -9.5 -32.5 /
	
	\plot -6 -36 -6 -32 /
	\plot -5 -36 -5 -32 /
	\plot -4 -36 -4 -32 /
	\plot -3 -36 -3 -32 /
	\plot -2 -36 -2 -32 /
	
	\plot -6 -36 -2 -36 /
	\plot -6 -35 -2 -35 /
	\plot -6 -34 -2 -34 /
	\plot -6 -33 -2 -33 /
	\plot -6 -32 -2 -32 /
	
	\plot -3.5 -32.5 -3.5 -33.5 /
	\plot -2.5 -32.5 -2.5 -33.5 /
	\plot -5.5 -34.5 -4.5 -34.5 /
	
	\plot 0 -36 0 -32 /
	\plot 1 -36 1 -32 /
	\plot 2 -36 2 -32 /
	\plot 3 -36 3 -32 /
	\plot 4 -36 4 -32 /
	
	\plot 0 -36 4 -36 /
	\plot 0 -35 4 -35 /
	\plot 0 -34 4 -34 /
	\plot 0 -33 4 -33 /
	\plot 0 -32 4 -32 /
	
	\plot 2.5 -35.5 3.5 -35.5 /
	\plot 2.5 -34.5 3.5 -34.5 /
	\plot 2.5 -32.5 3.5 -33.5 /
	
	\plot 6 -36 6 -32 /
	\plot 7 -36 7 -32 /
	\plot 8 -36 8 -32 /
	\plot 9 -36 9 -32 /
	\plot 10 -36 10 -32 /
	
	\plot 6 -36 10 -36 /
	\plot 6 -35 10 -35 /
	\plot 6 -34 10 -34 /
	\plot 6 -33 10 -33 /
	\plot 6 -32 10 -32 /
	
	\plot 8.5 -35.5 9.5 -35.5 /
	\plot 8.5 -34.5 9.5 -34.5 /
	\plot 8.5 -33.5 9.5 -32.5 /
	
	\plot 12 -36 12 -32 /
	\plot 13 -36 13 -32 /
	\plot 14 -36 14 -32 /
	\plot 15 -36 15 -32 /
	\plot 16 -36 16 -32 /
	
	\plot 12 -36 16 -36 /
	\plot 12 -35 16 -35 /
	\plot 12 -34 16 -34 /
	\plot 12 -33 16 -33 /
	\plot 12 -32 16 -32 /
	
	\plot 14.5 -35.5 15.5 -35.5 /
	\plot 14.5 -34.5 15.5 -34.5 /
	\plot 15.5 -33.5 15.5 -32.5 /
	
	\plot 18 -36 18 -32 /
	\plot 19 -36 19 -32 /
	\plot 20 -36 20 -32 /
	\plot 21 -36 21 -32 /
	\plot 22 -36 22 -32 /
	
	\plot 18 -36 22 -36 /
	\plot 18 -35 22 -35 /
	\plot 18 -34 22 -34 /
	\plot 18 -33 22 -33 /
	\plot 18 -32 22 -32 /
	
	\plot 19.5 -34.5 20.5 -34.5 /
	\plot 19.5 -33.5 20.5 -33.5 /
	\plot 19.5 -32.5 20.5 -35.5 /
	
	\plot 24 -36 24 -32 /
	\plot 25 -36 25 -32 /
	\plot 26 -36 26 -32 /
	\plot 27 -36 27 -32 /
	\plot 28 -36 28 -32 /
	
	\plot 24 -36 28 -36 /
	\plot 24 -35 28 -35 /
	\plot 24 -34 28 -34 /
	\plot 24 -33 28 -33 /
	\plot 24 -32 28 -32 /
	
	\plot 26.5 -32.5 27.5 -32.5 /
	\plot 26.5 -33.5 27.5 -33.5 /
	\plot 25.5 -34.5 25.5 -35.5 /
	
	\plot 30 -36 30 -32 /
	\plot 31 -36 31 -32 /
	\plot 32 -36 32 -32 /
	\plot 33 -36 33 -32 /
	\plot 34 -36 34 -32 /
	
	\plot 30 -36 34 -36 /
	\plot 30 -35 34 -35 /
	\plot 30 -34 34 -34 /
	\plot 30 -33 34 -33 /
	\plot 30 -32 34 -32 /
	
	\plot 32.5 -32.5 33.5 -32.5 /
	\plot 32.5 -34.5 33.5 -34.5 /
	\plot 32.5 -35.5 33.5 -35.5 /
	
	\plot -12 -42 -12 -38 /
	\plot -11 -42 -11 -38 /
	\plot -10 -42 -10 -38 /
	\plot -9 -42 -9 -38 /
	\plot -8 -42 -8 -38 /
	
	\plot -12 -42 -8 -42 /
	\plot -12 -41 -8 -41 /
	\plot -12 -40 -8 -40 /
	\plot -12 -39 -8 -39 /
	\plot -12 -38 -8 -38 /
	
	\plot -9.5 -41.5 -8.5 -40.5 /
	\plot -9.5 -40.5 -8.5 -41.5 /
	\plot -9.5 -38.5 -8.5 -38.5 /
	
	\plot -6 -42 -6 -38 /
	\plot -5 -42 -5 -38 /
	\plot -4 -42 -4 -38 /
	\plot -3 -42 -3 -38 /
	\plot -2 -42 -2 -38 /
	
	\plot -6 -42 -2 -42 /
	\plot -6 -41 -2 -41 /
	\plot -6 -40 -2 -40 /
	\plot -6 -39 -2 -39 /
	\plot -6 -38 -2 -38 /
	
	\plot -3.5 -38.5 -3.5 -39.5 /
	\plot -2.5 -38.5 -2.5 -39.5 /
	\plot -5.5 -41.5 -4.5 -41.5 /
	
	\plot 0 -42 0 -38 /
	\plot 1 -42 1 -38 /
	\plot 2 -42 2 -38 /
	\plot 3 -42 3 -38 /
	\plot 4 -42 4 -38 /
	
	\plot 0 -42 4 -42 /
	\plot 0 -41 4 -41 /
	\plot 0 -40 4 -40 /
	\plot 0 -39 4 -39 /
	\plot 0 -38 4 -38 /
	
	\plot 2.5 -41.5 3.5 -41.5 /
	\plot 2.5 -40.5 3.5 -40.5 /
	\plot 2.5 -39.5 2.5 -38.5 /
	
	\plot 6 -42 6 -38 /
	\plot 7 -42 7 -38 /
	\plot 8 -42 8 -38 /
	\plot 9 -42 9 -38 /
	\plot 10 -42 10 -38 /
	
	\plot 6 -42 10 -42 /
	\plot 6 -41 10 -41 /
	\plot 6 -40 10 -40 /
	\plot 6 -39 10 -39 /
	\plot 6 -38 10 -38 /
	
	\plot 7.5 -40.5 8.5 -40.5 /
	\plot 7.5 -39.5 8.5 -39.5 /
	\plot 7.5 -41.5 8.5 -38.5 /
	
	\plot 12 -42 12 -38 /
	\plot 13 -42 13 -38 /
	\plot 14 -42 14 -38 /
	\plot 15 -42 15 -38 /
	\plot 16 -42 16 -38 /
	
	\plot 12 -42 16 -42 /
	\plot 12 -41 16 -41 /
	\plot 12 -40 16 -40 /
	\plot 12 -39 16 -39 /
	\plot 12 -38 16 -38 /
	
	\plot 13.5 -40.5 14.5 -40.5 /
	\plot 13.5 -39.5 14.5 -39.5 /
	\plot 14.5 -38.5 14.5 -41.5 /
	
	\plot 18 -42 18 -38 /
	\plot 19 -42 19 -38 /
	\plot 20 -42 20 -38 /
	\plot 21 -42 21 -38 /
	\plot 22 -42 22 -38 /
	
	\plot 18 -42 22 -42 /
	\plot 18 -41 22 -41 /
	\plot 18 -40 22 -40 /
	\plot 18 -39 22 -39 /
	\plot 18 -38 22 -38 /
	
	\plot 19.5 -41.5 21.5 -41.5 /
	\plot 19.5 -39.5 21.5 -39.5 /
	\plot 19.5 -40.5 21.5 -38.5 /
	
	\plot 24 -42 24 -38 /
	\plot 25 -42 25 -38 /
	\plot 26 -42 26 -38 /
	\plot 27 -42 27 -38 /
	\plot 28 -42 28 -38 /
	
	\plot 24 -42 28 -42 /
	\plot 24 -41 28 -41 /
	\plot 24 -40 28 -40 /
	\plot 24 -39 28 -39 /
	\plot 24 -38 28 -38 /
	
	\plot 25.5 -39.5 26.5 -39.5 /
	\plot 25.5 -40.5 26.5 -40.5 /
	\plot 27.5 -39.5 27.5 -40.5 /
	
	\plot 30 -42 30 -38 /
	\plot 31 -42 31 -38 /
	\plot 32 -42 32 -38 /
	\plot 33 -42 33 -38 /
	\plot 34 -42 34 -38 /
	
	\plot 30 -42 34 -42 /
	\plot 30 -41 34 -41 /
	\plot 30 -40 34 -40 /
	\plot 30 -39 34 -39 /
	\plot 30 -38 34 -38 /
	
	\plot 31.5 -41.5 33.5 -41.5 /
	\plot 31.5 -39.5 33.5 -39.5 /
	\plot 32.5 -39.5 32.5 -41.5 /
	
	\plot -12 -48 -12 -44 /
	\plot -11 -48 -11 -44 /
	\plot -10 -48 -10 -44 /
	\plot -9 -48 -9 -44 /
	\plot -8 -48 -8 -44 /
	
	\plot -12 -48 -8 -48 /
	\plot -12 -47 -8 -47 /
	\plot -12 -46 -8 -46 /
	\plot -12 -45 -8 -45 /
	\plot -12 -44 -8 -44 /
	
	\plot -10.5 -47.5 -8.5 -47.5 /
	\plot -10.5 -45.5 -8.5 -45.5 /
	\plot -8.5 -46.5 -8.5 -44.5 /
	
	\plot -6 -48 -6 -44 /
	\plot -5 -48 -5 -44 /
	\plot -4 -48 -4 -44 /
	\plot -3 -48 -3 -44 /
	\plot -2 -48 -2 -44 /
	
	\plot -6 -48 -2 -48 /
	\plot -6 -47 -2 -47 /
	\plot -6 -46 -2 -46 /
	\plot -6 -45 -2 -45 /
	\plot -6 -44 -2 -44 /
	
	\plot -4.5 -47.5 -3.5 -47.5 /
	\plot -4.5 -46.5 -3.5 -46.5 /
	\plot -4.5 -45.5 -3.5 -45.5 /
	
	\plot 0 -48 0 -44 /
	\plot 1 -48 1 -44 /
	\plot 2 -48 2 -44 /
	\plot 3 -48 3 -44 /
	\plot 4 -48 4 -44 /
	
	\plot 0 -48 4 -48 /
	\plot 0 -47 4 -47 /
	\plot 0 -46 4 -46 /
	\plot 0 -45 4 -45 /
	\plot 0 -44 4 -44 /
	
	\plot 1.5 -47.5 3.5 -47.5 /
	\plot 1.5 -46.5 3.5 -46.5 /
	\plot 1.5 -45.5 3.5 -45.5 /

	\plot 6 -48 6 -44 /
	\plot 7 -48 7 -44 /
	\plot 8 -48 8 -44 /
	\plot 9 -48 9 -44 /
	\plot 10 -48 10 -44 /
	
	\plot 6 -48 10 -48 /
	\plot 6 -47 10 -47 /
	\plot 6 -46 10 -46 /
	\plot 6 -45 10 -45 /
	\plot 6 -44 10 -44 /
	
	\plot 7.5 -46.5 8.5 -45.5 /
	\plot 7.5 -45.5 8.5 -46.5 /
	\plot 9.5 -45.5 9.5 -46.5 /
	
	\plot 12 -48 12 -44 /
	\plot 13 -48 13 -44 /
	\plot 14 -48 14 -44 /
	\plot 15 -48 15 -44 /
	\plot 16 -48 16 -44 /
	
	\plot 12 -48 16 -48 /
	\plot 12 -47 16 -47 /
	\plot 12 -46 16 -46 /
	\plot 12 -45 16 -45 /
	\plot 12 -44 16 -44 /
	
	\plot 13.5 -47.5 15.5 -45.5 /
	\plot 13.5 -45.5 15.5 -47.5 /
	\plot 14.5 -45.5 14.5 -47.5 /
	
	\plot 18 -48 18 -44 /
	\plot 19 -48 19 -44 /
	\plot 20 -48 20 -44 /
	\plot 21 -48 21 -44 /
	\plot 22 -48 22 -44 /
	
	\plot 18 -48 22 -48 /
	\plot 18 -47 22 -47 /
	\plot 18 -46 22 -46 /
	\plot 18 -45 22 -45 /
	\plot 18 -44 22 -44 /
	
	\plot 20.5 -47.5 20.5 -46.5 /
	\plot 21.5 -47.5 21.5 -46.5 /
	\plot 20.5 -45.5 21.5 -45.5 /

	\plot 24 -48 24 -44 /
	\plot 25 -48 25 -44 /
	\plot 26 -48 26 -44 /
	\plot 27 -48 27 -44 /
	\plot 28 -48 28 -44 /
	
	\plot 24 -48 28 -48 /
	\plot 24 -47 28 -47 /
	\plot 24 -46 28 -46 /
	\plot 24 -45 28 -45 /
	\plot 24 -44 28 -44 /
	
	\plot 26.5 -47.5 27.5 -47.5 /
	\plot 26.5 -46.5 27.5 -46.5 /
	\plot 26.5 -45.5 27.5 -45.5 /
	
	\plot 30 -48 30 -44 /
	\plot 31 -48 31 -44 /
	\plot 32 -48 32 -44 /
	\plot 33 -48 33 -44 /
	\plot 34 -48 34 -44 /
	
	\plot 30 -48 34 -48 /
	\plot 30 -47 34 -47 /
	\plot 30 -46 34 -46 /
	\plot 30 -45 34 -45 /
	\plot 30 -44 34 -44 /
	
	\plot 32.5 -47.5 33.5 -46.5 /
	\plot 32.5 -46.5 33.5 -47.5 /
	\plot 32.5 -45.5 33.5 -45.5 /

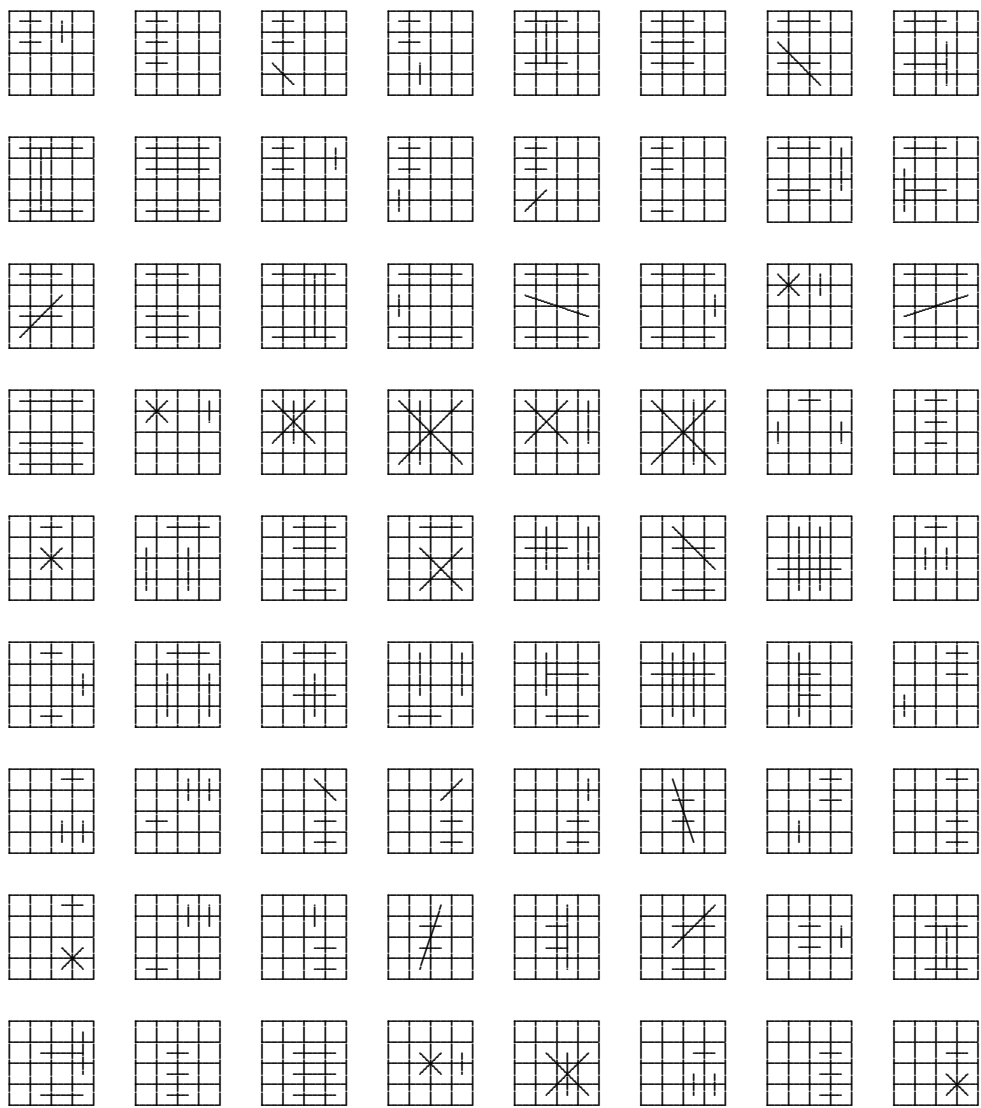
	\put {The stable permutations of rank 1 of cycle-type $(2,2,2)$ in $S([4]^2)$ up to transposition} at 12 -52
	
	\put{\bf{Fig.5}} at 9, -56
	
	\endpicture \]

\newpage
\noindent
{\bf Acknowledgments:}
FB acknowledges the MIUR Excellence Department Projects 
CUP E83C18000100006 and E83C23000330006 for financial support.
RC acknowledges financial support from Sapienza University of Rome (Progetti di Ateneo 2023) and GNAMPA-INDAM Project ``Operator Algebras and Infinite Quantum Systems'' CUP E53C23001670001.

\medskip
\noindent 
{\bf Data availability:} Data sharing is not applicable in this article as no new data was created or analyzed in this study.

\bigskip
{\parindent=0pt Addresses of the authors:\\
	
	\smallskip
	
	\noindent Francesco Brenti, Dipartimento di Matematica, Universit\`a di Roma 
	``Tor Vergata'', 
	Via della Ricerca Scientifica 1, I-00133 Roma, Italy.\\ E-mail: brenti@mat.uniroma2.it 
	
	\medskip
	
	\noindent Roberto Conti, Dipartimento di Scienze di Base e Applicate per l'Ingegneria, \\
	Sapienza Universit\`a di Roma, 
	Via A. Scarpa 16,
	I-00161 Roma, Italy.
	\\ E-mail: roberto.conti@sbai.uniroma1.it
	
	\medskip
	
	\noindent Gleb Nenashev, 
	Department of Mathematics and Computer Science, St. Petersburg State University, 14 line of the VO 29B, 199178 St. Petersburg, Russia,
	\\ E-mail: g.nenashev@spbu.ru   
	\par}

\end{document}